\documentclass[times]{amsart}

\usepackage[T1]{fontenc} 
\usepackage[utf8]{inputenc}
\usepackage[all,cmtip]{xy}
\usepackage{enumerate}
\usepackage{enumitem}
\usepackage[colorlinks=true,hyperindex,linktocpage=true]{hyperref}
\hypersetup{
	colorlinks,
	linkcolor={red!50!black},
	citecolor={red!50!black}, 
	urlcolor={red!50!black}, 
	filecolor={red!50!black} 
}
\usepackage[top=1in, left=1in, right=1in, bottom=1in]{geometry}
\usepackage{tikz-cd}

\usepackage{amsmath, amssymb, amsthm}
\usepackage{mathtools}
\usepackage{mathabx}
\usepackage{mathrsfs}
\usepackage{stmaryrd}
\usepackage{todonotes}
\usepackage[style=alphabetic,maxbibnames=9,maxcitenames=5,mincitenames=5]{biblatex}

\theoremstyle{plain}

\newtheorem{theorem}{Theorem}[subsection]
\newtheorem{proposition}[theorem]{Proposition}
\newtheorem{lemma}[theorem]{Lemma}
\newtheorem{corollary}[theorem]{Corollary}

\newtheorem{condition}[theorem]{Condition}

\theoremstyle{definition}
\newtheorem{definition}[theorem]{Definition}
\newtheorem{example}[theorem]{Example}
\newtheorem{construction}[theorem]{Construction}
\newtheorem{assumption}[theorem]{Assumption}
\theoremstyle{remark}
\newtheorem{remark}[theorem]{Remark}

\newcommand{\dbr}[1]{\left( \! \left( #1 \right) \!  \right)}

\newcommand{\dsq}[1]{\left\llbracket #1 \right\rrbracket}

\newcommand{\se}{\subseteq}

\newcommand{\ph}{\varphi}

\newcommand{\Ac}{\mathcal{A}}
\newcommand{\Bc}{\mathcal{B}}
\newcommand{\Cc}{\mathcal{C}}

\newcommand{\Ec}{\mathcal{E}}

\newcommand{\Gc}{\mathcal{G}}
\newcommand{\Hc}{\mathcal{H}}

\newcommand{\Lc}{\mathcal{L}}

\newcommand{\Oc}{\mathcal{O}}

\newcommand{\Tc}{\mathcal{T}}
\newcommand{\Uc}{\mathcal{U}}

\newcommand{\Xc}{\mathcal{X}}
\newcommand{\Yc}{\mathcal{Y}}
\newcommand{\Zc}{\mathcal{Z}}

\newcommand{\D}{\mathbb{D}}
\newcommand{\F}{\mathbb{F}}

\newcommand{\N}{\mathbb{N}}

\newcommand{\Q}{\mathbb{Q}}
\newcommand{\R}{\mathbb{R}}

\newcommand{\Z}{\mathbb{Z}}

\newcommand{\Fq}{{\F_q}}

\newcommand{\Bf}{\mathfrak{B}}

\newcommand{\af}{\mathfrak{a}}

\newcommand{\ff}{\mathfrak{f}}
\newcommand{\gf}{\mathfrak{g}}
\newcommand{\hf}{\mathfrak{h}}

\newcommand{\mf}{\mathfrak{m}}

\newcommand{\pf}{\mathfrak{p}}

\newcommand{\uf}{\mathfrak{u}}

\newcommand{\mr}[1]{\mathrm{#1}}
\DeclareMathOperator{\id}{id}

\DeclareMathOperator{\Hom}{Hom}
\DeclareMathOperator{\Aut}{Aut}

\DeclareMathOperator{\Spec}{Spec}

\DeclareMathOperator{\Spf}{Spf}

\DeclareMathOperator{\Gal}{Gal}
\DeclareMathOperator{\Frob}{Frob}

\newcommand{\FrobS}{\Frob_S}

\DeclareMathOperator{\GL}{GL}
\DeclareMathOperator{\SL}{SL}

\DeclareMathOperator{\Bun}{Bun}
\DeclareMathOperator{\Sht}{Sht}
\newcommand{\SSht}{\mathscr{S}\mr{ht}}
\DeclareMathOperator{\Hecke}{Hecke}

\newcommand{\Nilp}{\mathcal{N}ilp}

\DeclareMathOperator{\Gr}{Gr}

\begin{document}
	
	\title{Integral models of moduli spaces of shtukas with deep Bruhat-Tits~level~structures}
	
	\author{Patrick Bieker}
	\address{Bielefeld University, Faculty of Mathematics, Postfach 100131, 33501 Bielefeld, Germany}
	
	\email{pbieker@math.uni-bielefeld.de}
	
	\thanks{
		This work was supported by the Deutsche Forschungsgemeinschaft (DFG, German Research Foundation) via the CRC/TRR 326 \textit{Geometry and Arithmetic of Uniformized Structures}, project number 444845124, and the CRC/TRR 358 \textit{Integral Structures in Geometry and Representation Theory}, project number 491392403.
	}
	
	\begin{abstract}
		We construct integral models for moduli spaces of shtukas with deeper Bruhat-Tits level structures.
		We embed the moduli space of global shtukas for a deeper Bruhat-Tits group scheme into the limit of the moduli spaces of shtukas for all associated parahoric group schemes. Its schematic image defines an integral model of the moduli space of shtukas with deeper Bruhat-Tits level with favourable properties: They admit proper, surjective and generically \'etale level maps as well as a natural Newton stratification.
		In the Drinfeld case, this general construction of integral models recovers the moduli space of Drinfeld shtukas with Drinfeld level structures.
	\end{abstract}
	
	\maketitle

	\section{Introduction}
	
	Following Langlands' philosophy, the Hasse-Weil $\zeta$-function of a Shimura variety should be computable in terms of automorphic $L$-functions. 
	The Langlands-Kottwitz method for computing the Hasse-Weil $\zeta$-function (or more generally to study the cohomology of Shimura varieties) relies on two geometric inputs. 
	On the one hand one needs a description of the mod-$p$ points of the Shimura variety as conjectured by Langlands and Rapoport \cite{Langlands1987} and on the other hand one needs to identify a certain test function that serves as an input to the Arthur-Selberg trace formula.
	For both steps the existence of good integral models of the Shimura variety is necessary.
	For Shimura varieties with parahoric level structure there has been substantial progress towards both steps. The Langlands-Rapoport conjecture is proved under mild hypotheses in \cite{Hoften2022} (and \cite{Rad2016} in the function field setting), while the test function conjecture is shown in the series of papers \cite{Haines2021, Haines2020a, Anschuetz2022}.
	The goal of this work is to provide a first step towards generalising the results in the function field setting to more general (Bruhat-Tits-)level structures. Namely, we construct analogues of the integral models of Shimura varieties of \cite{Kisin2018} for these deeper level structures in the function field setting.
	
	The analogues of Shimura varieties over function fields are moduli spaces of (global) shtukas.
	They were first introduced in \cite{Drinfeld1987a} for $\GL_n$ and later generalised to arbitrary split reductive groups in \cite{Varshavsky2004} and further to flat affine group schemes of finite type in \cite{Rad2019a}. They are used to great success in establishing a Langlands correspondence over function fields in \cite{Drinfeld1987} for $\GL_2$, \cite{Lafforgue2002} for $\GL_n$ and \cite{Lafforgue2018} for arbitrary reductive groups. 
	Recently, a lot of progress has been made in studying the geometry of moduli spaces of shtukas with parahoric level, compare for example \cite{Rad2015}, \cite{Rad2017}, \cite{Breutmann2019}, \cite{Yun2019} and \cite{Zhu2014}. However, little is recorded for deeper level structures. 	
	A first result beyond the parahoric case is obtained in \cite{Bieker2023}, where Drinfeld $\Gamma_0(\pf^n)$-level structures for Drinfeld shtukas are defined (in the style of \cite{Drinfeld1976, Katz1985}) and it is shown that their moduli spaces admit well-behaved (that is finite flat and generically \'etale) level maps as in the case of modular curves in \cite{Katz1985}.
	The goal of this work is to construct integral models of moduli spaces of shtukas with deep Bruhat-Tits level structures for general reductive groups that generalise both the parahoric case and in the Drinfeld case the moduli space of shtukas with Drinfeld $\Gamma_0(\pf^n)$-level structures of \cite{Bieker2023}, and to study properties of these integral models.
	
	Let $X$ be a smooth, projective and geometrically connected curve over a finite field $\Fq$. 
	Let $G$ be a (connected) reductive group over the function field $F$ of $X$ and let us fix a parahoric model $\Gc \to X$ of $G$. In other words, $\Gc$ is a smooth affine group scheme with connected fibres over $X$ and generic fibre $G$ such that for all closed points $x$ of $X$ the pullback $\Gc_{\Spec(\Oc_x)}$ to the spectrum of the completed local ring $\Oc_x$ at $x$ is a parahoric group scheme in the sense of \cite{Bruhat1984}.
	
	Let us fix a closed point $x_0$ of $X$. 
	Let $\Omega$ be a bounded subset of an apartment in the Bruhat-Tits building of $G_{F_{x_0}}$, where $F_{x_0}$ is the completion of $F$ at $x_0$. 
	By Bruhat-Tits theory, we get a smooth affine $\Oc_{x_0}$-group scheme $\Gc_{\Omega}$ with connected fibres that we glue with $\Gc$ outside of $x_0$ to obtain a (global) Bruhat-Tits group scheme $\Gc_{\Omega} \to X$ which is smooth and affine with connected fibres by construction. 
	Without changing $\Gc_\Omega$, we may assume that $\Omega$ is convex, closed and a union of facets.
	
	Let $I$ be a finite set and let $\underline{\mu} = (\mu_i)_{i \in I}$ be a tuple of conjugacy classes of geometric cocharacters of $G$. For simplicity, we assume in this introduction that $\underline{\mu}$ is defined over the function field $F$ of $X$ (in general it will only be defined over a finite separable extension of $F$). 
	A global $\Gc_\Omega$-shtuka over a scheme $S$ is a $\Gc_\Omega$-bundle $\Ec$ on $X_S$ together with an isomorphism $\ph \colon  \Ec|_{X_S \setminus \Gamma_{\underline{x}}} \xrightarrow{\cong} \sigma^\ast\Ec|_{X_S \setminus \Gamma_{\underline{x}}}$ away from the graph $\Gamma_{\underline{x}}$ of an $I$-tuple $\underline{x} \in X^I(S)$ of points of $X$. 
	We denote by $\Sht^{I, \leq \underline{\mu}}_{\Gc_\Omega, X}$ the moduli space of global $\Gc_\Omega$-shtukas whose zeroes and poles are bounded by $\underline{\mu}$. It is a Deligne-Mumford stack locally of finite type over $\Fq$. 
	
	While for a subset $\Omega'$ of $\Omega$ there is still a natural map $\Sht^{I, \leq \underline{\mu}}_{\Gc_\Omega, X} \to \Sht^{I, \leq \underline{\mu}}_{\Gc_{\Omega'}, X}$ by \cite[Theorem 3.20]{Breutmann2019} (and Theorem \ref{thmLvlMapGen}), already in the Drinfeld case $G = \GL_2$, the level map  $\Sht_{\GL_{2, [0,n]}, X}^{(1,2),\leq ((0,-1), (1,0))} \to \Sht_{\GL_2, X}^{(1,2), \leq ((0,-1), (1,0))}$ is neither proper nor surjective for $n \geq 2$ by \cite[Remark 2.20]{Bieker2023}.
	In order to remedy this issue, we propose the following construction to relatively compactify  $\Sht^{I, \leq \underline{\mu}}_{\Gc_{\Omega}, X}$.
	\begin{definition}[compare Definition \ref{defnIntMod}]
		\label{defnIntroIntMod}
		In the situation above, that is, for a reductive group $G$ over $F$, and a Bruhat-Tits group scheme $\Gc_\Omega \to X$ for a subset $\Omega$ (assumed to be convex, closed and a union of facets) of the Bruhat-Tits building for $G_{F_{x_0}}$ at the fixed point $x_0$ of $X$ as above, 
		the \emph{integral model of the moduli space of shtukas with $\Gc_\Omega$-level structure} $\SSht_{\Gc_\Omega, X}^{I, \leq \underline{\mu}}$ is defined to be the schematic image in the sense of \cite{Emerton2021} of the map
		$$\Sht_{\Gc_\Omega, X}^{I, \leq \underline{\mu}} \to \varprojlim_{\ff \prec \Omega} \Sht_{\Gc_\ff, X}^{I, \leq \underline{\mu}},$$
		where the limit is taken over the set of all facets $\ff$ contained in $\Omega$ partially ordered with respect to inclusion.
	\end{definition}
	
	Clearly, in the parahoric case (that is, when $\Omega$ is a facet) we have 
	$$ \SSht_{\Gc_\Omega, X^I}^{\leq \underline{\mu}} = \varprojlim_{\ff \prec \Omega} \Sht_{\Gc_\ff, X^I}^{\leq \underline{\mu}} = \Sht_{\Gc_\Omega, X^I}^{\leq \underline{\mu}},$$ 
	as the index set has $\Omega$ as its final object. So our construction generalises the parahoric case. 
	Moreover, by \cite[Theorem 6.7]{Bieker2023}, our notion of integral models recovers in the Drinfeld case the moduli space of Drinfeld shtukas with Drinfeld $\Gamma_0(\pf^n)$-level structure at $x_0$.
	
	The main result of this work is to show that our integral models have good formal properties. Namely, that they admit proper, surjective and generically finite \'etale level maps (similar to the integral models of Shimura varieties with parahoric level structures):
	\begin{theorem}[compare Theorem \ref{thmImmersion} and Theorem \ref{thmLvlMapGeneral}]
		\label{thmIntroSht}
		In the situation of Definition \ref{defnIntroIntMod}, the map
		$$\Sht_{\Gc_\Omega, X}^{I, \leq \underline{\mu}} \to \varprojlim_{\ff \prec \Omega} \Sht_{\Gc_\ff, X}^{I, \leq \underline{\mu}} $$
		is schematic and a quasi-compact locally closed immersion. It factors into an open immersion $\Sht_{\Gc_\Omega, X}^{I, \leq \underline{\mu}} \to \SSht_{\Gc_\Omega, X}^{I, \leq \underline{\mu}}$ followed by the closed immersion $\SSht_{\Gc_\Omega, X}^{I, \leq \underline{\mu}} \to \varprojlim_{\ff \prec \Omega} \Sht_{\Gc_\ff, X}^{I, \leq \underline{\mu}}$.
		The restriction of the inclusion
		$$ \Sht_{\Gc_\Omega, X}^{I, \leq \underline{\mu}}|_{(X \setminus \{x_0\})^I} \xrightarrow{\cong} \SSht_{\Gc_\Omega, X}^{I, \leq \underline{\mu}}|_{(X \setminus \{x_0\})^I}$$
		away from $x_0$ is an open and closed immersion. 
		Moreover, for a subset $\Omega' \prec \Omega$, there is a natural level map
		$$\bar{\rho}_{\Omega', \Omega} \colon \SSht_{\Gc_\Omega, X}^{I, \leq \underline{\mu}} \to  \SSht_{\Gc_{\Omega'}, X}^{I, \leq \underline{\mu}} $$
		that is schematic, proper and surjective, and over $(X \setminus \{x_0\})^I$ it is finite \'etale.
	\end{theorem}	
	
	\begin{remark}
		In principle, we expect our methods to apply also to even more general level structures, namely to stabilisers of bounded subsets $\Omega$ of the Bruhat-Tits building that are not necessarily contained in a single apartment. 
		This includes in particular the case that $\Omega$ is a ball of radius $r$ around a point $x$ in the Bruhat-Tits building, whose (connected) stabiliser should be closely connected (up to the centre) to the Moy-Prasad subgroup $G(F_{x_0})_{x,r}$.
		
		However, the results from Bruhat-Tits theory we rely on in the proof of Theorem \ref{thmIntroSht} do not seem to be available in this case. 
		Compare Section \ref{sectOmegaGeneral} for further discussion of this situation and the statements we can prove in the more general setting.
	\end{remark}
	
	The construction of integral models is a first step towards generalising results used in the Langlands-Kottwitz method at parahoric level to the more general level structures considered here.
	In particular, a fine analysis of the (\'etale-local) structure of their special fibres together with the calculation of the trace of Frobenius on the nearby cycles might yield significant progress towards a geometric construction of test functions for deeper Bruhat-Tits level structures generalising the results in the parahoric case of \cite{Haines2021, Haines2020a, Anschuetz2022} (and \cite{Haines2012} in the Drinfeld case with $\Gamma_1(p)$-level structures). 
	Note that the test functions for $\Gamma(p^n)$-level (which are not Bruhat-Tits level structures in our sense) constructed in \cite{Scholze2013b} have no apparent connection to geometry.
	A geometric construction of test functions in our setting might shed some more light onto this situation as well. 
	
	Let us briefly comment on the proof of Theorem \ref{thmIntroSht}.
	In the parahoric case, level maps on moduli spaces of shtukas are also studied in \cite[Theorem 3.20]{Breutmann2019}.
	However, the assumption on the bounds used there does not apply to the bounds given by cocharacters. 
	We provide the relevant background on bound for global (and local) shtukas and generalise the result of \cite[Theorem 3.20]{Breutmann2019} to include bounds in this sense (compare Theorem \ref{thmLvlMapGen}).
	Using the assertion in the parahoric case, we are then able to deduce the result also for level maps at  deeper level.
	
	The first part of Theorem \ref{thmIntroSht} is based on a corresponding result on the moduli space of $\Gc_\Omega$-bundles.
	\begin{theorem}[compare Theorem \ref{thmBunGImm}]
		\label{thmIntroBunBT}
		In the situation of Definition \ref{defnIntroIntMod}, the natural map
		$$ \Bun_{\Gc_\Omega} \to \varprojlim_{\ff \prec \Omega} \Bun_{\Gc_\ff}$$
		is schematic and representable by a quasi-compact open immersion.
	\end{theorem}
	As a first step in the proof of this theorem, we show in the local case (and hence also for the corresponding global Bruhat-Tits group schemes), 
	that the not necessarily parahoric Bruhat-Tits group scheme $\Gc_\Omega$ is the limit of all its associated parahoric group schemes	$ \Gc_\Omega \xrightarrow{\cong} \varprojlim_{\ff \prec \Omega} \Gc_\ff,$
	compare Theorem \ref{thmBTGS}.
	
	Note that given a compatible system of $\Gc_\ff$-torsors for all facets $\ff \prec \Omega$, it is in general not true that their limit is a torsor for $\Gc_\Omega$, as it might be impossible to construct a compatible system of sections.
	By controlling the deformation theory of torsors for the $\Gc_\ff$, we are able to show that the locus where the limit of a compatible system of $\Gc_\ff$-bundles on $X$ is already a $\Gc_\Omega$-bundle on $X$ is open. 
	
	As a first step in studying the structure of the special fibre of $\SSht_{\Gc_\Omega, X}^{I, \leq \underline{\mu}}$, we show that the Newton stratification on the special fibre of the moduli space of shtukas in the parahoric case induces a well-defined Newton stratification on the special fibre in the case of deeper level. 
	For a reductive group $H$ over a local field $E$ we denote by $B(H)$ the set of $\sigma$-conjugacy classes in $H(\breve{E})$, where $\breve{E}$ is the completion of the maximal unramified extension of $k$. 
	Then $B(H)$ classifies quasi-isogeny classes of local shtukas for (an integral model of) $H$. 
	
	We fix a tuple of pairwise distinct closed points $\underline{y} = (y_i)_{i \in I}$ in $X$ and denote by $\SSht_{\Gc_\Omega, X, \F_{\underline{y}}}^{I, \leq \underline{\mu}} = \SSht_{\Gc_\Omega, X}^{I, \leq \underline{\mu}} \times_{X} \F_{\underline{y}}$ the special fibre over $\underline{y}$, where $\F_{\underline{y}}$ is the compositum of the residue fields of the points $y_i$ of $X$. 
	
	Combining our compactification with the results of \cite[Theorem 7.11]{Hartl2011} and \cite[Section 5]{Breutmann2019} in the parahoric case, we get the following result on the Newton stratification for deep level.
	\begin{theorem}[compare Definition \ref{defnNewtonDeep} and Corollary \ref{corNewtonLevel}]
		Let $\ell$ be an algebraically closed extension of $\F_{\underline{y}}$. 
		There is a well-defined map
		$$ \bar\delta_{\Gc_\Omega} \colon  \SSht_{\Gc_\Omega, X, \F_{\underline{y}}}^{I, \leq \underline{\mu}}(\ell) \to \prod_{i \in I} B(G_{F_{y_i}}) $$
		that is compatible with the level maps in the sense that for $\Omega' \prec \Omega$ we have 
		$$ \bar\delta_{\Gc_\Omega} =  \bar\delta_{\Gc_{\Omega'}} \circ \bar \rho_{\Omega', \Omega}.$$
		Moreover, for $\underline{b} = (b_i)_{i \in I} \in B(G_{F_{y_i}})$ the preimage of $\underline{b}$ under $\bar\delta_{\Gc_\Omega}$ is the set of $\ell$-valued points of a locally closed substack 
		$\Sht_{\Gc_\Omega, X, \F_{\underline{y}}}^{I, \leq \underline{\mu}, \underline{b}}$ of $\Sht_{\Gc_\Omega, X, \F_{\underline{y}}}^{I, \leq \underline{\mu}}$ called the \emph{Newton stratum} of $\Sht_{\Gc_\Omega, X, \F_{\underline{y}}}^{I, \leq \underline{\mu}}$ for $\underline{b}$. 
	\end{theorem} 
	In the parahoric case, the map $\bar{\delta}$ is given by associating to a point in the special fibre over $\underline{y}$ the quasi-isogeny classes of its local shtukas at the points $y_i$.
	We use the compatibility of the Newton stratification with the level maps in the parahoric case to extend this result to the case of deep level.
	
	Moreover, we show that in the hyperspecial case the Newton stratification satisfies the strong stratification property (as for Shimura varieties).
	Recall that there is a natural order on $B(H)$ induced by the dominance order on cocharacters.
	It is well-known in the parahoric case that the closure of a Newton stratum
	$$\overline{\Sht^{I, \leq \underline{\mu}, \underline{b}}_{\Gc, X, \F_{\underline{y}}}}  \se \bigcup_{\underline{b'} \leq \underline{b}} \Sht^{I, \leq \underline{\mu},  \underline{b'}}_{\Gc, X, \F_{\underline{y}}}$$  
	is contained in a union of Newton strata.	
	Note that this also generalises to deeper level.
	We say that the Newton stratification satisfies the strong stratification property when we even have equality. 
	However, the inclusion is strict in general. 
	For local shtukas for split reductive groups, the strong stratification property is due to  \cite{Viehmann2011}. We use the Serre-Tate theorem for shtukas of \cite{Rad2015} to deduce the corresponding result in the global case, compare Theorem \ref{thmNewtonHyp}.
	
	\subsection*{Organisation}
	This paper is organised as follows. 	
	Sections \ref{secBounds} and \ref{secSht} provide background on the general theory of moduli spaces of shtukas.
	Our construction of integral models for deeper level relies on functoriality results for moduli spaces of shtukas.
	As the functoriality result of \cite[Theorem 3.20]{Breutmann2019} is not applicable in our setting, 
	we first carefully set up (and slightly generalise) the theory of bounds for (iterated, global and local shtukas) in Section \ref{secBounds}. We show basic properties including local-global compatibility.
	In Section \ref{secSht}, we generalise various well-known results for moduli spaces of (bounded iterated) shtukas to our setting. 
	In particular, we show the relevant functoriality of moduli spaces of shtukas in Theorem  \ref{thmLvlMapGen} which also proves the analogue of Theorem \ref{thmIntroSht} at parahoric level.	
	In Section \ref{sectBT}, we study (torsors under) Bruhat-Tits group schemes as well as their deformation theory and show Theorem \ref{thmIntroBunBT}.
	In Section \ref{sectIntMod}, we give the construction of integral models at deeper Bruhat-Tits level and deduce our main results. We construct a Newton stratification on the integral models with deeper level.

	\subsection*{Notation}
	We fix the following notation.
	For a Dedekind scheme $X$ over a field $k$ (usually $X$ will be a smooth (even projective) curve over $k$ or the spectrum of $k \dsq{\varpi}$ and $k$ will often assumed to be finite) and a closed point $x$ of $X$ we denote by $\Oc_{X,x}$ the local ring at $x$ and by $\Oc_x$ its completion, by $\mf_{x} \se \Oc_{x}$ the maximal ideal with uniformiser $\varpi_{x}$ and by $\F_{x}$ the residue field. Moreover, we denote by $F_x$ the completion of the function field $F = F(X)$ of $X$ at $x$.
	When $k = \Fq$ is a finite field, we denote by $\sigma$ the (absolute) $q$-Frobenius endomorphism $\FrobS$ of some $\Fq$-scheme $S$, and also the map $\sigma = id_X \times \FrobS \colon X_S \to X_S$.

\subsection*{Acknowledgements}
I thank my advisor Timo Richarz for introducing me to this topic, his steady encouragement and his interest in my work. 
I thank the anonymous referee for their careful reading of the paper and their valuable suggestions to clarify the exposition.
I thank Urs Hartl, Alexander Kutzim and Qihang Li for their comments that lead to a significant improvement of the present work.
I thank Gebhard B\"ockle, Paul Hamacher, Jo\~{a}o Louren\c{c}o, Eva Viehmann and Torsten Wedhorn for helpful conversations surrounding this work. 
I thank Catrin Mair and Thibaud van den Hove for their comments on a preliminary version of this paper. 
I thank Urs Hartl for sharing the revised version of \cite{Breutmann2019} and Tasho Kaletha for sharing a preliminary version of \cite{Kaletha}.

	\section{Bounds for (global and local) shtukas}
	\label{secBounds}
	
	Bounds for (global or local) shtukas usually are defined as certain (equivalence classes of) subschemes of affine Beilinson-Drinfeld Grassmannians, compare for example \cite{Varshavsky2004, Lafforgue2018, Rad2015, Rad2017}.
	In this section we generalise the existing notion of bounds and provide a uniform treatment of global and local bounds.
	An important class of bounds are generically defined bounds, which arise as the closure of their generic fibre.
	In this case, the theory simplifies. For example, we show that a bound always admits a representative over its reflex field. 
	Moreover, generically defined bounds are the natural setting for our construction of integral models with deeper level structure in the following.
	
	This includes in particular the bounds given by Schubert varieties inside the affine Grassmannian for parahoric group schemes that are often considered in applications, compare for example \cite{Lafforgue2018, Yun2019, Feng2024}. 
	We discuss the relation of our notion of bounds to various other notions of bounds in the literature.
	
	The discussion of bounds is used to formulate and prove analogues of the functoriality results for moduli spaces of shtukas in the next section.
	In particular, generically defined bounds provide the natural setting to formulate the functoriality of moduli spaces of shtukas under generic isomorphisms of group schemes in Theorem \ref{thmLvlMapGen}.
	
	\subsection{Beilinson-Drinfeld affine Grassmannians}
	\label{subsecBDGrass}
	We recall the definition of Beilinson-Drinfeld affine Grassmannians. 
	We work in the following setting. Let $k$ be a field and let $X$ be either a smooth geometrically connected curve over $k$ or $X = \Spec(k \dsq{\varpi})$.
	Let $I$ be a finite set and $I_\bullet = (I_1, \ldots, I_m)$ a partition of $I$.
	Let $\Gc \to X$ be a flat affine $X$-group scheme of finite type.
	
	We need the following iterated version of Beilinson-Drinfeld affine Grassmannians first introduced by \cite{Beilinson1996} in the case of constant split reductive group schemes.
	\begin{definition}
		We denote by $\Gr_{\Gc, X}^{I_\bullet}$ the functor on $k$-algebras whose $R$-valued points are given by tuples 
		$ ((x_i)_{i \in I}, (\Ec_j)_{j = 0, \ldots, m}, (\ph_j)_{j = 1, \ldots, m}, \varepsilon),$
		where
		\begin{itemize}
			\item $x_i \in X(R)$ are points on $X$,
			\item $\Ec_j$ are $\Gc$-bundles on $X_R$ for $0 \leq j \leq m$,
			\item $\ph_j \colon \Ec_{j}|_{X_R \setminus \bigcup_{i \in I_j} \Gamma_{x_i}} \xrightarrow{ \cong } \Ec_{j-1}|_{X_R \setminus \bigcup_{i \in I_j} \Gamma_{x_i}}$ are isomorphisms of $\Gc$-bundles, and
			\item $\varepsilon \colon \Ec_0 \xrightarrow{\cong} \Gc \times_X X_R$ is a trivialisation of $\Ec_0$.
		\end{itemize}
	\end{definition}
	Then $\Gr_{\Gc,X}^{I_\bullet}$ is representable by a separated (strict) ind-scheme of ind-finite type over $X^I$ by \cite{Heinloth2010, Richarz2014} in the global setting generalising the previous works \cite{Beilinson1996, Gaitsgory2001}, and by \cite{Richarz2016, Richarz2021a} in the local setting. 
	
	As a shorthand notation we denote a point $((x_i)_{i \in I}, (\Ec_j)_{j = 0, \ldots, m}, (\ph_j)_{j = 1, \ldots, m}, \varepsilon) \in \Gr_{\Gc, X}^{I_\bullet}(R)$ by
	$$   \left( \Ec_m \overset{\ph_m}{\underset{\Gamma_{\underline{x}_m}}{\dashrightarrow}} \Ec_{m-1} \overset{\ph_{m-1}}{\underset{\Gamma_{\underline{x}_{m-1}}}{\dashrightarrow}} \ldots \overset{\ph_{1}}{\underset{\Gamma_{\underline{x}_1}}{\dashrightarrow}} \Ec_0 \overset{\varepsilon}{\rightarrow}  \Gc  \right),$$
	where for a point $\underline{x} = (x_i)_{i \in I} \in X^I(R)$ we denote by $\underline{x}_j =(x_i)_{i \in I_j}$ the projection to the $I_j$-components and by $\Gamma_{\underline{x}_j} = \bigcup_{i \in I_j} \Gamma_{x_i} $ the union of the graphs $\Gamma_{x_i}$ of $x_i \colon \Spec(R) \to X$.
	When $I$ is a singleton set, we denote the corresponding Beilinson-Drinfeld affine Grassmannian by $\Gr_{\Gc, X}$.
	In the local setting $X = \Spec(\Oc)$ we also write $\Gr_{\Gc, \Oc} = \Gr_{\Gc, X}$.
	Compare \cite{Richarz2021a} for an alternative description of $\Gr_{\Gc, \Oc}$. Note that when $\Oc = \Oc_{x_0}$ is the complete local ring of a point $x_0 \in X$ of some smooth curve $X$, we have $\Gr_{\Gc_{\Oc}, \Oc} = \Gr_{\Gc, X} \times_X \Spec(\Oc)$, where $\Gc_\Oc = \Gc \times_X \Spec(\Oc)$ is the pullback of $\Gc$ to $\Spec\Oc$ by \cite{Richarz2021a}.
	
	Let $R$ be a $k$-algebra. For a relative effective Cartier divisor $D \se X_{R}$, the formal completion of $X_R$ along $D$ is a formal affine scheme. We denote by $\hat \Oc_D$ the underlying $R$-algebra and by $\hat{D} = \Spec(\hat \Oc_D)$ the corresponding affine scheme. 
	Then $D$ is a closed subscheme of $\hat D$ and we set $\hat{D}^0 = \hat{D} \setminus D$. 
	Moreover, we denote by $D^{(r)} = \Spec(\hat{\Oc}_D/\mf^r)$ the $r$-the formal neighbourhood of $D$, where $\mf \se \hat{\Oc}_D$ is the ideal defining $D$. 
	In particular, when $D$ is the graph $\Gamma_x$ of a section $x \in X(R)$ (or a union thereof), we use the corresponding notation $\hat{\Gamma}_{x}$, $\hat{\Gamma}_x^\circ$ and $\Gamma_x^{(r)}$. 
	
	\begin{remark}
		Using Beauville-Laszlo descent, compare \cite{Beauville1995}, a point of $\Gr_{\Gc, X}^{I_\bullet}$ is given by a tuple 
		$$((x_i)_{i \in I}, (\Ec_j)_{j = 0, \ldots, m}, (\ph_j)_{j = 1, \ldots, m}, \varepsilon),$$
		where the $\Ec_j$ are now $\Gc$-torsors on $\hat{\Gamma}_{\underline{x}}$, the $\ph_j \colon \Ec_j \to \Ec_{j-1}$ are isomorphisms over $\hat{\Gamma}_{\underline{x}} \setminus \Gamma_{\underline{x}_j}$, and $\varepsilon$ is a trivialisation of $\Ec_0$ defined over $\hat{\Gamma}_{\underline{x}}$.
		This also shows that the natural map $\Gr_{\Gc,X}^{I_\bullet}|_U  \xrightarrow{\cong} \left(\prod_{i \in I} \Gr_{\Gc, X}\right)|_U$ is an isomorphism, where $U \se X^I$ is the complement of all diagonals.
		\label{remBLDesc}
	\end{remark}

	We use the following global version of the (positive) loop group. 
	
	\begin{definition}
		The \emph{global loop group} $\Lc_{X^I} \Gc$ is the functor on the category of $k$-algebras 
		$$ R \mapsto \left\{  (\underline x, g)\colon \underline x \in X^I(R), g \in \Gc(\hat{\Gamma}_{\underline{x}}^0) \right\}.$$
		The \emph{positive global loop group} $\Lc_{X^I}^+ \Gc$ is the functor on the category of $k$-algebras
		$$  R \mapsto \left\{  (\underline x, g)\colon \underline x \in X^I(R), g \in \Gc(\hat{\Gamma}_{\underline{x}}) \right\}.$$
		Moreover, for $r \in \N$ let $\Lc_{X^I}^{(r)} \Gc$ be the functor on $k$-algebras
		$$  R \mapsto \left\{  (\underline x, g)\colon \underline x \in X^I(R), g \in \Gc(\Gamma_{\underline{x}}^{(r)}) \right\}.$$
	\end{definition}
	By \cite[Proposition 2]{Heinloth2010}, $\Lc_{X^I} \Gc$ is representable by an ind-affine group ind-scheme over $X^I$ and $\Lc_{X^I}^+\Gc$ is representable by a closed affine group subgroup scheme of $\Lc_{X^I} \Gc$ over $X^I$. 
	Moreover by \cite[Lemma 2.11]{Richarz2016}, all  $\Lc^{(r)}_{X^I} \Gc \to X^I$ are representable by flat affine group schemes of finite type and the transition maps for varying $r$ are affine.
	The canonical map given by the projections
	\begin{equation}	
		\Lc^+_{X^I} \Gc \xrightarrow{\cong} \varprojlim_r \Lc^{(r)}_{X^I} \Gc
		\label{eqnPosLoopGroup}
	\end{equation} 
	is an isomorphism. 
	Assume moreover that $\Gc$ is smooth. Then all $\Lc^{(r)}_{X^I} \Gc$ are smooth and by (\ref{eqnPosLoopGroup}) we have that $\Lc^+_{X^I} \Gc \to X^I$ is flat.

	There is a natural $\Lc^+_{X^I}\Gc$-action on $\Gr_{\Gc, X}^{I_\bullet}$ by changing the trivialisation $\varepsilon$.

	\subsection{Definition of bounds}
	
	We define bounds inside the iterated affine Grassmannian by adapting the definition in the non-iterated setting of \cite[Definition 3.1.3]{Rad2017}.
	We work in the setup of \cite[Section 2.6]{Hartl2023} and allow different types of bounds (global/local/finite) at different legs. More precisely we define bounds in the following cases.
	
	\begin{definition}
		\label{defnGlobalBound}
		
		\begin{enumerate}
			\item 
			For each $i \in I$ a ground field $F^{(i)}$ is fixed, which can be either
			\begin{enumerate}
				\item \emph{of global type} $F^{(i)} = F$ (only in the \emph{global setting} when $X$ is a smooth geometrically connected curve over $k$),
				\label{bdTypeGlobal}
				\item \emph{of local type} $F^{(i)} = F_y$, the completion of $F$ at a closed point $y$ of $X$,
				\label{bdTypeLocal}
				\item or \emph{of finite type} $F^{(i)} = \F_y$, the residue field at $y$. 
				\label{bdTypeFinite}
			\end{enumerate}
			\item
			All finite extension of $F^{(i)}$ are without further mention assumed to be separable and contained in a fixed separable closure $\widebar{F}^{(i)}$ of $F^{(i)}$. 
			We set $\widetilde{X}_{F^{(i)}}$ to be either $X$, $\Spec(\Oc_y)$ or $\Spec(\F_y)$ according to the type of $F^{(i)}$. 
			Moreover, let $\underline{F} = (F^{(i)})_{i \in I}$ and $\widetilde{X}_{\underline F} = \prod_{i \in I} \widetilde{X}_{F^{(i)}}$, where the product is taken over $k$.
			For a finite extension $(F^{(i)})'$ of $F^{(i)}$ we set $\widetilde{X}_{(F^{(i)})'}$ to be the normalisation of $\widetilde{X}_{F^{(i)}}$ in $(F^{(i)})'$. It is either a smooth curve in the global type, the (spectrum of the) valuation ring of $(F^{(i)})'$ in the local type or $\Spec(F^{(i)})'$ in the finite type.
			In each case the natural map $\widetilde{X}_{(F^{(i)})'} \to \widetilde{X}_{F^{(i)}}$ is finite and faithfully flat.
			For an $I$-tuple $\underline F' = ((F^{(i)})')_{i \in I}$ of componentwise finite extensions of $\underline F$ we set $\widetilde{X}^I_{\underline F'} = \prod_{i \in I} \widetilde{X}_{(F^{(i)})'}$.
			\item
			Let $\underline F_1$ and $\underline F_2$ be two $I$-tuples of finite extensions of $\underline F$. 
			Two quasi-compact closed subschemes $Z_1\se \Gr_{\Gc, X}^{I_\bullet}\times_{X^I} \widetilde X_{\underline F_1}^I$ and $Z_2\se \Gr_{\Gc, X}^{I_\bullet}\times_{X^I} \widetilde X_{\underline F_2}^I$ are called \emph{equivalent} if there is an $I$-tuple of finite extensions $\underline{F}'$ of $F$ such that componentwise $\underline{F}'$ is a finite extension of both $\underline{F}_1$ and $\underline{F}_2$ with $Z_1 \times_{\widetilde X_{\underline F_1}^I} \widetilde X_{\underline F'}^I=Z_2 \times_{\widetilde X_{\underline F_2}^I} \widetilde X_{\underline F'}^I$ in $\Gr_{\Gc, X}^{I_\bullet}\times_{X^I} \widetilde X_{\underline F'}^I$. 
			
			\item 
			\label{reflexscheme}
			Let $\Zc$ be an equivalence class of quasi-compact closed subschemes $Z_{\underline F'} \se \Gr_{\Gc, X}^{I_\bullet}\times_{X^I} \widetilde X_{\underline F'}^I$ and let us fix a representative $Z_{\underline{F}'}$ defined over some $\underline{F}'$ such that all components $(F^{(i)})'$ are Galois extensions of $F^{(i)}$. 
			We set 
			$ \Aut_\Zc(\underline{F}') =  \{(g_i)_{i \in I} \in \Gal((F^{(i)})'/F^{(i)}): (g_i)_{i \in I}^\ast(\Zc)=\Zc\}$. The \emph{reflex scheme} $\widetilde{X}^I_\Zc$ of $\Zc$ is defined to be the quotient 
			$$ \widetilde{X}^I_\Zc = \widetilde{X}^I_{\underline{F}'}/\Aut_\Zc(\underline{F}').$$
			
			\item
			By Lemma \ref{lemBdSchImage} the reflex scheme is independent of the choice of representative.
			We say that $\Zc$ \emph{admits a representative over $\widetilde{X}^I_\Zc$} if  $Z_{\underline F'}$ (or by Lemma \ref{lemBdSchImage} any other representative of $\Zc$) descends to $\widetilde{X}^I_\Zc$. 
			
			\item
			A \emph{bound} is an equivalence class $\Zc$ of quasi-compact closed subschemes $Z_{\underline F'}\subset \Gr_{\Gc, X}^{I_\bullet}\times_{X^I} \widetilde X_{\underline F'}^I$, such that all its representatives $Z_{\underline F'}$ are stable under the left $\Lc_{X^I}^+\Gc \times_{X^I} \widetilde X_{\underline F'}^I$-action on $\Gr_{\Gc, X}^{I_\bullet}\times_{X^I} \widetilde X_{\underline F'}^I$. 
		\end{enumerate}
	\end{definition}
	
	\begin{remark}
		The definition of the reflex scheme is due to \cite[Definition 2.6.3]{Hartl2023}. We thank U.Hartl for suggesting to use it in the present work.
		\begin{enumerate}
			\item 
			Later only global bounds (i.e. when $X$ is a curve) of global type and local bounds (i.e. when $X = \Spec(k \dsq{\varpi}))$) of local type will be relevant and the reader who is only interested in these situations can safely ignore the extra generality we allow in the definition. 
			However, restricting to this situation does not significantly simplify the arguments in this section.
			Bounds (in the global setting) with legs of different types are for example used in \cite{Hartl2023}.
			\item 
			In general, the reflex scheme can be complicated, but in all cases relevant later it will be a product $\prod_{i \in I} \tilde{X}_{F_i'}$ for finite separable extensions $F_i'/F$.
			In particular, when $I = \{\ast\}$ contains only one element $\widetilde X_\Zc = \widetilde{X}_{F_\Zc}$ where $F_\Zc$ is the field of definition of the bound $\Zc$ in the sense of \cite[Definition 3.1.3]{Rad2017}.
			In many cases the reflex scheme of a bound can be determined as in \cite[Remark 3.1.4]{Rad2017} (and \cite[Remark 4.7]{Rad2015} in the local setting).  
			It is however not clear if there always exists a representative over the reflex scheme in general. 
			\item 
			By construction, the reflex scheme comes equipped with maps $\widetilde{X}^I_{\underline{F}'} \to \widetilde{X}^I_\Zc \to \widetilde X_{\underline F}^I$ for any $I$-tuple $\underline{F}'$ of finite extensions of $\underline F$ such that $\Zc$ admits a representative over $\underline{F}'$.
		\end{enumerate}
	\end{remark}

	The following lemma is an analogue of \cite[Remark 4.6]{Rad2015}.
	\begin{lemma}
		Let $\Zc$ be an equivalence class of closed subschemes of $\Gr_{\Gc, X}^{I_\bullet}$ and let $\underline F' \se \underline F''$ be two $I$-tuples of finite extensions of $F$ such that $\Zc$ admits representatives $Z_{\underline F'}$ and $Z_{\underline F''}$ over $\underline F'$ and $\underline F''$, respectively. Then $Z_{\underline F'}$ is the scheme-theoretic image of the map $Z_{\underline F''} \to \Gr_{\Gc, X}^{I_\bullet} \times_{X^I} \tilde{X}^I_{\underline F''} \to  \Gr_{\Gc, X}^{I_\bullet} \times_{X^I} \tilde{X}^I_{\underline F'}$. 
		In particular, if $\Zc$ has a representative defined over some $\underline F'$, it is unique.
		\label{lemBdSchImage}  
	\end{lemma}
	\begin{proof}
		Let $\underline F'''$ be componentwise a common extension of $\underline F'$ and $\underline F''$ such that $Z_{\underline F'} \times_{\tilde X^I_{\underline F'}} \tilde{X}^I_{\underline F'''} =Z_{\underline F''} \times_{\tilde X^I_{\underline F''}} \tilde{X}^I_{\underline F'''}$.  
		Then the projection $Z_{\underline F'''} \hookrightarrow \Gr_{\Gc, X}^{I_\bullet} \times_{X^I} \tilde{X}^I_{\underline F'''}\to \Gr_{\Gc, X}^{I_\bullet} \times_{X^I} \tilde{X}^I_{\underline F'}$ factors through $Z_{\underline F'}$. But as the question is Zariski-local on $Z_{\underline F'}$, we may assume that $\tilde{X}^I_{\underline F'} = \Spec(R)$ and $Z_{\underline F'} = \Spec(A)$ are affine. 
		Then both $\tilde{X}^I_{\underline F'''} = \Spec(R')$ and $Z_{\underline F'''} = \Spec(A \otimes_R R')$ are affine as well. The map $R \to R'$ is a finite and faithfully flat and thus a universally injective map of $R$-modules.
		Hence, the map $A \to A \otimes_R R'$ is injective. This shows that $Z_{\underline F'}$ is the scheme-theoretic image of $Z_{\underline F'''} \to Z_{\underline F'}$. 
		By the functoriality properties of scheme-theoretic images the projection $Z_{\underline F''} \to \Gr_{\Gc, X}^{I_\bullet} \times_{X^I} \tilde{X}^I_{\underline F'}$ factors through $Z_{\underline F'}$ and has $Z_{\underline F'}$ as its scheme-theoretic image.
	\end{proof} 
	
	One way to construct bounds is to prescribe a bound for each leg as follows. 
	\begin{construction}
		\label{consBoundSingleLeg}
		Assume $\Lc^+_{X^I}\Gc \to X^I$ is flat. Recall that this is satisfied when $\Gc \to X$ is smooth by \cite[Lemma 2.11]{Richarz2016}.
		Let $(\Zc^{(i)})_{i \in I}$ be an $I$-tuple of bounds in $\Gr_{\Gc, X}$ with reflex schemes $\widetilde{X}_{\Zc^{(i)}}$. 
		From the $(\Zc^{(i)})_{i \in I}$ we construct a bound $\Zc = \prod_{i \in I} \Zc^{(i)}$ in $\Gr_{\Gc, X}^{I_\bullet}$ as follows.
		Let $\underline F' = ({F^{(i)}}')_{i \in I}$ be an $I$-tuple of finite extensions of $F$ such that $\Zc^{(i)}$ admits a representative over ${F^{(i)}}'$.
		Let $U \se X^I$ be the complement of all diagonals and set $U_{\underline F'} = U \times_{X^I} \tilde{X}^I_{\underline F'}$. 
		If no two different components $\tilde{X}_{F^{(i)}}$ are of the finite type (\ref{bdTypeFinite}) in Definition \ref{defnGlobalBound}, then $U_{\underline F'}$ is non-empty. 
		In this situation we define
		$$  Z_{\underline F'} = \mr{image}\left( \prod_{i \in I} Z^{(i)}_{{F^{(i)}}'}|_{U_{\underline F'}} \hookrightarrow \Gr_{\Gc, X}^{I_\bullet} \times_{X^I} \tilde{X}_{\underline F'}^I \right)$$
		as the scheme-theoretic closure of the product of the $ Z^{(i)}_{{F^{(i)}}'}$ away from the diagonals via the map of Remark \ref{remBLDesc}.
		We denote by $\Zc$ the equivalence class of $Z_{\underline F'}$.
		
		In a similar fashion, we can define a bound $\Zc$ in $\Gr_{\Gc, X}^{I_\bullet}$ for any refinement $I_\bullet' = (I'_{j'})_{1 \leq j' \leq m'}$ of $I_\bullet$ and bounds $\Zc^{(I'_{j'})}$ in $\Gr_{\Gc, X}^{(I'_{j'})}$. 
	\end{construction}
	
	\begin{lemma}
		\label{lemBoundCons}
		In the situation of Construction \ref{consBoundSingleLeg}, let $\underline F'$ and $\underline F''$ be two $I$-tuples of finite extensions of $\underline F$ such that the $\Zc^{(i)}$  admit representatives over both $(F^{(i)})'$ and $(F^{(i)})''$. 
		Let $Z_{\underline F'}$ and $Z_{\underline F''}$ be the corresponding subschemes of the affine Grassmannian as constructed in Construction \ref{consBoundSingleLeg}. 
		\begin{enumerate}
			\item Then $Z_{\underline F'}$ and $Z_{\underline F''}$ are equivalent and stable under the $\Lc_{X^I}^+\Gc$-action. In particular, $\Zc$ is a bound and independent of the choice of $\underline F'$.
			The reflex scheme of $\Zc = \prod_{i \in I} \Zc^{(i)}$ is 
			$\widetilde{X}^I_\Zc = \prod_{i \in I} \widetilde{X}_{\Zc^{(i)}}$.
			\item 
			Assume that $\Zc$ admits a representative $Z_{\underline F'''}$ over $\underline F'''$. Then for every $i \in I$ the bound $\Zc^{(i)}$ admits a representative $Z^{(i)}_{(F^{(i)})'''}$ over $(F^{(i)})'''$ and $Z_{\underline F'''}$ arises via Construction \ref{consBoundSingleLeg} from the $Z^{(i)}_{ (F^{(i)})'''}$.
		\end{enumerate}
	\end{lemma}
	\begin{proof}
		\begin{enumerate}
			\item That $Z_{\underline F'}$ and $Z_{\underline F''}$ are equivalent follows from the fact that the formation of scheme-theoretic images commutes with flat base change.
			
			By assumption, the action map $\Lc_{X^I}^+\Gc|_{U_{\underline F'}} \times_{U_{\underline F'}} \prod_{i \in I} Z^{(i)}_{(F^{(i)})'}|_{U_{\underline F'}} \to \Gr_{\Gc, X}^{I_\bullet}|_{U_{\underline F'}}$ factors through $Z_{U_{\underline F'}}$. 
			By the functoriality of scheme-theoretic images together with the fact that $\Lc^+_{X^I} \Gc \to X^I$ is flat, also $Z_{\underline F'}$ is stable under the action of $\Lc^+_{X^I} \Gc$. 
			\item We first note that the representative $Z_{{F^{(i)}}'}^{(i)}$ of $\Zc^{(i)}$ can be recovered from $Z_{\underline F'}$ as the scheme-theoretic image
			$$ Z_{(F^{(i)})'}^{(i)} = \mr{image} \left( Z_{\underline F'}|_{U_{\underline F'}}  \to \prod_{i \in I} \Gr_{\Gc, X} |_{U_{\underline F'}} \xrightarrow{\mr{pr^{(i)}}} \Gr_{\Gc, X} \times_{X} \tilde{X}_{(F^{(i)})'} \right),$$
			where $\mr{pr}^{(i)}$ denotes the projection to the $i$-th component.
			
			Let now $\underline F'''$ be an $I$-tuple of extensions of $\underline F$ as in the assertion and let $Z_{\underline F'''}$ be a representative of $\Zc$ over $\underline F'''$. As above, we denote by $Z_{(F^{(i)})'''}^{(i)}$ the scheme-theoretic image of the map $Z_{\underline F'''}|_{U_{\underline F'''}} \xrightarrow{\mr{pr}^{(i)}} \Gr_{\Gc, X} \times_X \tilde{X}_{(F^{(i)})'''}$. Without loss of generality, we may assume that $F'$ is an extension of $F'''$. As the formation of scheme-theoretic images commutes with flat base change, $Z_{(F^{(i)})'''}^{(i)}$ and $Z_{(F^{(i)})'}^{(i)}$ are equivalent. In particular, $\Zc^{(i)}$ admits a representative over $(F^{(i)})'''$. That $Z_{\underline F'''}$ arises from Construction \ref{consBoundSingleLeg} then follows from the fact that representatives of a bound over a given field extension are unique by Lemma \ref{lemBdSchImage}.
		\end{enumerate}
	\end{proof}

	\begin{remark}
		Note that we crucially used the the flatness of $\Lc^+_{X^I} \Gc$ to show that the representatives of $\Zc$ are stable under the $\Lc^+_{X^I} \Gc$-action. 
		It does not seem to be clear whether or not $\Lc^+_{X^I} \Gc \to X^I$ is flat in general for flat affine group schemes $\Gc \to X$ of finite type.
	\end{remark}

	On the other hand, from a bound in $\Gr_{\Gc, X}^{I_\bullet}$ we can for every subset $I' \se I$ construct a bound in $\Gr_{\Gc, X}^{(I'_\bullet)}$ essentially by projection to the $I'$-components, where $I'_\bullet = (I' \cap I_j)_{1 \leq j \leq m}$. This generalises the construction from \cite[proof of Proposition 4.3.3]{Rad2017}.
	\begin{construction}
		\label{consBdProj}
		Let $\Zc$ be a bound in $\Gr_{\Gc, X}^{I_\bullet}$ with a representative $Z_{\underline F'}$ defined over $\underline F'$. Let $I' \se I$.
		We denote by $Z_{\underline F'}^{(I')}$ the scheme-theoretic image of the projection $Z_{\underline F'}|_{U_{\underline F'}} \hookrightarrow (\prod_{i \in I} \Gr_{\Gc, X} )|_{U_{\underline F'}} \to \Gr_{\Gc,X}^{I'_\bullet} \times_{X^{I'}} \widetilde{X}_{\underline F'}^{I'}$ to the $I'$-components.
		We denote by $\Zc^{(I')}$ the equivalence class of $Z_{\underline F'}^{(I')}$. Clearly, $\Zc^{(I')}$  does not depend on the choice of representative $Z_{\underline F'}$.
		
		For the computation of the reflex scheme of $\Zc^{(I')}$ we note that  $\Aut_{\Zc^{(I')}}(\underline{F}')$ is just the projection of $\Aut_\Zc(\underline{F}')$ to the $I'$-components.
		In particular, when $\Zc = \prod_{i \in I} \Zc^{(i)}$ as in \ref{consBoundSingleLeg}, we find $\Zc^{(I')} = \prod_{i \in I'} \Zc^{(i)}$ and the map $\prod_{i \in I} \widetilde{X}^I_{\Zc^{(i)}}  = \widetilde{X}^I_\Zc \to \widetilde{X}^{I'}_{\Zc^{(I')}} = \prod_{i \in I'}\widetilde{X}^I_{\Zc^{(i)}}$ is projection to the $I'$-components.
		
		Assume additionally that $\Lc^+_X \Gc \to X$ is flat. By a similar argument as in Lemma \ref{lemBoundCons}, we see that $Z^{(I')}_{\underline F'}$ is stable under $\Lc^+_{X^{I'}} \Gc \times_{X^{I'}} \widetilde{X}_{\underline F'}$. Hence, $\Zc^{(I')}$ is a bound.
	\end{construction}

	We give two conditions for points on the affine Grassmannian to factor through (representatives of) a bound. These lemmas will be helpful to check boundedness of (global or local) shtukas in the next section.
	\begin{lemma}
		Let $\Zc$ be a bound in $\Gr_{\Gc, X}^{I_\bullet}$. Let $S$ be a $k$-scheme and let $g \in \left( \Gr_{\Gc, X}^{I_\bullet} \times_{X^I} \widetilde{X}^I_{\Zc} \right) (S)$. Then $g \times \mr{id}_{\widetilde X_{\underline F'}} \colon S \times_{\widetilde X_\Zc^I} \widetilde{X}^I_{\underline F'} \to  \Gr_{\Gc, X}^{I_\bullet} \times_{X^I} \widetilde{X}^I_{\underline F'}$ factors through $Z_{\underline F'}$ for one $I$-tuple of finite extensions $\underline F'$ for which a representative $Z_{\underline F'}$ of $\Zc$ exists if and only if $g \times \mr{id}_{\tilde X_{\underline F'}}$ factors through $Z_{\underline F'}$ for all such $I$-tuples of finite extensions $\underline F'$.   
		\label{lemBoundAllReps}
	\end{lemma}
	\begin{proof}
		Assume that $g \times \mr{id}_{\widetilde X_{\underline F'}}$ factors through $Z_{\underline F'}$ for some $\underline F'$ and let $\underline F''$ be a second $I$-tuple of finite extensions such that $\Zc$ admits a representative over $\underline F''$.
		By replacing $\underline F'$ by the componentwise compositum of $\underline F'$ and $\underline F''$, we may assume that $\underline F'$ is a componentwise extension of $\underline F''$.
		The claim then follows from Lemma \ref{lemBdSchImage}.
	\end{proof}

	\begin{lemma}
		\label{lemSectBoundRefl}
		Let $\Zc$ be a bound in $\Gr_{\Gc, X}^{I_\bullet}$. 
		Let $S$ be a $k$-scheme 
		and let $g \colon S \to \Gr_{\Gc, X}^{I_\bullet} \times_{X^I} \widetilde X^I_{\underline F'}$ be a point in the affine Grassmannian for some $I$-tuple of finite extensions $\underline F'$ such that $\Zc$ admits a representative over $\underline F'$.
		\begin{enumerate}
			\item Assume that $ g \times \mr{id}_{\widetilde{X}_{\underline F'}^I} \colon S \times_{Y} \tilde X^I_{\underline F'} \to \Gr_{\Gc, X}^{I_\bullet} \times_{X^I}  \tilde{X}^I_{\underline F'} $ factors through $Z_{\underline F'}$ for some scheme $Y \to \widetilde{X}^I_\Zc$ such that the map $\widetilde{X}^I_{\underline{F}''} \to \widetilde{X}^I_\Zc$ factors through $Y$. Then $g$ factors through $Z_{\underline F'}$. 
			\label{lemSectBoundReflExt}
			\item Assume that $\Zc$ admits a representative over $\widetilde{X}^I_\Zc$. If $g$ factors through $Z_{\underline F'}$, then the induced map
			$$ g \times \mr{id}_{\tilde{X}_{\underline F''}^I} \colon S \times_{\widetilde X^I_\Zc} \tilde X^I_{\underline F''} \to \Gr_{\Gc, X}^{I_\bullet} \times_{X^I} \widetilde X^I_{\underline F''} $$
			factors through $Z_{\underline F''}$ for all finite extensions $\underline F''$ such that $\Zc$ admits a representative $Z_{\underline F''}$ over $\underline F''$.
			\label{lemSectBoundReflDesc}
		\end{enumerate}
	\end{lemma}
	\begin{proof}
		\begin{enumerate}
			\item 
			By construction, $g$ factors through $S \to S \times_{Y} \tilde X^I_{\underline F'} \xrightarrow{g \times \mr{id}_{\tilde{X}_{\underline F'}^I}} \Gr_{\Gc, X}^{I_\bullet} \times  \tilde X^I_{\underline F'}$.
			\item 
			Assume that $g \colon S \to \Gr_{G, X^I, I_\bullet} \times_{X^I} \widetilde{X}_{\underline F'}^I$ factors through $Z_{\underline F'}$. 
			By Lemma \ref{lemBdSchImage}, the induced map $g \colon S \to \Gr_{G, X^I, I_\bullet} \times_{X^I} X_{\Zc}^I$ factors through the representative $Z_{\widetilde X_\Zc}$ over the reflex scheme. Hence, $ g \times \mr{id}_{\tilde{X}_{\tilde F}^I} \colon S \times_{X^I_\Zc} \widetilde X^I_{\underline F''} \to \Gr_{\Gc, X}^{I_\bullet} \times_{X^I} \tilde X^I_{\underline F''} $ factors through $Z_{\underline {F}''}$ for all finite extensions $\underline{F}''$ by the argument of Lemma \ref{lemBoundAllReps}.
		\end{enumerate}
	\end{proof}
	
	\begin{remark}
		We do not know whether (\ref{lemSectBoundReflDesc}) holds without the assumption that $\Zc$ admits a representative over the reflex field. 
		However, this assumption is satisfied for all examples of bounds considered in practice, in particular for all bounds considered in the following two subsections.
	\end{remark}

	\subsection{Generically defined bounds}
	We now discuss a special class of bounds, namely those that are \emph{generically defined} in the sense that they are given as the scheme-theoretic closure of their generic fibre.
	In this case, the theory simplifies significantly, for example we show that in this case representatives always exist over the reflex field.
	In this subsection we always assume without further mention that $\Lc_{X^I}^+\Gc$ is flat over $X^I$, which is the case for example if $\Gc \to X$ is smooth by \cite[Lemma 2.11]{Richarz2016}.  
	
	For this subsection we restrict ourselves to bounds of global type in the global setting or bounds of local type in the local setting.
	
	\begin{definition}
		Let $\Zc$ be a bound in $\Gr_{\Gc, X}^{I_\bullet}$. We say that $\Zc$ is \emph{generically defined} if there is an $I$-tuple of finite extensions $\underline F'$ of $F$ such that there exists a representative $Z_{\underline F'}$ over $\underline F'$ which is the scheme-theoretic closure of its generic fibre over every connected component of $\widetilde{X}^I_{\underline F'}$.
	\end{definition}
	
	\begin{remark}
		A bound with a representative $Z_{\underline F'}$ that is flat over $\widetilde{X}_{\underline F'}$ is clearly generically defined, and all its representatives will then be flat.
		When $I$ has only one element, any representative of a generically defined bound in $\Gr_{\Gc,X}$ is automatically flat over $\widetilde{X}_{F'}$ (as the base is a Dedekind scheme in this case). 
	\end{remark}
	
	\begin{lemma}
		\label{lemGen}
		Let $\Zc$ be a generically defined bound in $\Gr_{\Gc, X}^{I_\bullet}$. Then every representative of $\Zc$ arises as the closure of its generic fibre.
	\end{lemma}
	\begin{proof}
		Let $\underline F'$ be such that the representative $Z_{\underline F'}$ of $\Zc$ is the closure of its generic fibre. As taking the scheme theoretic image commutes with flat base change, this property remains true for all representatives defined over some finite extension of $\underline F'$. 
		Let $\underline F''$ be a second $I$-tuple of field extensions over which a representative of $\Zc$ exists. Without loss of generality we may assume that $\underline F'$ is a componentwise extension of $\underline F''$. 
		Then the induced map of the generic fibre $Z_{\underline F', \eta_{\underline F'}^I} \to Z_{\underline F''}$ factors through $Z_{\underline F'', \eta_{\underline F''}^I}$ and both the maps $Z_{\underline F', \eta_{\underline F'}^I} \to Z_{\underline F'} \to Z_{\underline F''}$ are scheme-theoretically surjective. This shows the claim.
	\end{proof}

	We show that Constructions \ref{consBoundSingleLeg} and \ref{consBdProj} preserve the property of being generically defined.
	\begin{lemma}
		Let $\Zc$ be a bound in $\Gr_{\Gc, X}^{I_\bullet}$. Then $\Zc$ is generically defined if and only if there exist generically defined bounds $(\Zc^{(i)})_{i \in I}$ in $\Gr_{\Gc, X}$ such that $\Zc$ arises from $(\Zc^{(i)})_{i \in I}$ as in Construction \ref{consBoundSingleLeg}.
		\label{lemGenDefProd}
	\end{lemma}
	The $\Zc^{(i)}$ are unique by the Lemma \ref{lemBoundCons}. We call them the components of $\Zc$.
	\begin{proof}
		Assume first that $\Zc$ arises from generically defined bounds $\Zc^{(i)}$ via Construction \ref{consBoundSingleLeg}.
		Let $\underline F'$ be an $I$-tuple of finite extensions of $F$ such that $\Zc^{(i)}$ admits a representative over $(F^{(i)})'$ for every $i \in I$. 
		As the generic point of $X^I$ lies in the complement $U$ of all diagonals and $\prod_{i \in I} Z^{(i)}_{(F^{(i)})'}$ is flat over $X^I$,
		$\left( \prod_{i \in I} Z^{(i)}_{F'} \right)|_{U_{\underline F'}}$ is the closure of its generic fibre.
		
		Conversely, assume that $\Zc$ is generically defined and let $(\Zc^{(i)})_{i \in I}$ be the components of $\Zc$ from Construction \ref{consBdProj}. By construction, all $\Zc^{(i)}$ are generically defined and the generic fibre of any representative of $\Zc$ is the product of the corresponding representatives of the $\Zc^{(i)}$.
		This shows that $\Zc$ arises from the $\Zc^{(i)}$ via Construction \ref{consBoundSingleLeg}.
	\end{proof}
	
	\begin{lemma}
		Let $\Zc$ be a generically defined bound. Its reflex scheme is given by $\widetilde{X}^I_\Zc = \prod_{i \in I} \widetilde{X}_{\Zc^{(i)}}$ where the $\Zc^{(i)}$ are the components of $\Zc$ as in Lemma \ref{lemGenDefProd}.
		Moreover, $\Zc$ admits a representative over its reflex scheme.
		\label{lemGenRefl}
	\end{lemma}
	
	\begin{proof}
		The description of the reflex scheme follows from Lemma \ref{lemBoundCons}. Moreover, it suffices to show that $\Zc$ admits a representative over its reflex scheme in the case $I = \{\ast\}$. In this case $\widetilde{X}_\Zc = \widetilde{X}_{F_\Zc}$ for some finite extension $F_\Zc/F$. 
		In particular there is a finite Galois extension $F'/F_\Zc$ such that there exists a representative $Z_{F'}$ of $\Zc$ over $F'$.
		By assumption it suffices to descend $Z_{F', \eta_{F'}}$ to some subscheme $Z_{{F_\Zc}, \eta_{{F_\Zc}}}$ of the affine Grassmannian over $\eta_{F_\Zc}$, the representative of $\Zc$ over $F_\Zc$ will then be given by the scheme-theoretic closure of $Z_{{F_\Zc}, \eta_{{F_\Zc}}^I}$.
		But $\eta_{F'} \to \eta_{{F_\Zc}}$ is a finite \'etale Galois cover with Galois group $\Gal(F'/{F_\Zc})$ and $Z_{F', \eta_{F'}}$ is fixed by the Galois action by assumption. Thus,  $Z_{F', \eta_{F'}}$ descends to ${F_\Zc}$.
	\end{proof}
	
	\begin{remark}
		By Lemma \ref{lemGenDefProd}, a generically defined bound is uniquely determined by an $I$-tuple $(Z^{(i)})$ of (equivalence classes of) $\Lc_X^+\Gc \times_X \eta$-stable closed subschemes of the generic fibre $\Gr_{\Gc,X} \times_X \eta$ of the Beilinson-Drinfeld affine Grassmannian.
		
		If the generic fibre $G$ of $\Gc$ is reductive, the generic fibre of Beilinson-Drinfeld Grassmannian for $\Gc$ can be identified (non canonically) with the (classical) affine Grassmannian $\Gr_G$ for $G \times_{\Spec(F)} \Spec(F\dsq{\varpi})$ by \cite[Section 0.2]{Richarz2021a}.
		Hence, $I$-tuples of equivalence classes of closed subschemes (that are stable under the loop group action) in $\Gr_G$ determine bounds in $\Gr_{\Gc, X}^{I_\bullet}$ for all smooth models $\Gc \to X$ of $G$ by taking closures. 
		Note that the bounds do not depend on the choice of isomorphism $\Gr_{\Gc,X} \times_X \eta \cong \Gr_G$. 
		
		By a slight abuse of notation, we call such an equivalence class $\Zc = (\Zc^{(i)})_{i \in I}$ in $\Gr_G$ a \emph{generically defined bound for $G$} and denote the resulting bound in $\Gr_{\Gc, X}^{I_\bullet}$ for any smooth model $\Gc$ of $G$ by $\Zc$ as well.
		\label{remBdGenFibre}
	\end{remark}

	\subsection{Examples of bounds}
	\subsubsection{Beilinson-Drinfeld Schubert varieties}
	The bounds commonly used in practice are bounds given by Beilinson-Drinfeld Schubert varieties.
	We recall and generalise their construction following \cite[D\'efinition 1.12]{Lafforgue2018}.
	We assume now that the generic fibre $G$ of $\Gc$ is reductive.
	Let $\mu$ be a conjugacy class of geometric cocharacters of $G$ with reflex field $F_\mu$ and let $F'/F$ be a finite separable extension that splits $G$.
	By a slight abuse of notation, we denote by $\Gr_G$ the (classical) affine Grassmannian $\Gr_{G \times_F \Spec(F\dsq{t})}$ over $F$.
	We denote by $\Gr_{G_{F'}}^{\leq \mu} \se \Gr_{G_{F'}} = \Gr_{G} \times_F F'$ the Schubert variety corresponding to $\mu$, which descends to the reflex field $F_\mu$ of $\mu$.
	For an $I$-tuple of conjugacy classes of geometric cocharacters $\underline{\mu} = (\mu_i)_{i \in I}$ the (iterated) BD-Schubert variety $\Gr_{\Gc, X}^{I_\bullet, \leq \underline{\mu}}$ is the bound induced by the $\Gr_{G}^{\leq \mu_i}$ as in the previous section, compare Remark \ref{remBdGenFibre}. More explicitly, it is defined as follows.
	
	\begin{definition}[compare {\cite[Definition 3.1]{Zhu2014}, \cite[D\'efinition 1.12]{Lafforgue2018}}]
		\begin{enumerate}
			\item  Let $\mu$ be a conjugacy class of cocharacters of $G$. Then 
			$\Gr^{\leq \mu}_{\Gc, X}$ is defined as the scheme-theoretic image 
			$$ \Gr^{\leq \mu}_{\Gc, X} = \mr{image} \left(\Gr_{G}^{\leq \mu} \hookrightarrow \Gr_{\Gc, X} \times_X X_\mu \right)$$
			where $X_\mu = \widetilde{X}_{F_\mu}$ is the reflex curve of $\mu$. 
			\item Let $\underline{\mu} = (\mu_i)_{i \in I}$ be an $I$-tuple of conjugacy classes of geometric cocharacters of $G$ and let $F_{\underline{\mu}} = (F_{\mu_i})_{i \in I}$ be the $I$-tuple of reflex fields $F_{\mu_i}$ of the $\mu_i$. Then $\Gr_{\Gc, X}^{I_\bullet,\leq \underline \mu}$ is defined from the $(\Gr_{\Gc, X}^{\leq \mu_i})_{i \in I}$ as in Construction \ref{consBoundSingleLeg}. More precisely, it admits a representative defined over $F_{\underline{\mu}}$ given by 
			$$ \Gr_{\Gc, X}^{I_\bullet, \leq \underline \mu} = \mr{image} \left( \prod_{i \in I} \Gr^{\leq \mu_i}_{\Gc, X}|_{U_{F_{\underline{\mu}}}} \hookrightarrow \Gr_{\Gc, X}^{I_\bullet} \times _{X^I} X^I_{\underline{\mu}}\right),$$
			where $U$ denotes the complement of all diagonals in $X^I$, $U_{F_{\underline{\mu}}}  = U \times_{X^I} X_{\underline{\mu}}^I$ and the map $ \prod_{i \in I} \Gr^{\leq \mu_i}_{\Gc, X}|_{U_{F_{\underline{\mu}}}} \hookrightarrow \Gr_{\Gc, X}^{I_\bullet} \times _{X^I} X^I_{\underline{\mu}}$ is induced by the isomorphism of Remark \ref{remBLDesc}. 
		\end{enumerate}
	\end{definition}
	This construction generalises many of the boundedness conditions used in applications. When $\Gc$ is constant split reductive, these bounds are constructed and used in \cite{Lafforgue2018}. In particular, when $\Gc = \GL_n$, $I_{\bullet} = (\{1\}, \{2\})$ and $\underline{\mu} = ((1,0,\ldots,0),(0,\ldots,0,-1))$ this bound gives rise to the notion of Drinfeld shtuka as defined in \cite{Drinfeld1987a}.
	Further examples include the shtukas for unitary groups considered in \cite{Feng2024} for $\Gc = U(n)$  and $\underline{\mu} = ((1,0,\ldots,0), \ldots, (1,0, \ldots,0))$.
	Moreover, in the local setting (in the case that $\Gc$ is constant split reductive these bounds are considered in  \cite{Genestier2018, LiHuerta2023}.
	
	Both  $\Gr^{\leq \mu}_{\Gc, X}$ and $\Gr^{ I_\bullet, \leq \underline \mu}_{\Gc, X}$ are bounds in the sense of Definition \ref{defnGlobalBound} with reflex fields $F_{\mu}$ and $F_{\underline{\mu}}$, respectively. In particular, they are generically defined. For $\Gr^{I_\bullet, \leq \underline \mu}_{\Gc, X}$ this follows from Lemmas \ref{lemBoundCons} and \ref{lemGenDefProd}.

	\subsubsection{Further examples of global bounds}
	Let now $X$ be a smooth projective and geometrically connected curve over $k$. We call bounds (in the sense of Definition \ref{defnGlobalBound}) defined in this situation also \emph{global bounds}. 
	We now discuss various other constructions of global bounds used in the literature.
	\begin{example}
		\begin{enumerate}
			\item 
			Let $W$ be an $E$-linear representation of $(^LG)^I$, where $^LG$ denotes the $L$-group of $G$ and where $E$ is a finite extension of $\Q_\ell$ such that all highest weights of the dual group of $G$ are defined over $E$. 
			Then $\bigcup_{\underline{\mu}}\Gr_{\Gc, X}^{I_\bullet, \leq \underline{\mu}}$, where $\underline{\mu}$ runs over all highest weights appearing in $W$, is defined over the compositum of all reflex fields of all the highest weights appearing and naturally descends to a closed subscheme $\Gr_{\Gc, X}^{I_\bullet, W}$ of $\Gr_{\Gc, X}^{I_\bullet}$ defined over $X^I$.
			\label{exBoundLaf}
			\item 
			In the case that $\Gc$ is constant split reductive, \cite{Varshavsky2004} defines bounds as follows. For a $\Gc$-torsor $\Ec$ and a dominant coweight $\lambda$ of $\Gc$ we denote by $\Ec_\lambda = \Ec \times^\Gc V_\lambda$, where $V_\lambda$ is the Weyl module for $\lambda$.
			Then $\Gr_{G,X}^{I_\bullet, \lesssim \underline{\mu}}$ is defined to be the closed subscheme of $\Gr_{\Gc, X}^{I_\bullet}$ parametrising points $ ((x_i)_{i \in I}, (\Ec_j)_{j = 0, \ldots, m}, (\ph_j)_{j = 1, \ldots, m}, \varepsilon)$ such that
			$ \ph_{j, \lambda}(\Ec_{j, \lambda}) \se \Ec_{j-1, \lambda} \left(\sum_{i \in I_j} \langle \lambda, \mu_i \rangle \Gamma_{x_i} \right)$
			for all $1 \leq j \leq m$ and all dominant weights $\lambda$ of $\Gc$.
			
			In contrast to the Beilinson-Drinfeld Schubert varieties defined above, these bounds admit a moduli description by definition. However, $\Gr_{\Gc,X}^{I_\bullet, \lesssim \underline{\mu}}$ might fail to be flat or reduced.
			\cite[Definition 3.3.1]{Neupert2016} gives a variant of this definition also for quasi-split reductive groups over $\Fq$. 
			\label{exBoundVar}
			\item For general flat affine group schemes of finite type $\Gc \to X$, \cite{Rad2019a} give a boundedness condition by choosing a closed immersion $f \colon \Gc \hookrightarrow \SL(\mathcal V)$ for some vector bundle $\mathcal V$ on $X$ and an $I$-tuple of cocharacters $\underline{\mu} = (\mu_i)_{i \in I}$ of $\SL_n$, where $n$ is the rank of $\mathcal{V}$.
			The condition given in corresponds to the bound $\Gr_{\Gc, X}^{I_\bullet, \lesssim \underline{\mu}} := \Gr_{\SL(\mathcal V), X}^{ I_\bullet, \lesssim \underline{\mu}} \times_{\Gr_{\SL(\mathcal V), X}^{ I_\bullet}} \Gr_{\Gc, X}^{I_\bullet}$ defined via pullback of the corresponding bound for $\SL_n$ along the map $f_\ast \colon \Gr_{\Gc, X}^{I_\bullet} \to \Gr_{\SL(\mathcal V), X}^{I_\bullet}$.
			
			As noted in \cite[Remark 3.20]{Rad2019a}, when $\Gc$ is constant split reductive, this definition of bounds agrees with the bound $\Gr_{\Gc, X}^{I_\bullet, \lesssim f \circ \underline{\mu}}$ of \cite{Varshavsky2004} as in (\ref{exBoundVar}) for the $I$-tuple of cocharacters $f \circ \underline{\mu} = (f \circ \mu_i)_{i \in I}$ of $\Gc$.
			\label{exBoundHar}
		\end{enumerate}
		\label{exBound}
		All the bounds considered in this example clearly have reflex scheme $X^I$ and a representative defined over $X^I$.
		Moreover, $\Gr_{\Gc, X}^{I_\bullet, W}$ is generically defined by construction.
	\end{example}

	\subsubsection{Local bounds}
	
	By a \emph{local bound} we mean a bound in the sense of Definition \ref{defnGlobalBound} when $X = \Spec(k \dsq{\varpi})$ for a finite field $k$. Below, these are used as bounds for local shtukas.
	\cite[Definition 3.5]{Hartl2011} give a definition of local bounds in the style of \cite{Varshavsky2004} for constant split reductive groups when $I = \{\ast\}$, compare also \cite[Example 4.11]{Rad2015}.
	
	A different definition of local bounds is given in \cite[Definition 4.5]{Rad2015}, where bounds are not defined as subschemes of $\Gr_{\Gc, X}$, but rather as formal subschemes of $\Gr_{\Gc} \hat \times_{\Spec k} \Spf(k \dsq{\varpi})$. 
	By \cite[Lemma 1.4]{LiHuerta2023}, we have 
	$$\Gr_{\Gc} \hat{\times}_{\F} \Spf(\Oc) \cong \Gr_{\Gc, k \dsq{\varpi}} \hat{\times}_{\Spec k \dsq{\varpi}} \Spf(k \dsq{\varpi}),$$ 
	so the local bounds in our setting (for $I = \{ \ast\})$) define bounds in the sense of \cite{Rad2015} by pullback under this isomorphism.

	\subsection{Construction of local bounds from global bounds}
	\label{secLocGlobBounds}
	We explain how to construct local bounds from global ones in our setting following the construction in \cite[Proposition 4.3.3]{Rad2017}. 
	We generalise their construction to the iterated setting and moreover also construct local bounds at points on the diagonal of $X^I$.
	
	We consider the global setting. In other words, $X$ denotes a smooth geometrically connected curve over $k$ for the remainder of this section. Let $\Zc$ be an equivalence class of quasi-compact closed subschemes in $\Gr_{\Gc, X}^{I_\bullet}$.
	We fix a closed point $y^0$ of $X$ and assume there is a closed point $\underline y$ of the reflex scheme $\widetilde{X}_{\Zc}^{I}$ of $\Zc$ lying over $\Delta_{I}(y^0)$, where $\Delta_{I} \colon X \to X^{I}$ denotes the diagonal embedding.
	This is the case if and only if all components of local or finite type in $\Zc$ are defined at $y^0$. This is satisfied in particular in the case that $\Zc$ only has components of global type that will be relevant later. 
	
	Let $\underline F'$ be an $I$-tuple of finite extensions of $\underline F$ such that $\Zc$ admits a representative $Z_{\underline F'}$ over $\underline F'$. 
	We fix a closed point $\underline{y}'$ of $\widetilde{X}_{\underline{F}'}$ over $\underline y$ and let us denote by $y_i' \in \widetilde{X}_{(F^{(i)})'}$ its $i$-th component.
	We set $Z_{\underline F', \underline y'} = Z_{\underline F'} \times_{\tilde{X}_{\underline F'}^{I}} \Spec(\bigotimes_{\F_{y^0}} \Oc_{y'_i})$ where we denote by $\Oc_{y'_i}$ the completed local ring at $y_i'$ in $\widetilde{X}_{(F^{(i)})'}$.
	We denote by $\Zc_{\underline y}$ the equivalence class of $Z_{\underline F', \underline y'}$ in  $\Gr_{\Gc, \Spec(\Oc_{y^0})}^{I_\bullet}$ (which will be of local type in the sense of Definition \ref{defnGlobalBound}).  
	Note that components of global type for $\Zc$ become components of local type for $\Zc_{\underline{y}}$ while components of local (respectively finite) type for $\Zc$ remain of local (respectively finite) type for $\Zc_{\underline{y}}$.
	
	\begin{lemma}[compare {\cite[Proposition 4.3.3]{Rad2017}}]
		The equivalence class $\Zc_{\underline y}$ does neither depend on the choice of  representative $Z_{\underline F'}$ nor the choice of the point $\underline y'$.
		Moreover, when $\Zc$ is a (global) bound, $\Zc_{\underline y}$ is a (local) bound.
	\end{lemma}
	\begin{proof}
		We first assume that $\underline F'$ is a componentwise Galois extension of $\underline F$.  
		Let $\underline y''$ be a second place of $\tilde{X}_{\underline F'}$ lying over $\underline y$.
		As $Z_{\underline F'}$ is Galois-invariant, $Z_{\underline F', \underline y'}$ and $Z_{\underline F', \underline y''}$ are equivalent.
		Let now $\underline F''$ be a componentwise finite extension of $\underline F'$ and let $\underline y''$ be a closed point of $\underline F''$ lying over $\underline y'$. Then clearly $Z_{\underline F'', \underline y''} = Z_{\underline F', \underline y'} \times_{\Spec(\bigotimes_{\F_{y^0}} \Oc_{y'_i})} \Spec(\bigotimes_{\F_{y^0}} \Oc_{y''_i})$.
		This shows that $\Zc_{\underline y}$ is independent of the auxiliary choices. 
		Moreover, when $\Zc$ is stable under the loop group action so is $\Zc_{\underline y}$. 
	\end{proof}
	
	We slightly change notation and fix an $I$-tuple $\underline{y}^0 = (y^0_i)_{i \in I}$ of closed points of $X$ for the remainder of this subsection.
	For a closed point $y^0$ of $X$ we write $y^0 \in \underline{y}^0$ if $y^0 = y^0_i$ for some $i \in I$. We set $I^{y^0} = \{i \in I \colon y^{0}_i = y^0\}$ and define $I^{y^0}_\bullet = (I_j^{y^0})_{j \in J}$ analogously.
	
	\begin{definition}
		\begin{enumerate}
			\item 
			An \emph{$I$-tuple of local bounds at $\underline{y}^0$} is a tuple of (local) bounds $(\Zc_{y^0})_{y^0 \in \underline{y}^0}$, where $\Zc_{y^0}$ is defined in $\Gr_{\Gc, \Spec(\Oc_{y^0})}^{I^{y^0}_\bullet}$.
			\item Let $\Zc$ be a global bound and assume that $\Lc^+_{X^I} \Gc \to X^I$ is flat. For each ${y^0} \in \underline{y}^0$ we fix a closed point $\underline y$ on $\widetilde X^{I^{y^0}}_{\Zc^{(I^{y^0})}}$ lying over $\Delta_{I^{y^0}}(y^0)$. Its associated $I$-tuple of local bounds at $\underline{y}^0$ is $(\Zc_{\underline y})_{y^0 \in \underline y^0}$, where $\Zc_{\underline y}$ is as constructed above (using Construction \ref{consBdProj}).
		\end{enumerate}
		\label{defITupLocBounds}
	\end{definition}
	Note that when all $y^0_i$ are pairwise distinct, an $I$-tuple of bounds in the sense of the previous definition is an actual $I$-tuple of bounds $(\Zc_{y^0_i})_{i \in I}$, where each $\Zc_{y^0_i}$ lives inside $\Gr_{\Gc, \Oc_{y^0_i}}$. Our definition gives the natural generalisation to the case where $\underline{y}^0$ is also allowed to lie on a diagonal of $X^I$.

	\begin{example}
		By construction, the local bounds attached to a Schubert variety in the global setting is again a Schubert variety (in the local setting).
		Compare also \cite[Theorem 3.3.9]{Neupert2016} for a local-global compatibility for global bounds in the style of \cite{Varshavsky2004} and local bounds in the style of \cite{Hartl2011}.
	\end{example}


\section{Moduli spaces of bounded, iterated, global shtukas}
\label{secSht}

Global shtukas for $\GL_n$ were first introduced in \cite{Drinfeld1987a} and generalised to split reductive groups (respectively to flat affine group schemes of finite type) by \cite{Varshavsky2004} and \cite{Rad2019a}, respectively.
In this section, we recall the definition and basic properties of moduli spaces of (bounded, iterated, global) shtukas.
We rely of the definition of bounds given in the previous section and generalise the Serre-Tate theorem, the local model theorem and functoriality results to this setting.
The functoriality results will be crucial to the construction of integral models with deeper level structure below.

\subsection{Global shtukas}

We consider the ``global setting'' of the previous section:
Let $X$ be a smooth projective and geometrically connected curve over $k = \Fq$, let $\Gc \to X$ be a smooth affine group scheme,  
let $I$ be a finite set, and let $I_\bullet = (I_1, \ldots, I_m)$ be a partition of $I$.

We denote by $\Bun_\Gc$ the stack of $\Gc$-bundles on $X$. It is representable by a smooth Artin stack locally of finite type over $\Fq$ by \cite[Proposition 1]{Heinloth2010}.

\begin{definition}
	We denote by $\Hecke_{\Gc, X}^{I_\bullet}$ the stack fibred in groupoids over $\Fq$ whose $S$ valued points are given by tuples 
	$ ((x_i)_{i \in I}, (\Ec_j)_{j = 0, \ldots, m}, (\ph_j)_{j = 1, \ldots, m}),$
	where
	\begin{itemize}
		\item $x_i \in X(S)$ are points on $X$ called the \emph{characteristic sections} (or \emph{legs}) for $i \in I$,
		\item $\Ec_j \in \Bun_\Gc(S)$ are $\Gc$-bundles on $X_S$ for $0 \leq j \leq m$, and
		\item $\ph_j \colon \Ec_{j}|_{X_S \setminus \bigcup_{i \in I_j} \Gamma_{x_i}} \xrightarrow{ \cong } \Ec_{j-1}|_{X_S \setminus \bigcup_{i \in I_j} \Gamma_{x_i}}$ are isomorphisms of $\Gc$-bundles away from the graphs $\Gamma_{x_i} \se X_S$ of the sections $x_i$.
	\end{itemize}
	As above, we use the shorthand notation 
	$ \left( \Ec_m \overset{\ph_m}{\underset{\Gamma_{\underline{x}_m}}{\dashrightarrow}} \Ec_{m-1} \overset{\ph_{m-1}}{\underset{\Gamma_{\underline{x}_{m-1}}}{\dashrightarrow}} \ldots \overset{\ph_{1}}{\underset{\Gamma_{\underline{x}_1}}{\dashrightarrow}} \Ec_0 \right) $
	for points of $\Hecke_{\Gc, X}^{I_\bullet}$.
\end{definition}
The projection to $\Ec_j$ defines a map $p_j \colon \Hecke_{\Gc, X}^{I_\bullet} \to \Bun_{\Gc}$ for all $0 \leq j \leq m$. 

\begin{definition}[{\cite[Definition 3.3]{Rad2019a}}]
	The moduli space of shtukas $\Sht_{\Gc, X}^{I_\bullet}$ is defined by the Cartesian diagram
	\begin{center}
		\begin{tikzcd}
			\Sht_{\Gc, X}^{I_\bullet} \arrow[r]\arrow[d] & \Hecke_{\Gc, X}^{I_\bullet} \arrow{d}{(p_m, p_0)} \\
			\Bun_\Gc \arrow{r}{(\id, \sigma^\ast)} & \Bun_\Gc \times \Bun_\Gc.
		\end{tikzcd}
	\end{center}
	More explicitly, $\Sht_{\Gc, X}^{I_\bullet}$ is the stack fibred in groupoids over $\Fq$ whose $R$-valued points for an $\Fq$-algebra $R$ are given by tuples 
	$ ((x_i)_{i \in I}, (\Ec_j)_{j = 0, \ldots, m}, (\ph_j)_{j = 1, \ldots, m}, \theta),$
	where
	\begin{itemize}
		\item $x_i \in X(R)$ are points on $X$ called the \emph{characteristic sections} (or \emph{legs}) for $i \in I$,
		\item $\Ec_j \in \Bun_\Gc(R)$ are $\Gc$-bundles on $X_R$ for $0 \leq j \leq m$,
		\item $\ph_j \colon \Ec_{j}|_{X_R \setminus \bigcup_{i \in I_j} \Gamma_{x_i}} \xrightarrow{ \cong } \Ec_{j-1}|_{X_R \setminus \bigcup_{i \in I_j} \Gamma_{x_i}}$ are isomorphisms of $\Gc$-bundles away from the graphs $\Gamma_{x_i} \se X_R$ of the sections $x_i$, and
		\item $\theta \colon  \Ec_0 \xrightarrow{\cong}  \sigma^\ast\Ec_m$ is an isomorphism of $\Gc$-bundles on $X_R$.
	\end{itemize}
	We denote a point of $\Sht_{\Gc, X}^{I_\bullet}$ by 
	$ \left(   \Ec_m \overset{\ph_m}{\underset{\Gamma_{\underline{x}_m}}{\dashrightarrow}} \Ec_{m-1} \overset{\ph_{m-1}}{\underset{\Gamma_{\underline{x}_{m-1}}}{\dashrightarrow}} \ldots \overset{\ph_{1}}{\underset{\Gamma_{\underline{x}_1}}{\dashrightarrow}} \Ec_0 \xrightarrow{\theta} \sigma^\ast \Ec_m \right) .$
\end{definition}
The projection to the characteristic sections defines a map $\Sht_{\Gc, X}^{I_\bullet} \to X^I$. 
By \cite[Theorem 3.15]{Rad2019a}, $\Sht_{\Gc, X}^{I_\bullet}$ is an ind-Deligne Mumford stack that is separated and locally of ind-finite type over $X^I$.

Let $I_\bullet'$ be a second partition of $I$ that is finer than $I_\bullet$. The forgetful map 
$ \Sht_{\Gc,X}^{I'_\bullet} \to \Sht_{\Gc, X}^{I_\bullet}$
is an isomorphism over the complement $U$ of all diagonals in $X^I$
by the argument in \cite[Lemma A.8 a)]{Varshavsky2004}.
Let us fix a closed point $\underline{y}  \in X^I$. We denote by 
$  \Sht^{I_\bullet}_{\Gc, \underline{y}} = \Sht_{\Gc, X}^{I_\bullet} \times_{X^I} \Spf(\Oc_{\underline{y}})$
the restriction of the moduli space of shtukas to the formal neighbourhood $\Spf\left( \Oc_{\underline{y}} \right)$ of $\underline{y}$. By the previous observation, this stack does not depend on the choice of the partition $I_\bullet$ of $I$ whenever the $y_i$ are pairwise different. 
Usually, when the index set $I$, its partition $I_\bullet$ and the curve $X$ are understood from the context, we drop them from the notation and denote
$ \Sht_{\Gc} = \Sht^{I_\bullet}_{\Gc, X}.$

\begin{definition}
	Let $\Zc$ be a bound in the sense of Definition \ref{defnGlobalBound}. Let 
	$$ \underline{\Ec} = ((x_i)_{i \in I}, (\Ec_j)_{j = 0, \ldots, m}, (\ph_j)_{j = 1, \ldots, m}, \theta) \in (\Sht_{\Gc, X}^{I_\bullet} \times_{X^I} \widetilde X_\Zc^I)(R).$$
	By \cite[Lemma 3.4]{Haines2020a}, there exists an \'etale cover $\Spec(R') \to \Spec(R)$ such that 
	$\hat{\Gamma}_{\underline x_{R'}}\to\hat{\Gamma}_{\underline x}$ trivializes $\Ec_0|_{\hat{\Gamma}_{\underline x}}$.
	Fixing a trivialisation $\alpha \colon \Ec_0|_{\hat\Gamma_{\underline x_{R'}}} \xrightarrow{\cong} \Gc|_{\hat \Gamma_{\underline x_{R'}}}$
	defines a point 
	$$g =   \left( \Ec_m|_{\hat\Gamma_{\underline x_{R'}}} \overset{\ph_m}{\underset{\Gamma_{\underline{x}_m}}{\dashrightarrow}} \Ec_{m-1}|_{\hat\Gamma_{\underline x_{R'}}} \overset{\ph_{m-1}}{\underset{\Gamma_{\underline{x}_{m-1}}}{\dashrightarrow}} \ldots \overset{\ph_{1}}{\underset{\Gamma_{\underline{x}_1}}{\dashrightarrow}} \Ec_0|_{\hat\Gamma_{\underline x_{R'}}} \overset{\varepsilon}{\rightarrow}  \Gc|_{\hat\Gamma_{\underline x_{R'}}}  \right) \in (\Gr_{\Gc, X}^{I_\bullet} \times_{X^I} \widetilde X_\Zc^I)(R'),$$ using the description of the Beilinson-Drinfeld affine Grassmannian of Remark \ref{remBLDesc}. 
	We say that $\underline{\Ec}$ is \emph{bounded by} $\Zc$ if for some (or by Lemma \ref{lemBoundAllReps} equivalently for all) $I$-tuples of finite extensions $\underline F'$ such that $\Zc$ admits a representative over $\underline F'$, the map 
	$$ g \times_{\widetilde X_{\Zc}^I} \id_{\tilde{X}^I_{F'}} \colon \Spec(R') \times_{\widetilde X_\Zc^I} \tilde{X}^I_{\underline F'} \to (\Gr_{\Gc, X}^{I_\bullet} \times_{X^I} \tilde X_{\underline F'}^I) $$
	factors through $Z_{\underline F'}$.	
	As $\Zc$ is invariant under the left $\Lc_{X^I}^+\Gc$-action, the definition is independent of the choices of the \'etale cover $R \to R'$ and the trivialisation $\alpha$.

	We denote the moduli stack of shtukas bounded by $\Zc$ by $\Sht_\Gc^\Zc = \Sht^{I_\bullet, \Zc}_{\Gc, X}$. It is a Deligne-Mumford stack defined over $\widetilde X_\Zc^I$ and a closed substack of $\Sht_{\Gc} \times_{X^I} \widetilde X_\Zc^I$.
	When the bound $\Zc$ is given by Beilinson-Drinfeld Schubert varieties corresponding to an $I$-tuple $\underline{\mu}$ of conjugacy classes of geometric cocharacters of $G$, we write $\Sht_\Gc^{\leq \underline{\mu}} = \Sht_{\Gc, X}^{I_\bullet, \leq \underline \mu} = \Sht_{\Gc, X}^{I_\bullet, \Zc}$. 
\end{definition}

\subsection{Local shtukas with multiple legs and local-global compatibility}

We again consider the local setting.
Let $\Oc$ be the ring of integers of a local field in characteristic $p$ with residue field $k$ a finite extension of $\Fq$. 
The choice of a uniformiser $\varpi$ of $\Oc$ induces an isomorphism $\Oc \cong k\dsq{\varpi}$.
Let $\Gc \to \Spec \Oc$ be a smooth affine $\Oc$-group scheme, let $G$ be its generic fibre.
As before, let $I$ be a finite set and let $I_\bullet = (I_1, \ldots, I_m)$ be a partition of $I$.
We set $\Oc_I = \Oc \hat{\otimes}_k \ldots \hat{\otimes}_k \Oc$, i.e., $\Oc_I \cong k\dsq{\varpi_i}_{i \in I}$.
We denote by $\Nilp_{\Oc_I}$ the category of $\Oc_I$-algebras $R$ such that the images $\zeta_i$ of $\varpi_i$ in $R$ under the structure map $\Oc_I \to R$ are nilpotent. 

Similarly to the global loop group defined in the previous section, the (local) loop group $LG$ of $G$ is defined as the functor on $k$-algebras
$$ R \mapsto G(R \dbr{\varpi}).$$
Analogously, the positive loop group $L^+\Gc$ is defined as the functor on $k$-algebras
$$ R \mapsto \Gc(R \dsq{\varpi}).$$
The functor $LG$ is representable by an ind-affine group ind-scheme while $L^+\Gc$ is representable by an affine group scheme over $k$. 

By \cite[Proposition 2.4]{Rad2015}, the category of $L^+\Gc$-torsors over a $k$-algebra $R$ is equivalent to the category of formal $\Gc$-torsors on $\Spf(R \dsq{\varpi})$. By a slight abuse of notation, we denote by $\Ec\left[\frac 1 \varpi\right] = \Ec \times^{L^+\Gc} LG$ the associated $LG$-torsor of an $L^+\Gc$-torsor $\Ec$. 
Note that for $R \in \Nilp_{\Oc_I}$, we have $R\dbr{\varpi} = R \dsq{\varpi} \left[\frac{1}{\varpi - \zeta_i} \right]$ for all $i \in I$ as the $\zeta_i$ are nilpotent. 

\begin{definition}
	A \emph{local shtuka} (with $I$-legs) over $R \in \Nilp_{\Oc_I}$ is given by a tuple 
	$((\Ec_j)_{j = 0, \ldots, m}, (\ph_j)_{j = 1, \ldots, m}, \theta),$ where 
	\begin{itemize}
		\item the $\Ec_j$ are $L^+\Gc$-torsors on $R$,
		\item the $\ph_j \colon \Ec_j[\frac{1}{\varpi}] \xrightarrow{\cong} \Ec_{j-1}[\frac{1}{\varpi}]$ are isomorphisms of the associated $L\Gc$-torsors,
		\item and $\theta \colon \Ec_0 \xrightarrow{\cong} \sigma^\ast \Ec_m$ is an isomorphism of $L^+\Gc$-torsors on $R$.
	\end{itemize} 
\end{definition}
In a similar fashion as in the global setting, we define bounded local shtukas.
\begin{definition}
	Let $\Zc$ be a (local) bound in the sense of Definition \ref{defnGlobalBound} and let $\underline{y}$ be a closed point of $\widetilde{X}^I_\Zc$. Let $ \underline{\Ec} = (\Ec, \ph)$ be a local shtuka over $R \in \Nilp_{\Oc_{\underline{y}}}$. 
	Let $R \to R'$ be an fppf-cover that trivialises $\Ec_0$ and choose a trivialisation $\alpha$ of $\Ec_0$ over $R'$. 
	Using the equivalence of $L^+\Gc$-torsors and formal $\Gc$-torsors of \cite[Proposition 2.4]{Rad2015}, 
	$   \left( \Ec_m \overset{\ph_m}{\dashrightarrow} \Ec_{m-1} \overset{\ph_{m-1}}{\dashrightarrow} \ldots \overset{\ph_{1}}{\dashrightarrow} \Ec_0 \overset{\alpha}{\rightarrow}  \Gc  \right)$
	defines a point $g \in \left(\Gr_{\Gc, \Oc}^{I_\bullet} \hat \times_{\Spec(\Oc_I)} \Spf(\Oc_I)\right)(R')$.
	We say that $\underline{\Ec}$ \emph{is bounded by $\Zc$} if for some (or equivalently for all) representatives $Z_{\underline F'}$ of $\Zc$ over an $I$-tuple of finite extensions $\underline F'$ such that $\Zc$ admits a representative over $\underline F'$, the map
	$$ g \times \id_{\Oc_I'} \colon \Spec(R') \times_{\Spf(\Oc_{\underline{y}})} \Spf(\Oc_{\underline y'}) \to (\Gr_{\Gc, X}^{I_\bullet} \hat \times_{\Spec(\Oc_I)} \Spf(\Oc_{\underline{y}'}))$$ factors through $Z_{\underline F'} \times_{\tilde X_{\underline{F}'}^I} \Spf(\Oc_{\underline{y}'}) $ for all closed points $\underline{y}'$ of $\widetilde{X}^I_{\underline{F}'}$ lying over $\underline{y}$. As $Z_{\underline F'}$ is invariant under the left $\Lc_{\widetilde{X}^I_{\underline{F}'}}^+\Gc$-action, the definition is independent of the choice of the trivialisation $\alpha$.
\end{definition}

\begin{remark}
	In the literature on local shtukas the Frobenius-linear map is often defined to be an isomorphism $\theta' \colon \sigma^*\Ec_0 \cong \Ec_m$, compare for example \cite{Hartl2011, Rad2015}. In particular, for a local shtuka $\underline \Ec$ in our sense inverting all maps yields a local shtuka 
	$$\underline \Ec^{-1} = \left( \Ec_m \overset{\ph^{-1}_m}{\dashleftarrow} \Ec_{m-1} \overset{\ph^{-1}_{m-1}}{\dashleftarrow} \ldots \overset{\ph^{-1}_{1}}{\dashleftarrow} \Ec_0 \overset{\theta^{-1}}{\leftarrow}  \sigma^*\Ec_m  \right)$$
	in the more classical sense (and vice versa).
	Moreover, $\underline{\Ec}$ is bounded by some bound $\Zc$ if and only if $\underline{\Ec}^{-1}$ is bounded by the inverse bound $\Zc^{-1}$ as defined in \cite[Lemma 2.12]{Hartl2017}.
\end{remark}

We explain how to associate to a global shtuka an $I$-tuple of local shtukas. 
At points away from the diagonal, this is contained in \cite{Rad2015}.
We extend their construction to also include points on the diagonal using local shtukas with multiple legs.

We consider the global setting from above. 
Let $X$ be a smooth, projective, and geometrically connected curve over $\Fq$, and let $\Gc \to X$ be a smooth affine group scheme.
We use the notation following \cite[Remark 5.2]{Rad2015}. Let $y \in X$ be a closed point. We denote by $\Oc_y$ the completed local ring at $y$, and by $\mf_y \se \Oc_y$ and $\F_y = \Oc_y/\mf_y$ its maximal ideal with uniformiser $\varpi_y$ and residue field, respectively. 
Moreover, let $d = [\F_y \colon \Fq]$ be the degree of $y$. 
Then the $\mf$-adic completion of $\Oc_y \otimes_\Fq R$ factors as 
$$\Oc_y {\widehat{\otimes}}_\Fq R = (\F_y \otimes_\Fq R)\dsq{\varpi_y} = \prod_{\ell \in \Z/d\Z} \Oc_y {\widehat \otimes}_{\F_y} R = \prod_{\ell \in \Z/d\Z} R\dsq{\varpi_y},$$ 
where the $\ell$-th factor is defined by the ideal $\af_\ell = \langle a \times 1 - 1 \otimes x(a)^{q^\ell} \colon  a \in \kappa_y \rangle$ in $\Oc_y {\widehat \otimes}_\Fq R$. The Frobenius $\sigma$ cyclically permutes the factors.

Let us fix an $I$-tuple of closed points $\underline y = (y_i)_{i \in I}$ of $X$.
We use the notation of Section \ref{secLocGlobBounds} and say $y \in \underline{y}$ if $y = y_i$ for some $i \in I$. Moreover, let  $I^y = \{ i \in I \colon y_i = y\}$ and let $I_\bullet^y = (I_1^y, \ldots, I_m^y)$. 
Let $\underline{\Ec} \in \Sht_{\Gc, \underline{y}}(R)$. 
By the observation above, the $y$-adic completion of $\Ec_j$ decomposes as 
$$ \Ec_j {\widehat{\times}}_{X_R} \Spf(\Oc_{y} {\widehat{\otimes}}_{\Fq} R) = \coprod_{\ell \in \Z/d\Z} \Ec_j {\widehat{\times}}_{X_R} \Spf(R\dsq{\varpi_{y}}) =: \coprod_{\ell \in \Z/d\Z} \Ec_{j,y}^{(\ell)}.$$
Hence, each component $\Ec_{j,y}^{(\ell)} = \Ec_j {\widehat{\times}}_{X_R} V(\af_\ell) $ is a formal $\hat \Gc_{y} = \Gc \times_X \Spf(\Oc_{y})$-torsor over $R$ and corresponds to a $\Lc^+ \Gc_y$-torsor $\widehat{ \Ec_{j,y} }^{(\ell)}$ over $R$ by \cite[Proposition 2.4]{Rad2015}. 

Since $\sigma$ permutes the components cyclically, we find $\widehat{\Ec_{j,y}}^{(\ell)} = (\sigma^{(\ell)})^{\ast} \widehat{\Ec_{j,y}}^{(0)}$.
By \cite[Lemma 5.3]{Rad2015}, the component $V(\af_0)$ is also given by the completion of $X_R$ along $\Gamma_{x_i}$ for any $i \in I_j$. 
Hence, 
$$\ph_j|_{V(\af_\ell)} \colon (\sigma^{(\ell)})^{\ast} \widehat{\Ec_{j,y}}^{(0)} = \widehat{\Ec_{j,y}}^{(\ell)} \xrightarrow{\cong} \widehat{\Ec_{j-1,y}}^{(\ell)} = (\sigma^{(\ell)})^{\ast} \widehat{\Ec_{j-1,y}}^{(0)}$$ 
is an isomorphism for all $\ell \neq 0$. We define the isomorphism  $\tilde \theta \colon  \widehat{\Ec_{0,y}}^{(0)} \xrightarrow{\cong} (\sigma^{(d)})^\ast\widehat{\Ec_{m,y}}^{(0)}$ by 
$$\tilde{\theta} =  \theta|_{V(\af_{d-1})}   \circ \ph_1|_{V(\af_{d-1})} \circ \ldots \circ \ph_m|_{V(\af_{2})} \circ \theta|_{V(\af_1)} \circ \ph_1|_{V(\af_1)} \circ \ldots  \circ \ph_m|_{V(\af_1)} \circ \theta|_{V(\af_0)}.$$
Then, 
$\widehat{\underline{\Ec}_{y}} = \left( \left( \widehat{\Ec_{j,y}}^{(0)} \right)_{0 \leq j \leq m}, \left(\ph_j|_{V(\af_0)} \right)_{1 \leq j \leq m}, \tilde \theta \right) $ is a local $\Gc_{\Oc_{y}}$-shtuka over $R$.
\begin{definition}
	Let $\underline{y}$ be an $I$-tuple of closed points of $X$.
	The \emph{global-to-local functor} associates to a global shtuka $\underline{\Ec} \in  \Sht_{\Gc, \underline{y}}(R)$ a tuple of local $\Gc_{\Oc_y}$-shtukas for $y \in \underline{y}$ given by 
	$ \widehat{\underline\Ec_{\underline{y}}} =  (\widehat{\underline\Ec_{y}})_{y \in \underline{y}}.$
	Then, $\widehat{\underline\Ec_{y}}$ is called the \emph{local shtuka} of $\underline{\Ec}$ at $y$. 
\end{definition}

\begin{remark}
	\label{remEtLocSht}
	In a similar fashion, for a closed point $y$ of $X$ we can associate to a global shtuka $\underline{\Ec} = ((x_i),(\Ec_j), (\ph_j), \theta) \in \Sht_\Gc|_{(X\setminus\{y\})^I}(R)$ with characteristic sections away from $y$ an \'etale local shtuka at $y$ by \cite[Remark 5.6]{Rad2015} as follows.
	We denote by $\widetilde{\Gc}_{y} = \mr{Res}_{\Oc_y/\F_q\dsq{\varpi_y}} \Gc_{\Oc_y}$. Then  $\widetilde{\Gc}_{y}$ is a smooth affine group scheme over $\Fq \dsq{\varpi_y}$.
	The \emph{\'etale local $\widetilde{\Gc}_{y}$-shtuka} associated to $\Ec$ is then given by $\underline{\widetilde{\Ec}}_y = (\widetilde{\Ec}_y, \ph)$ with
	$\widetilde{\Ec}_y = \mr{Res}_{(\Oc_y \widehat{\otimes}_\Fq R)/R\dsq{\varpi_y}}(\Ec_0 \widehat{\times}_{X_R} \Spf(\Oc_y \widehat{\otimes}_\Fq R))$ and $\ph = \theta \circ \ph_1 \circ \ldots \circ \ph_m$. 
	Note that $\ph$ is an isomorphism by assumption.
\end{remark}

The global-to-local functor is compatible with our notion of bounds in the following sense.
Let us fix a global bound $\Zc= \prod_{y \in \underline{y}} \Zc^{(I^y)}$ in $\Gr_{\Gc, X}^{I_\bullet}$ and for each $y \in \underline{y}$ a closed point $y'$ of $\widetilde{X}^{I^y}_{\Zc^{(I^y)}}$ over $\Delta(y) \in X^{I^y}$. 
We denote by $\Sht^{\Zc}_{\Gc,  \underline{y}'} =  \Sht_{\Gc}^{ \Zc} \times_{\widetilde X_\Zc^I} \prod_{y \in \underline{y}} \Spf(\Oc_{y'})$.
\begin{proposition}
	\label{propBoundLocGlobComp}
	Let $\Zc = \prod_{y \in \underline{y}} \Zc^{(I^y)}$ be a global bound and let $(\Zc_{y'})_{y \in \underline{y}}$ be the associated $I$-tuple of local bounds at $\underline{y}$ as in Definition \ref{defITupLocBounds}.
A global shtuka $\underline{\Ec} \in \Sht_{\Gc, \underline{y}'} (R) $ is bounded by $\Zc$ if and only if for all $y \in \underline{y}$ its associated local shtuka $\widehat{\underline{\Ec}_{y}}$ at $y$ is bounded by $\Zc_{y'}$.
\end{proposition} 
\begin{proof}
Let $R' \to R$ be an \'etale cover trivialising $\Ec|_{\hat\Gamma_{\underline x}}$ and let us fix a trivialisation $\alpha \colon \Ec|_{\hat\Gamma_{\underline{x}_{R'}}} \xrightarrow{\cong} \Gc|_{\hat \Gamma_{\underline x_{R'}}}$. 
The pair $(\underline{\Ec}, \alpha)$ defines a point $g \in \left(\Gr_{\Gc, X}^{I_\bullet} \times_{X^I} \widetilde X_{\Zc}^I\right)(R')$ that even lies in $\left(\Gr_{\Gc, X}^{I_\bullet} \times_{X^I} \prod_{y \in \underline{y}} \Spf(\Oc_{y'})\right)(R')$ by assumption on $R$.  
Restricting $\alpha$ yields trivialisations of $\widehat{\underline{\Ec}_y}$ for all $y \in \underline{y}$ and hence points 
$\widehat{g_y} = (\widehat{\underline{\Ec}_y}, \widehat{\alpha_y}) \in \left(\Gr_{\Gc, \Oc_y}^{I^y_\bullet} \times_{\Spec(\Oc_y)^I} \Spf(\Oc_{y'})\right)(R')$.
By Remark \ref{remBLDesc} we find $g  = \prod_{y \in \underline{y}} \widehat{g_y}$. 

Let us fix an $I$-tuple of finite extensions $\underline F'$ of $F$ such that $\Zc$ admits a representative over $\underline F'$. Let us fix for every $y \in \underline{y}$ a closed point $y_{F'}$ of $\widetilde{X}^{I^y}_{\underline F'}$ lying over $y'$.  
By construction, the representative $Z_{\underline{F}'}$ of $\Zc$ factors over $\Spf(\Oc_{\underline{y}_{\underline{F'}}})$ as a product of the corresponding representatives of the local bounds.
The assertion follows. 
	\end{proof}

	\subsection{A Serre-Tate theorem}
	The classical Serre-Tate theorem relates the deformation theory of an abelian variety in characteristic $p$ to the deformation theory of its $p$-divisible group. \cite[Theorem 5.10]{Rad2015} proves an analogue of the Serre-Tate theorem for $\Gc$-shtukas.
	We generalise the Serre-Tate theorem for shtukas to the iterated case using local shtukas with multiple legs.
	The local-global compatibility then also gives a bounded version of the Serre-Tate theorem.
	
	We follow \cite[Section 5]{Rad2015} and start by analysing quasi-isogenies of local and global shtukas.
	\begin{definition}
\begin{enumerate}
	\item 
	Let $\underline{\Ec}$
	and $\underline{\Ec}'$
	be two global shtukas over an $\Fq$-algebra $R$ with the same characteristic sections. 
	A \emph{quasi-isogeny} $f \colon \underline{\Ec} \dashrightarrow \underline{\Ec}'$ is a tuple of isomorphisms $f_j \colon \Ec_j|_{X_R \setminus D_R} \xrightarrow{\cong}  \Ec'_j|_{X_R \setminus D_R}$ for all $0 \leq j \leq m$ away from some effective divisor $D$ of $X$, such that $f_{j-1} \circ \ph_j = \ph'_j \circ f_{j}$ for all $1 \leq j \leq m$ and $f_m \circ \theta = \theta' \circ \sigma^\ast f_0 $.
	\item 
	Let $\underline{\Ec}$ and $\underline{\Ec}'$ be two local shtukas over an $\Fq$-algebra $R$.
	A \emph{quasi-isogeny} $f \colon \underline{\Ec} \dashrightarrow \underline{\Ec}'$ is a tuple of isomorphisms $f_j \colon \Ec_j\left[\frac{1}{\varpi} \right] \xrightarrow{\cong}  \Ec'_j\left[\frac{1}{\varpi} \right]$ of the associated $LG$-torsors for all $0 \leq j \leq m$, such that $f_{j-1} \circ \ph_j = \ph'_j \circ f_{j}$ for all $1 \leq j \leq m$ and $\sigma^*f_m \circ \theta = \theta' \circ  f_0 $.
\end{enumerate}
In either case, we denote by $\mr{QIsog}(\underline{\Ec}, \underline{\Ec}')$ the set of quasi-isogenies between $\underline{\Ec}$ and $\underline{\Ec}'$. 
\end{definition}

\begin{example}
Let $\underline{y}$ be an $I$-tuple of closed points of $X$ and let $\underline{\Ec} \in \Sht_{\Gc, \underline{y}}(R)$ for some $R \in \Nilp_{\Oc_{\underline{y}}}$. Let $D = \sum_{y \in \underline{y}} [y]$ be the corresponding Cartier divisor on $X$. The maps $ (  \sigma^*\ph_j \ldots\circ\sigma^*\ph_m\circ \theta \circ \ph_1 \circ \ldots \ph_{j-1} )|_{X_R \setminus D_R} \colon  \Ec_j|_{X_R \setminus D_R} \xrightarrow{\cong} \sigma^\ast\Ec_j|_{X_R \setminus D_R}$ for $0 \leq j \leq m$ define a quasi-isogeny $\ph \colon\underline{\Ec} \dashrightarrow  \sigma^\ast \underline{\Ec}$. 
\label{exampleQuasiIsogFrob}
\end{example}

We also have the following analogues of \cite[Proposition 5.7, Remark 5.8, Proposition 5.9]{Rad2015} in our setting. Their proofs carry over essentially verbatim. For the convenience of the reader we include sketches of the proofs in our setting.
\begin{proposition}
Let $\underline{y}$ be an $I$-tuple of closed points of $X$, and let $\underline{\Ec} \in \Sht_{\Gc, \underline{y}}(R)$ for some $R \in \Nilp_{\Oc_{\underline{y}}}$.
Let $f_y \colon \underline{\Ec'}_y \dashrightarrow \widehat{ \underline \Ec_{y} }$ be a quasi-isogeny of local shtukas at $y$.
There exists a global shtuka $\underline{\Ec}' \in \Sht_{\Gc,  \underline{y}}(R)$ and a quasi-isogeny $f \colon \underline{\Ec}' \dashrightarrow \underline{\Ec}$ of global shtukas such that $\widehat{ \underline \Ec_{y}' } = \underline \Ec'_{y}$ and $\widehat{f_y} = f_y$. 
\label{propPullQIsog}
\end{proposition}
\begin{proof}
For $0 \leq j \leq m$ we construct a $\Gc$-bundle $\Ec'_j$ on $X_R$ by glueing $\Ec_j|_{(X\setminus \{y\})_R}$ and $\Ec'_{y, j}$ along $f_{y,j}$ using \cite[Lemma 5.1]{Rad2015} together with the argument of \cite[Remark 5.5]{Rad2015}. The map $f_{y, j}$ glues with the identity on $\Ec_j|_{(X\setminus \{y\})_R}$ to an isomorphism $f_j \colon \Ec'_j|_{(X\setminus\{y\})_R} \xrightarrow{\cong} \Ec_j|_{(X\setminus \{y \})_R}$. 
Moreover, the maps $\ph_j$ glue with $\ph'_{y,j}$ to isomorphisms $\ph'_j \colon \Ec_j'|_{X_R \setminus \Gamma_{\underline{x}_j}} \xrightarrow{\cong}\Ec_{j-1}'|_{X_R \setminus \Gamma_{\underline{x}_j}}$ (again using \cite[Remark 5.5]{Rad2015}). Similarly, we obtain $\theta'$. Then $\underline{\Ec}' = ((x_i)_{i \in I}, (\Ec'_j), (\ph'_j), \theta')$ and $f = (f_j) \colon \underline{\Ec}' \dashrightarrow \underline{\Ec}$ clearly satisfies all required properties.
\end{proof}
\begin{proposition}[Rigidity of quasi-isogenies]
Let $R \in \Nilp_{\Oc_{\underline{y}}}$ and let $I \se R$ be a nilpotent ideal with quotient $\bar{R} = R/I$. Let $s \colon \Spec(\bar{R}) \hookrightarrow \Spec(R)$ the corresponding closed immersion on spectra.
Moreover, let $\underline{\Ec}, \underline{\Ec}' \in  \Sht_{\Gc, \underline{y}}(R)$ two global shtukas over $R$ with the same characteristic sections.
Then the pullback along $s$ induces a bijection
$$ \mr{QIsog} (\underline{\Ec}, \underline{\Ec}') \xrightarrow{\cong} \mr{QIsog} (s^\ast\underline{\Ec}, s^\ast\underline{\Ec}'). $$
\label{propRigQIsog}
\end{proposition}
Note that rigidity of quasi-isogenies for local shtukas is shown in \cite[Proposition 2.11]{Rad2015} and \cite[Proposition 2.3]{LiHuerta2023} in the iterated setting.
\begin{proof}
It suffices to treat the case $I^q = 0$, in which case the Frobenius $\sigma \colon \Spec(R) \to \Spec(R)$ factors as $\Spec(R) \xrightarrow{\sigma'} \Spec(\bar{R}) \xrightarrow{s} \Spec(R)$. We can then recover $f_j \colon \Ec_j \dashrightarrow \Ec_j'$ from $f_{j-1}$ (respectively $\sigma^\ast f_m = (\sigma')^\ast s^\ast f_m$) and $\ph_j, \ph_j'$ (respectively $\theta$ and $\theta'$) as in the proof of \cite[Proposition 5.9]{Rad2015}.
\end{proof}

We keep the notation from the previous proposition. Let $\underline{\bar\Ec}$ be a global shtuka over $\bar{R}$.  We denote by $\mr{Defo}_R(\underline{\bar{\Ec}})$ the category of pairs $(\underline{\Ec}, \alpha)$ of a global shtuka $\underline{\Ec}$ over $R$ together with an isomorphism $\alpha \colon s^\ast \underline{\Ec} \xrightarrow{\cong} \underline{\bar\Ec}$. 
Similarly, we define  $\mr{Defo}_R(\underline{\bar{\Ec}})$  for a local shtuka $\underline{\bar{\Ec}}$ over $\bar{R}$. Note that by rigidity of quasi-isogenies the Hom-sets in  $\mr{Defo}(\underline{\bar{\Ec}})$ are trivial (both in the local and in the global setting).
Let $\Zc$ be a global (respectively local) bound such that $\underline{\bar{\Ec}}$ is bounded by $\Zc$. 
Let $\mr{Defo}^\Zc_R(\underline{\bar{\Ec}})$ be the full subcategory of $\mr{Defo}_R(\underline{\bar{\Ec}})$ where $\underline{\Ec}$ is bounded by $\Zc$ as well.
\begin{proposition}[Serre-Tate Theorem for iterated (bounded) shtukas, compare {\cite[Theorem 5.10]{Rad2015}}]
Let $\underline{y}$ be an $I$-tuple of closed points of $X$, let $R \in \Nilp_{\Oc_{\underline{y}}}$ together with a nilpotent ideal $I \se R$ and quotient $\bar{R} = R/I$. Let $\underline{\bar{\Ec}} \in  \Sht_{\Gc, \underline{y}} (\bar{R})$.
The global-to-local functor induces an equivalence of categories
$$\mr{Defo}_R(\underline{\overline{\Ec}}) \to \prod_{y \in \underline{y}} \mr{Defo}_R(\widehat{\underline{\overline{\Ec}}_y})  $$ 
Moreover, let $\Zc = \prod_{y \in \underline{y}} \Zc^{(I^y)}$ be a global bound and for each $y \in \underline{y}$ let $y'$ be a closed point of $\widetilde X^{I^y}_\Zc$ lying over $\Delta(y)$.  
Let $\underline{\overline{\Ec}} \in  \Sht^{\Zc}_{\Gc, \underline{y}'} (\bar{R})$. Then the global-to-local functor induces an equivalence of categories
$$ \mr{Defo}^\Zc_R(\underline{\overline{\Ec}}) \to \prod_{y \in \underline{y}} \mr{Defo}^{\Zc_{y'}}_R(\widehat{\underline{\overline{\Ec}}_y}) .$$
\label{propSerreTate}
\end{proposition}

\begin{proof}
The second statement clearly follows from the first and Proposition \ref{propBoundLocGlobComp}. We prove the first following the proof of \cite[Theorem 5.10]{Rad2015} by constructing an inverse functor.
It suffices to treat the case $I^q = 0$. Then $\sigma$ factors as $\sigma = s \circ \sigma'$.
We set $\underline{\overline{\Ec}}^{\sigma'} = (\sigma')^\ast \underline{\overline{\Ec}}$, hence $s^\ast \underline{\overline{\Ec}}^{\sigma'} = \sigma^\ast \underline{\overline{\Ec}}$ and the quasi-isogeny constructed in Example \ref{exampleQuasiIsogFrob} gives a quasi-isogeny $\ph \colon \underline{\overline\Ec} \dashrightarrow s^\ast \underline{\overline{\Ec}}^{\sigma'}$. 	
Let $(\underline{\Ec}'_y, \alpha'_y)_{y \in \underline{y}}$ be an object of $ \prod_{y \in \underline{y}} \mr{Defo}_R(\widehat{\underline{\overline{\Ec}}_y})$. Then $\widehat{\ph_y} \circ \alpha'_y \colon s^\ast \underline{\Ec}_y' \dashrightarrow s^\ast \widehat{\underline{\overline{\Ec}}^{\sigma'}_y}$ is a quasi-isogeny for all $y \in \underline{y}$ and lifts by rigidity of quasi-isogenies for local shtukas \cite[Proposition 2.3]{LiHuerta2023} to a quasi-isogeny $\tilde{f}_y \colon  \underline{\Ec}_y' \dashrightarrow \widehat{\underline{\overline{\Ec}}^{\sigma'}_y}$.
We define the global shtuka $\underline{\Ec}$ over $R$ as the pullback of $\underline{\overline{\Ec}}^{\sigma'}$ along the $\tilde{f}_y$ as constructed in Proposition \ref{propPullQIsog}, it comes equipped with a quasi-isogeny $f\colon \underline \Ec \to \underline{\overline{\Ec}}^{\sigma'}$ which is an isomorphism outside the characteristic sections.
Moreover, we consider the quasi-isogeny $\alpha = \ph^{-1} \circ s^*f \colon s^* \underline{\Ec} \to \underline{\overline{\Ec}}$. It is an isomorphism outside the characteristic sections by construction and at the characteristic sections we find $\widehat{\alpha_y} = \widehat{\ph_y}^{-1} \circ \widehat{s^*f_y} = \alpha_y'$. Hence, $\alpha$ is an isomorphism and $(\underline{\Ec}, \alpha) \in  \mr{Defo}_R(\underline{\overline{\Ec}})$.
As in \cite[Theorem 5.10]{Rad2015} one checks that the constructions are indeed inverse functors.
\end{proof}

\subsection{Local model theorems}

Local model theorems for moduli spaces of shtukas were shown under varying hypotheses in \cite[Theorem 2.20]{Varshavsky2004} (for $\Gc$ constant split reductive), \cite[Proposition 2.8]{Lafforgue2018}, and \cite[Theorem 3.2.1]{Rad2017} (for smooth affine $\Gc$), compare also \cite[Theorem 4.19]{Feng2020} for a different argument. We slightly generalise their results.
While the results are certainly known to the experts, they do not seem to be contained in the literature in this generality. 

Let $\Bun_{\Gc, \hat{\Gamma}^I}$ be the stack such that for a $\Fq$-algebra $R$ an $R$ point is given by the data $(\underline{x}, \Ec, \alpha)$, where $\underline{x} \in X^I(R)$ is an $I$-tuple of sections of $X$, $\Ec \in \Bun_\Gc(R)$ a $\Gc$-bundle on $X_R$, and $\alpha \colon \Ec|_{\hat{\Gamma}_{\underline{x}}} \xrightarrow{\cong} \Gc$ is a trivialisation of $\Ec$ on $\hat{\Gamma}_{\underline{x}}$. 
For a positive integer $r$ the stack $\Bun_{\Gc, r\Gamma^I}$ is defined to parametrise the same data as $\Bun_{\Gc, \hat{\Gamma}^I}$, but where $\alpha$ is a trivialisation of $\Ec$ on $\Gamma^{(r)}_{\underline{x}}$ where $\Gamma^{(r)}_{\underline{x}}$ is the $r$-th infinitesimal neighbourhood of the union of the graphs $\Gamma_{x_i}$ as introduced in Section \ref{subsecBDGrass}.

Then, $\Bun_{\Gc, \hat{\Gamma}^I}$ admits a $\Lc^+_{X^I} \Gc$-action (given by the operation on $\alpha$), and the forgetful map $\Bun_{\Gc, \hat{\Gamma}^I} \to X^I \times \Bun_\Gc$ is a $\Lc^+_{X^I} \Gc$-torsor.
Similarly, $\Bun_{\Gc, r\Gamma^I} \to X^I \times \Bun_\Gc$ is a $\Lc^{(r)}_{X^I}\Gc$-torsor. 

We also need the following versions of both the Hecke stack and the moduli space of shtukas.
We define $\widetilde{\Hecke}_\Gc = \widetilde{\Hecke}_{\Gc, X}^{I_\bullet} = \Hecke_\Gc \times_{X^I \times \Bun_{\Gc}} \Bun_{\Gc, \hat{\Gamma}^I}$, it parametrises the same data as $\Hecke_{\Gc}$ together with a trivialisation of $\Ec_0$ over $\hat{\Gamma}_{\underline{x}}$.
Let $\Zc$ be a bound such that the $\Lc^+_{X^I} \Gc$-action on $\Zc$ factors through $\Lc^{(r)}_{X^I} \Gc$. In this situation, let $\widetilde{\Hecke}^{(r), \Zc}_{\Gc}= \widetilde{\Hecke}^{(r), I_\bullet, \Zc}_{\Gc, X} = \Hecke^\Zc_\Gc \times_{\widetilde X_\Zc^I \times \Bun_\Gc} \widetilde X_\Zc^I \times \Bun_{\Gc, r\Gamma^I}$. 
In a similar fashion we define $\widetilde{\Sht}_{\Gc}$ and $\widetilde{\Sht}_\Gc^{(r), \Zc}$. 
\begin{lemma}
The map 
\begin{align*}
	\widetilde{\Hecke}_{\Gc}  & \xrightarrow{\cong} \Gr_{\Gc} \times_{X^I} \Bun_{\Gc, \hat \Gamma_I} \\
	\left( \Ec_m \overset{\ph_m}{\dashrightarrow} \Ec_{m-1} \overset{\ph_{m-1}}{\dashrightarrow} \ldots \overset{\ph_{1}}{\dashrightarrow} \Ec_0 \overset{\alpha}{\rightarrow}  \Gc  \right) & \mapsto \left( \Ec_m|_{\hat{\Gamma}_{\underline{x}}} \overset{\ph_m}{\dashrightarrow} \Ec_{m-1}|_{\hat{\Gamma}_{\underline{x}}} \overset{\ph_{m-1}}{\dashrightarrow} \ldots \overset{\ph_{1}}{\dashrightarrow} \Ec_0|_{\hat{\Gamma}_{\underline{x}}} \overset{\alpha}{\rightarrow}  \Gc  \right), (\underline{x}, \Ec_0, \alpha )
\end{align*}
is an isomorphism and for a bound $\Zc$ restricts to an isomorphism 
$$ \widetilde\Hecke_\Gc^{(r), \Zc} \times_{\widetilde X_\Zc^I} \widetilde{X}^I_{\underline F'} \xrightarrow{\cong} Z_{\underline F'} \times_{\widetilde{X}_{\underline F'}^I} \left( \Bun_{\Gc, r \Gamma_I} \times_{X^I} \widetilde{X}^I_{\underline F'} \right),$$
where $\underline F'$ is an $I$-tuple of finite extensions of $\underline F$ such that $\Zc$ admits a representative $Z_{\underline F'}$ over $\underline F'$.
\label{lemLocModHecke}
\end{lemma}
Note that in \cite{Lafforgue2018} the former isomorphism is used to define the bounded version of the Hecke stack. 
\begin{proof}
The inverse to the map is constructed by Beauville-Laszlo glueing. The bounded version follows directly from the definition of $\Hecke_\Gc^\Zc$. 	
\end{proof}

We get the following generalisation of \cite[Proposition 2.8]{Lafforgue2018} (and \cite[Proposition 2.9]{Lafforgue2018}). Let $\Zc$ be a bound such that the $\Lc^+_{X^I} \Gc$-action factors through $\Lc^{(r)}_{X^I} \Gc$ with representative $Z_{\underline F'}$ defined over $\underline F'$.
Then by construction we get a local model roof
\begin{center}
\begin{tikzcd}
	& \widetilde{\Sht}^{(r), \Zc}_{\Gc} \arrow[dr] \arrow[dl] \times_{\widetilde X_\Zc^I} \widetilde{X}^I_{\underline F'} & \\
	\Sht^\Zc_\Gc \times_{\widetilde X_\Zc^I} \widetilde{X}^I_{\underline F'} & & Z_{\underline F'}, 
\end{tikzcd}
\end{center}
where the maps are given by the obvious projections.
\begin{proposition}
Let $\Zc$ be a bound and $\underline F'$ an $I$-tuple of finite extensions of $\underline F$ such that $\Zc$ admits a representative over $\underline F'$.
The map $\Sht_{\Gc}^\Zc \times_{\widetilde X_\Zc^I} \widetilde{X}^I_{\underline F'} \to \left[ \Lc^{(r)}_{X^I}\Gc \backslash Z_{\underline F'} \right]$ is a smooth map of Artin stacks.

Moreover, assume that $\Zc$ arises as from bounds $\Zc^{(I_j)}$ defined in $\Gr_{\Gc, X}^{(I_j)}$ via Construction \ref{consBoundSingleLeg}. 
Then the induced map
$\Sht_{\Gc}^\Zc \times_{\widetilde X_\Zc^I} \widetilde{X}^I_{\underline F'} \to \prod_{j = 1}^m \left[ \Lc^{(r)}_{X^{I_j}}\Gc \backslash Z^{(I_j)}_{\underline F'} \right]$ is a smooth map of Artin stacks.
\label{thmLocModThm}
\end{proposition} 
\begin{proof}
The proofs of \cite[Proposition 2.8 and 2.9]{Lafforgue2018} still apply to our situation.
\end{proof}

We also get an \'etale-local version of the local model theorem, generalising \cite[Theorem 3.2.1]{Rad2017}, \cite[Proposition 2.11]{Lafforgue2018} and \cite[Theorem 2.20]{Varshavsky2004}.
\begin{proposition}
Let $\Zc$ be a bound and $\underline F'$ a finite extension of $F$ such that $\Zc$ admits a representative over $\underline F'$. Then $\Sht_\Gc^\Zc \times_{\widetilde X_\Zc^I} \widetilde{X}^I_{\underline F'}$ and $Z_{\underline F'}$ are \'etale-locally isomorphic.
More precisely, there is an \'etale cover $U \to \Sht_{\Gc}^\Zc \times_{\widetilde X_\Zc^I} \widetilde{X}^I_{\underline F'}$ admitting an \'etale map $U \to Z_{\underline F'}$ over $\widetilde{X}^I_{\underline F'}$. 

Moreover, assume that $\Zc$ arises as from bounds $\Zc^{(I_j)}$ defined in $\Gr_{\Gc, X}^{(I_j)}$ via Construction \ref{consBoundSingleLeg}. Then $\Sht_\Gc^\Zc \times_{\widetilde X_\Zc^I} \widetilde{X}^I_{\underline F'}$ is \'etale-locally isomorphic to $\prod_{j = 1}^m Z^{(I_j)}_{\underline F'}$.

\end{proposition}

\begin{proof}

We sketch how to adapt the argument of \cite[Proposition 2.11]{Lafforgue2018}. 
By the usual reductions (notably bounding the Harder-Narasimhan slopes appearing and introducing an auxiliary level structure, compare also \cite[Remark 2.9]{Rad2019a} for the generalisation to arbitrary smooth affine group schemes $\Gc$), for example as at the beginning of the proof of \cite[Proposition 2.11]{Lafforgue2018}, we may treat $\Bun_\Gc$ (and thus $\Hecke_{\Gc}^\Zc$) as a scheme of finite type over $\Fq$.

For every $j$ we denote by $(\Ec^{\mr{univ}}, (x_i)_{i \in I})$ the universal object of $\Bun_\Gc \times X^{I}$. Let $U \to \Bun_\Gc \times X^{I}$ be an \'etale cover trivialising $\Ec^{\mr{univ}}$ over $\hat{\Gamma}_{\underline{x}}$.
By \cite[Proposition 3.9]{Rad2019a}, the map $\Hecke_{\Gc, X}^{I_\bullet, \Zc} \to \widetilde X_\Zc^I \times  \Bun_\Gc$ is schematic and of finite type. Hence, the base change $W = \Hecke_{\Gc, X}^{I_\bullet, \Zc} \times_{\widetilde X_\Zc^I \times \prod_j \Bun_\Gc} \widetilde X_{\Zc}^{I} \times_{X^{I}} U$ is a scheme of finite type over $\Fq$, and the projection $W \to \Hecke_{\Gc, X}^{I_\bullet, \Zc}$ is \'etale.
By construction of $U$, the scheme $W$ admits a natural map $W \to \Bun_\Gc \times  \Gr_{\Gc, X}^{I_\bullet}$, that factors through $\Bun_\Gc \times Z_{\underline F'}$ and is \'etale by Lemma \ref{lemLocModHecke}.

The first statement for now follows from \cite[Lemme 2.13]{Lafforgue2018}.	
As in \cite[Proposition 2.11]{Lafforgue2018}, the second statement can be shown inductively using a similar argument.
\end{proof}

\subsection{Functoriality results for moduli spaces of shtukas}
We study functoriality properties of moduli spaces of shtukas under homomorphisms of group schemes. 
As the main result in this section we show that for a generic isomorphism $\Gc \to \Gc'$ the induced map on the moduli stacks of (iterated, bounded) shtukas is schematic, separated and of finite type. Moreover, when $\Gc$ is parahoric and the bound is generically defined, the map is proper and surjective.
The result is an extension of \cite[Theorem 3.20]{Breutmann2019} to our setting.
In particular, the use of global bounds allows us to work on the whole curve $X$ and we need not restrict the legs to a formal neighbourhood of a fixed point as in \cite{Breutmann2019}.
Moreover, we show that the level maps are generically finite \'etale, which is not part of \cite{Breutmann2019}.

In our setting, the notion of a shtuka datum (respectively a map of shtuka data) in the sense of \cite[Definitions 3.1 and 3.9]{Breutmann2019} restricts to the following.
\begin{definition}
\label{defnShtDat}
A \emph{shtuka datum} $(\Gc, \Zc)$ is a pair of a smooth affine group scheme $\Gc \to X$ and a bound $\Zc$ in $\Gr_{\Gc, X}^{I_\bullet}$.
Let $(\Gc, \Zc)$ and $(\Gc', \Zc')$ be two shtuka data such that for every $i \in I$ the $i$-th components of $\Zc$ and $\Zc'$ have the same type.
A \emph{map of shtuka data} $f \colon (\Gc, \Zc) \to (\Gc', \Zc')$ is a map of group schemes $f \colon \Gc \to \Gc'$ such that for all $I$-tuples of finite extensions $\underline F'$ such that both $\Zc$ and $\Zc'$ admit representatives over $\underline F'$ the map 
$$ Z_{\underline F'} \hookrightarrow \Gr_{\Gc, X}^{I_\bullet} \times_{X^I}  \widetilde X^I_{\underline F'} \xrightarrow{f_\ast}  \Gr_{\Gc', X}^{I_\bullet} \times_{X^I}  \widetilde X^I_{\underline F'}$$ factors through $Z'_{\underline F'}$.
\end{definition}
In analogy to the construction of the reflex scheme we set $\widetilde{X}^I_{\Zc.\Zc'} = \widetilde{X}^I_{\underline{F}'}/\Aut_{\Zc.\Zc'}(\underline{F}')$ where $\underline{F}'$ is a componentwise Galois extension such that both $\Zc$ and $\Zc'$ admit representatives over $\underline{F}'$ and  
$$\Aut_{\Zc.\Zc'}(\underline{F}') =  \{(g_i)_{i \in I} \in \Gal(F_i'/F) \colon  (g_i)_{i \in I}^\ast(\Zc)=\Zc, \ (g_i)_{i \in I}^\ast(\Zc')=\Zc' \}.$$ 
Under mild hypotheses, a map of shtuka data $f \colon (\Gc, \Zc) \to (\Gc', \Zc')$ induces a map on the corresponding moduli stacks of shtukas
$$ f_\ast \colon \Sht^{I_\bullet, \Zc}_{\Gc, X} \times_{\widetilde X^I_\Zc} \widetilde X^I_{\Zc.\Zc'} \to \Sht^{I_\bullet, \Zc'}_{\Gc', X} \times_{\widetilde X^I_{\Zc'}} \widetilde X^I_{\Zc.\Zc'}$$
by the following lemma that is an analogue of \cite[Lemma 3.15]{Breutmann2019}.
\begin{lemma}
\label{lemChangeBound}
Let $f \colon (\Gc, \Zc) \to (\Gc', \Zc')$ be a map of shtuka data. Assume that $\Zc'$ admits a representative over its reflex scheme.
Let 
$$\underline{\Ec} \in (\Sht^{\Zc}_{\Gc} \times_{X^I} \widetilde X^I_{\Zc. \Zc'})(R)$$
for some $\Fq$-algebra $R$. 
Then $f_\ast \underline{\Ec} \in (\Sht_{\Gc'} \times_{X^I} \widetilde X^I_{\Zc. \Zc'})(R)$ is bounded by $\Zc'$.
\end{lemma}
\begin{proof} 
Let $\underline{\Ec} = ((x_i)_{i \in I}, (\Ec_j)_{j=0, \ldots,m}, (\ph_j)_{j = 1, \ldots, m}, \theta) \in (\Sht^{\Zc}_{\Gc} \times_{X^I} X_{\Zc.\Zc'}^I)(R)$. Let $R \to R'$ be an \'etale cover that trivialises $\Ec_m|_{\hat{\Gamma}_{\underline{x}}}$ and choose a trivialisation $\alpha \colon \Ec_m|_{\hat{\Gamma}_{\underline{x}_{R'}}} \xrightarrow{\cong} \Gc|_{\hat{\Gamma}_{\underline{x}_{R'}}}$. Then $(\underline{\Ec}_{R'}, \alpha)$ defines an $R'$-valued point in $\Gr_{\Gc, X}^{I_\bullet} \times_{X^I} X_{\Zc.\Zc'}^I$.
Let $\underline F'$ be an $I$-tuple of finite extensions of $\underline F$ such that both $\Zc$ and $\Zc'$ admit representatives over $\underline F'$.
As $\underline{\Ec}$ is bounded by $\Zc$,  the induced point $\Spec(R') \times_{\widetilde X_\Zc^I} {\widetilde X_{F'}^I} \to \Gr_{\Gc, X}^{I_\bullet} \times_{X^I} {\widetilde X_{F'}^I} $ factors through $Z_{\underline F'}$, hence its image under $f_\ast$ factors through $Z'_{\underline F'}$ by assumption. 
By Lemma \ref{lemSectBoundRefl} 
the map $ \Spec(R') \times_{\widetilde X_{\Zc'}^I} \widetilde X_{\underline F'}^I \to \Gr_{\Gc', X}^{I_\bullet} \times_{X^I} \widetilde X_{\underline F'}^I$ factors through $Z'_{\underline F'}$, too.
\end{proof}

\begin{remark}
\label{remProbFunct}
Note that we used the assumption on the bounds in the last step of the proof. We do not know how to prove the lemma without this assumption, neither in our setting nor in the setting of local boundedness conditions of \cite{Breutmann2019}.
In this sense \cite[Lemma 3.15]{Breutmann2019} which does not make a similar kind of assumption seems to be problematic in general.	
\end{remark}

We need the following lemma on twisted flag varieties in the local setting.
\begin{lemma}
\label{lemTwistedFlagVariety}
Let $k \dbr{t}$ be the field of formal Laurent series over an arbitrary field $k$ and let $\Oc = k \dsq{t}$ the subring of formal power series.
Let $G$ be a smooth affine group scheme over $k \dbr{t}$ and let $\Gc$ and $\Gc'$ be two smooth integral models of $G$ with geometrically connected fibres. 
Let $f \colon \Gc \to \Gc'$ be a homomorphism of $\Oc$-group schemes that is the identity on $G$ over $k$. Then the induced map on loop groups $L^+\Gc \to L^+\Gc'$ is a closed immersion.
\begin{enumerate}
	\item The corresponding twisted flag variety $L^+\Gc'/L^+\Gc$ is representable by a 
	separated scheme of finite type over $k$. 
	If $k$ is finite or separably closed, then 
	$$\left( L^+\Gc'/L^+\Gc \right)(k) = \Gc'(\Oc)/\Gc(\Oc).$$
	\item 
	Assume that $k$ is finite. We equip $G(k \dbr{t})$ with the analytic topology induced by the natural topology on $k \dbr{t}$ (note that $k \dbr{t}$ is locally compact in this case). 
	Then $\Gc(\Oc) \se \Gc'(\Oc)$ are compact open subgroups of $G(k \dbr{t})$.
	In particular, the quotient $\Gc'(\Oc)/\Gc(\Oc)$ is discrete and finite.
	\item Let $S$ be an $k$-scheme. Giving a $L^+ \Gc$-torsor over $S$ is equivalent to giving a $L^+ \Gc'$-torsor $\Ec'$ over $S$ together with an isomorphism $\Ec'/L^+\Gc \xrightarrow{\cong} L^+ \Gc'/L^+ \Gc$. 
\end{enumerate}
\end{lemma}
Note that giving an isomorphism $\Ec'/L^+\Gc \xrightarrow{\cong} L^+ \Gc'/L^+ \Gc$ in (3) is also clearly equivalent to giving a section in $\left(\Ec'/L^+\Gc  \right)(S)$.
\begin{proof}
\begin{enumerate}
	\item 
	By the argument in the proof of \cite[Lemma 3.17]{Breutmann2019}, the quotient stack $L^+\Gc'/L^+\Gc$ is representable by a separated scheme of finite type over $k$ that is moreover a closed subscheme of the affine Grassmannian $\Gr_\Gc$. 
	For the second claim, it suffices to show that $H^{1}(k, L^+\Gc)$ is trivial by the moduli description of the quotient stack.
	But this is shown in the proof of \cite[Corollary 3.22]{Richarz2019}.
	\item 
	Clearly, both $\Gc(\Oc)$ and $\Gc'(\Oc)$ are compact open subgroups of $G(k \dbr{t})$ by construction. The existence of the map $f$ then means that   $\Gc(\Oc)$ is a subgroup of $\Gc'(\Oc)$.
	The assertion on the quotient then directly follows from basic facts from topology.
	\item 
	Given a $L^+ \Gc$-torsor $\Ec$ on $S$, its associated $L^+ \Gc'$-torsor is given by $\Ec \times^{L^+ \Gc} L^+\Gc'$.
	The map on sections given by $(e,g) \mapsto g$ then induces an isomorphism 
	$\Ec'/ L^+ \Gc \xrightarrow{\cong} L^+ \Gc'/ L^+ \Gc.$
	This construction is an equivalence.
\end{enumerate}
\end{proof}

We are now in a position to prove the main result in this section, an extension of the functoriality result of \cite[Theorem 3.20]{Breutmann2019} to moduli spaces of shtukas with global and generically bounds.
\begin{theorem}
\label{thmLvlMapGen}
Let $\Gc$ and $\Gc'$ be two smooth affine group schemes over $X$ with geometrically connected fibres.
Let $f \colon (\Gc, \Zc) \to (\Gc', \Zc')$ be a map of shtuka data such that the map $f \colon \Gc \to \Gc'$ is an isomorphism over $U = X \setminus \left\{y_1, \ldots, y_n\right\}$ for a finite set of closed points $ \left\{y_1, \ldots, y_n\right\}$ of $X$. Assume that $\Zc'$ admits a representative over its reflex scheme. 
\begin{enumerate}
	\item The induced map 
	$$f_\ast \colon \Sht^{ I_\bullet, \Zc}_{\Gc,X} \times_{\widetilde X^I_\Zc} \widetilde X^I_{\Zc.\Zc'} \to \Sht^{I_\bullet,\Zc'}_{\Gc',X} \times_{\widetilde X^I_{\Zc'}} \widetilde X^I_{\Zc.\Zc'}$$
	is schematic, separated and of finite type.
	\label{thmLvlMapGenRep}
	\item Assume that $\Gc$ is a parahoric Bruhat-Tits group scheme. Then the map $f_\ast$ is proper.
	\item Assume that $Z_{\underline F'}  \to Z'_{\underline F'} $ is an isomorphism over $U^I  \times_{X^I} \widetilde X^I_{\underline F'}$ for some $\underline F'$. Then over $U^I  \times_{X^I} \widetilde X^I_{\Zc.\Zc'}$ the map $f_\ast$ is \'etale locally representable by the constant scheme
	$$ \prod_{i = 1}^n \underline{\Gc'(\Oc_{y_i})/\Gc(\Oc_{y_i})}.$$
	In particular, $f_\ast$ is
	finite \'etale and surjective over $U^I  \times_{X^I} \widetilde X^I_{\Zc.\Zc'}$.
	\item Under the assumptions of (2) and (3) assume additionally that $\Zc'$ is generically defined. Then $f_\ast$ is surjective.
\end{enumerate}
\end{theorem}

\begin{remark}
The first two statements are direct analogues of the corresponding statements in \cite[Theorem 3.20]{Breutmann2019}, while there is no analogue of the third assertion in \cite[Theorem 3.20]{Breutmann2019}.
In order to get surjectivity of the map $f_\ast$, in \cite{Breutmann2019} it is assumed that the bound $\Zc$ arises as the base change of $\Zc'$ under the map $f_\ast$ on affine Grassmannians. 
This assumption does not seem adequate in our setting, in particular, it is not satisfied for the bounds given by Schubert varieties.

We thus replace this assumption by the condition that the bounds are generically defined and that the map is a generic isomorphism, both of which are satisfied in the examples of interest.
Note that when $\Zc$ arises as a base change, the map $\Zc \to \Zc'$ is clearly an isomorphism over $U^I$.
\end{remark}

\begin{proof}
\begin{enumerate}
	\item 
	We proceed as in the proof of \cite[Theorem 3.20]{Breutmann2019}.
	We consider the projection $\Sht_{\Gc}^{\Zc} \to \prod_{j=1, \ldots, m} \Bun_\Gc$ given by $\underline{\Ec} \mapsto (\Ec_j)_{j=1, \ldots, m}$.
	Let us fix 
	$$\underline{\Ec}' = ((x_i)_{i \in I}, (\Ec'_j)_{j = 0, \ldots, m}, (\ph_j)_{j = 1, \ldots, m}, \theta) \in \left( \Sht^{\Zc'}_{\Gc'} \times_{\widetilde X^I_{\Zc'}} \widetilde X^I_{\Zc. \Zc'}\right) (S).$$
	We claim that the induced map 
	$$  S \times_{\left( \Sht^{\Zc'}_{\Gc'}  \times_{\widetilde X^I_{\Zc'}} \widetilde X^I_{\Zc. \Zc'} \right)} \left( 
	\Sht^{\Zc}_{\Gc}  \times_{\widetilde X^I_{\Zc}} \widetilde X^I_{\Zc. \Zc'} \right)\to S \times_{\prod_{j = 1}^m \Bun_{\Gc'}} \prod_{j = 1}^m \Bun_{\Gc} $$ 
	is a quasi-compact locally closed immersion. This shows the assertion (\ref{thmLvlMapGenRep}) using that $\Bun_\Gc \to \Bun_{\Gc'}$ is schematic and quasi-projective by \cite[Proposition 3.18]{Breutmann2019}.
	
	In order to show the claim, let us fix a point 
	$$(s, (\Ec_j)_{j = 1, \ldots m}, (\psi_j)_{j = 1, \ldots, m}) \in (S \times_{\prod_{j = 1}^m \Bun_{\Gc'}} \prod_{j = 1}^m \Bun_{\Gc})(T),$$ where $s \colon T \to S$ is a map of schemes, the $\Ec_j$ are $\Gc$-bundles and $\psi_j \colon s^\ast \Ec'_j \xrightarrow{\cong} f_\ast \Ec_j$ is an isomorphism of $\Gc'$-bundles over $X_T$. 
	As in the proof of \cite[Theorem 3.20]{Breutmann2019}, there is at most one $T$-valued point $(s, \underline{\Ec}, \psi)$ of $S \times_{\left( \Sht^{\Zc'}_{\Gc'}  \times_{\widetilde X^I_{\Zc'}} \widetilde X^I_{\Zc. \Zc'} \right)} \left( 
	\Sht^{\Zc}_{\Gc}  \times_{\widetilde X^I_{\Zc}} \widetilde X^I_{\Zc. \Zc'} \right)$ mapping to $(s, (\Ec_j)_{j = 1, \ldots m}, (\psi_j)_{j = 1, \ldots, m})$ as the maps $\ph_j$ of $\underline{\Ec}$ are already uniquely determined over an open dense subset by the $\ph_j'$. 
	
	It remains to check that the locus where such an extension exists is closed in $T$.
	Let $D = X \setminus U$ be the effective Cartier divisor in $X$ given by $\underline{y}$. Let $1 \leq j \leq m$. The map $\ph_{j, T}' \colon \Ec'_{j}|_{X_T \setminus  \Gamma_{\underline{x}_j}} \to {\Ec'_{j-1}}|_{X_T \setminus  \Gamma_{\underline{x}_j}} $ defines a map  
	$\ph_{j} \colon \Ec_{j}|_{X_T \setminus (D \cup \Gamma_{\underline{x}_j}) } \to {\Ec_{j-1}}|_{X_T \setminus (D \cup  \Gamma_{\underline{x}_j})}$. 
	We may work \'etale-locally on $T$ and assume that $T = \Spec(R)$ is affine and that both $\Ec_{j-1}$ and $\Ec_j$ are trivial over $\widehat{ D_R \cup \Gamma_{\underline{x}_j}}$.
	After fixing a trivialisation for both $\Ec_{j-1}$ and $\Ec_j$, the map $\ph_j$ defines an element $\ph_j \in \Gc(\widehat{ D_R \cup \Gamma_{\underline{x}_j}}^\circ)$. 
	By the argument that the positive loop group is a closed subscheme of the loop group, the locus where $\ph_j$ can be extended to $\hat{D}_R \setminus \Gamma_{\underline{x}_j}$ is closed.
	Finally, the locus where $\underline{\Ec}$ is bounded by $\Zc$ is representable by a closed immersion. 
	\item 
	This follows from the argument in (\ref{thmLvlMapGenRep}) as in the parahoric case the map $\Bun_\Gc \to \Bun_{\Gc'}$ is projective by \cite[Proposition 3.18]{Breutmann2019}.
	\item 		
	It suffices to show the first claim, namely that the map $f_\ast$ is \'etale locally representable by the constant scheme $ \prod_{\ell = 1}^n \underline{\Gc'(\Oc_{y_\ell})/\Gc(\Oc_{y_\ell})}.$
	We follow the proof of \cite[Proposition 2.16]{Varshavsky2004}. Let $$\underline{\Ec'} = ((x_i), (\Ec'_i), (\ph'_i), \theta) \in \Sht^{ \Zc'}_{\Gc'}|_{U_{\Zc, \Zc'}^I}(S).$$
	For $\ell = 1, \ldots, n$, we denote by $\widetilde{\underline{\Ec'}_{y_k}} = (\widetilde{\Ec'}_{y_\ell}, \ph)$ the associated \'etale local shtuka of $\underline{\Ec'}$ at $y_\ell$ as defined in Remark \ref{remEtLocSht}. 
	The fibre product 
	$$S' =  S \times_{\underline{\Ec}',  \Sht^{\Zc'}_{\Gc'}|_{U_{\Zc, \Zc'}^I}, f_\ast}  \Sht^{\Zc}_{\Gc}|_{U_{\Zc, \Zc'}^I}$$
	is then given by the set of tuples $(\widetilde{\underline{\Ec'}_{y_\ell}})_{\ell = 1, \ldots, n}$ of \'etale local $\widetilde{\Gc_{\Oc_{y_\ell}}}$-shtukas such that $f_\ast \widetilde{\underline{\Ec'}_{y_\ell}}  = \widetilde{\underline{\Ec'}_{y_\ell}}$.
	As the claim is \'etale-local on $S$, we may assume that all $\widetilde{\Ec'}_{y_\ell}$ are trivial $L^+ \widetilde{\Gc'_{\Oc_{y_\ell}}}$-torsors.
	By Lemma \ref{lemTwistedFlagVariety} (3), the fibre product $S'$ is then representable by the scheme of Frobenius fixed points of $\prod_{\ell  = 1}^n \widetilde{L^+\Gc'_{\Oc_{y_\ell}}} / L^+ \widetilde{\Gc_{\Oc_{y_\ell}}}$, which is given by the constant scheme $\prod_{\ell  = 1}^n \underline{ \left(L^+ \widetilde{\Gc'_{\Oc_{y_\ell}}}  / L^+ \widetilde{\Gc_{\Oc_{y_\ell}}}  \right)(\Fq)}  $ by \cite[Lemma 3.3]{Varshavsky2004}. By Lemma \ref{lemTwistedFlagVariety} (1), this scheme can be identified with  $\prod_{\ell = 1}^n \underline{\Gc'(\Oc_{y_\ell})/\Gc(\Oc_{y_\ell})}$, and by Lemma \ref{lemTwistedFlagVariety} (2) it is finite over $\Fq$ using that the residue field of $\Oc_{y_\ell}$ is finite.
	\item 
	Let us fix a point $s \in \Sht^{\Zc'}_{\Gc'}$. If $s$ lies over $U$, it is in the image of $\Sht^{\Zc'}_{\Gc'}$ by (3). 
	Let us thus assume that $s$ maps to $X^I \setminus U$. 
	By the local model theorem (compare Proposition \ref{thmLocModThm}), we have a smooth map $\Sht^{\Zc'}_{\Gc'} \to [\Lc^{(r)}_{X^I}\Gc \backslash Z']$, where $Z'$ is the representative of $\Zc'$ over its reflex scheme. 
	As $\Zc'$ is generically defined, the image of $s$ in $[\Lc^{(r)}_{X^I} \Gc \backslash Z']$ has a generalisation $s'$ over $U$. 
	As the local model map is smooth, $s'$ lifts to a generalisation $s''$ of $s$ in $\Sht^{\Zc'}_{\Gc'}$.
	As $f_\ast$ is generically surjective by (3), there is a point $t \in \Sht_{\Gc}^{\Zc}$ mapping to $s''$. 
	As $f_\ast$ is proper by (2), specialisations lift along $f_\ast$. Hence, $s$ is in the image of $f_\ast$. 
\end{enumerate}

\end{proof}

We will usually apply the theorem in the following setting, which in particular applies to bounds defined by Schubert varieties.
\begin{corollary}
\label{corLvlMapParahoric}
Let $G$ be a reductive group over $F$, let $\Zc$ be a generically defined bound for $G$ (in the sense of Remark \ref{remBdGenFibre}) and let $f \colon \Gc \to \Gc'$ be a map of two smooth affine models of $G$ over $X$ that is an isomorphism over some dense open subset $U$ of $X$. The induced map
$$f_\ast \colon \Sht^{\Zc}_{\Gc} \to \Sht^{\Zc}_{\Gc'}$$
is schematic, separated and of finite type. Moreover, it is finite \'etale and surjective over $U^I \times_{X^I} \widetilde X_{\Zc}^I$.
When $\Gc$ is a parahoric Bruhat-Tits group scheme,  $f_\ast$ is proper and surjective.
\end{corollary} 
\begin{proof}
The generically defined bound $\Zc$ for $\Gc$ and $\Gc'$ clearly satisfies the conditions of Theorem \ref{thmLvlMapGen}.
\end{proof}


\section{Torsors under Bruhat-Tits group schemes}
\label{sectBT}
We show that a Bruhat-Tits group scheme is the limit of all corresponding parahoric group schemes and use this observation to show that the induced map on the level of $\Bun_\Gc$ is an open immersion.
We first discuss (pseudo-)torsors for limits of groups.

\subsection{Pseudo-torsors for limits of groups}
We use the following result on pseudo-torsors under limits of groups.
For a sheaf of groups $\underline{G}$ on a site $\Cc$ we denote by $\mr{PTor}_{\underline G}$ the category of \emph{$\underline G$-pseudo-torsors} 
for $\underline G$ with $\underline G$-equivariant maps. 
In other words, an object of $\mr{PTor}_{\underline G}$ is given by a sheaf $E$ on $\Cc$ together with a (right) action $E \times \underline G \to E$ of $\underline G$ such that the induced map $E \times \underline G \to E \times E$ given by $(e,g) \mapsto (e, eg)$ is an isomorphism.
A map $f \colon \underline G \to \underline G'$ of sheaves of groups on $\Cc$ induces a functor $f_\ast \colon \mr{PTor}_{\underline G} \to \mr{PTor}_{\underline G'}$ given by $E \mapsto E\times^{\underline G} \underline G'$, where the action of $\underline G'$ is by right multiplication in the second factor. Moreover, the canonical map $(\id_E, \mathbf{1}_{\underline G'}) \colon E \to E \times^{\underline G} \underline G'$ is $\underline G$-equivariant for the $\underline G$-action on $E \times^{\underline G} \underline G'$ via $f$ on the second factor.

A $\underline{G}$-pseudo-torsor $E$ is a \emph{$G$-torsor} if for every object $U$ on $\Cc$ there is a cover $\{U_i \to U \colon i \in I\}$ of $U$ in $\Cc$ such that $\Gamma(U_i, E) \neq \emptyset$. We denote by $\Bf(\underline G)$ the full subcategory of $\mr{PTor}_{\underline G}$ of $\underline G$-torsors on $\Cc$. The map $f_\ast$ for a map of sheaves of groups $f \colon \underline G \to\underline  G'$ restricts to a map $f_\ast \colon \Bf(\underline G) \to \Bf(\underline G')$. 

\begin{lemma}
	\label{lemPTor}
	Let $I$ be a finite partially ordered set and let $(\underline G_i)_{i \in I}$ be a diagram of sheaves of groups over $I$.
	Let $\underline G = \varprojlim_{i \in I} \underline G_i$. Then $\underline G$ is a sheaf of groups on $\Cc$ together with a compatible system of projection maps $f_i \colon \underline G \to \underline G_i$. 
	The functor
	$$ \varprojlim_{i \in I} f_{i, \ast} \colon \mr{PTor}_{\underline G} \to \varprojlim_{i \in I} \mr{PTor}_{\underline G_i}, \qquad E \mapsto (E \times^{\underline G} \underline G_i)_{i \in I}$$
	has a right-adjoint given by 
	$$ \lim \colon \left(\varprojlim_{i \in I} \mr{PTor}_{\underline G_i}\right) \to \mr{PTor}_{\underline G},  \qquad (E_i)_{i \in I} \mapsto \varprojlim_{i \in I} E_i.$$
	Moreover, the restriction $ \varprojlim_{i \in I} f_{i, \ast}\colon \Bf(\underline G) \to \varprojlim_{i \in I} \Bf(\underline{G}_i)$ to the full subcategory of torsors is fully faithful.
\end{lemma}

\begin{proof}
	As a first step, we show that $\varprojlim_{i \in I} E_i$ is indeed a pseudo-torsor for $\underline G$. The sheaf of groups $\underline G$ acts on $E_i$ by the action induced by $f_i$, and all these actions are compatible by the observation above that the reduction maps are equivariant. Hence, $\varprojlim_{i \in I} E_i$ carries a canonical $\underline G$-action.
	As all the $E_i$ are pseudo-torsors under $\underline G_i$, the induced map
	\begin{align*}
		\left( \varprojlim_{i \in I} E_i \right) \times \underline G & \to \left( \varprojlim_{i \in I} E_i \right) \times \left( \varprojlim_{i \in I} E_i \right) \\
		((e_i)_{i \in I}, g) & \mapsto ((e_i)_{i \in I}, (e_i f_i(g))_{i \in I})
	\end{align*}
	is an isomorphism, so $\varprojlim_{i \in I} E_i$ is a $\underline G$-pseudo-torsor.
	
	As a next step, we show that the limit is right adjoint to the family of projections.
	Let $(F_i)_{i \in I} \in \varprojlim_{i \in I} \mr{PTor}_{\underline G_i}$. 
	A $\underline G$-equivariant map $E \to F_i$ factors as $E \to E \times^{\underline G} \underline G_i \to F_i$ for a unique $\underline G_i$-equivariant map $E \times^{\underline G} \underline G_i \to F_i$. Hence, we get 
	$$ \Hom_{\mr{PTor}_{\underline G}}(E, \varprojlim_{i \in I} F_i) = \Hom_{\varprojlim_{i \in I} \mr{PTor}_{\underline G_i}}((E \times^{\underline G} \underline G_i )_{i \in I}, ( F_i)_{i \in I}).$$
	
	In order to see that the restriction to $\Bf(\underline G)$ is fully faithful, we check that the unit of the adjunction $E \mapsto \varprojlim_{i \in I} E \times^{\underline G} \underline G_i$ is an isomorphism for $E \in \Bf(\underline G)$. We can do so locally, so we may assume that $E$ is trivial. As all maps $E \to  E \times^{\underline G} \underline G_i$ are $\underline G$-equivariant, choosing a trivialisation of $E$ induces a compatible choice of trivialisations of all $E \times^{\underline G} \underline G_i$. Hence, the map $E \to \varprojlim_{i \in I} E \times^{\underline G} \underline G_i$ is given by $\underline G \to \varprojlim_{i \in I} \underline G_i$, which is an isomorphism by construction.
\end{proof}

\begin{remark}
	\label{remCounterLimTors}
	Note that given a compatible family of $\underline{G}_i$-torsors $(E_i)_{i \in I} \in \varprojlim_{i \in I} \Bf(G_i)$, their limit will in general not be a $\underline G$-torsor, as it might not be possible to produce a compatible system of sections for $(E_i)_{i \in I}$.
	For example, consider $G_1 = G_2 = \{e\}$ the trivial group and $G_3 = \Z/2$. Then $G_1 \times_{G_3} G_2 = \{e\}$ is again the trivial group. Let us moreover consider the sets $E_1 = E_2 = \{ \ast \}$ and $E_3 = \{a_1,a_2\}$. Then $E_i$ is a trivial $G_i$-torsor for all $i = 1,2,3$. However, under the maps $f_i \colon E_i \to E_3, \ast \mapsto a_i$ for $i = 1,2$, the fibre product $E_1 \times_{E_3} E_2$ is empty, hence in particular not a torsor under the trivial group.
\end{remark}

\subsection{Deeper Bruhat-Tits group schemes are limits of parahoric group schemes}

Let us briefly recall some facts from Bruhat-Tits theory. 
In this subsection, let $F$ be a discretely valued henselian field with ring of integers $\Oc$. We denote by $\mf \se \Oc$ its maximal ideal and by $k = \Oc/\mf$ its residue field.
Moreover, we denote by $K^{\mr{ur}}$ the maximal unramified extension inside some fixed algebraic closure of $K$, by  $\Oc^{\mr{ur}}$ its ring of integers and by $\breve{K}$ (respectively $\breve{\Oc}$) the completion of $K^{\mr{ur}}$ (respectively $\Oc^{\mr{ur}}$).

Let $G$ be a (connected) reductive group over $K$ such that $G$ is quasi-split over $K^{\mr{ur}}$. 
Note that $G$ is automatically quasi-split over $K^{\mr{ur}}$ when the cohomological dimension of $K^{\mr{ur}}$ is at most 1 by a theorem of Steinberg \cite[Theorem 1.9]{Steinberg1965} and Borel-Springer \cite[Section 8]{Borel1968}. This includes in particular the case that the residue field of $K$ is perfect and thus the case $K = k\dbr{\varpi}$ for a finite field $k$ we are interested in later.
Let us fix a maximal $K$-split torus $S \se G$.
We denote by $\Bc(G,K)$ the corresponding (reduced) Bruhat-Tits building and by $\Ac = \Ac(G,S, K) \se \Bc(G, K)$ the apartment corresponding to $S$.
Let $\Phi = \Phi(G,S)$ be the set of roots of $G$ with respect to $S$ and let $\Phi^+ \se \Phi$ be a system of positive roots. We denote by $\Phi^- = - \Phi^+$ and by $\Phi_{\mr{nd}}^{+} \se \Phi^+$ (respectively by $\Phi_{\mr{nd}}^{-} \se \Phi^-$) the subset of non-divisible positive (respectively negative) roots.

We consider the space of affine functionals $\Ac^\ast$ on $\Ac$ and  the set of affine roots $\Psi = \Psi(G,S) \se \Ac^\ast$ of $G$ with respect to $S$. For an affine functional $\psi \in \Ac^\ast$, let $\Hc_{\Psi} \se \Ac$ be the vanishing hyperplane for $\psi$ and let $\Hc_{\psi \geq 0} = \{x \in \Ac\colon \psi(x) \geq 0 \}$ (respectively $\Hc_{\psi \leq 0} = \{x \in \Ac\colon \psi(x) \leq 0 \}$) be the corresponding half-spaces. 
For an affine functional $\psi \in \Ac^\ast$, we denote by $\dot{\psi}$ its gradient. By construction, for $\psi \in \Psi$ we have $\dot{\psi} \in \Phi$.

For a non-empty bounded subset $\Omega \se \Ac$, we consider the corresponding (local) Bruhat-Tits group scheme\footnote{In the literature it is often additionally required that $\Omega$ is contained in a facet. We explicitly allow $\Omega$ to not be contained in the closure of a facet (this will be the interesting case later) and call $\Gc_\Omega$ with $\Omega$ contained in the closure of a facet a \emph{parahoric} (Bruhat-Tits) group scheme.}  $\Gc_\Omega$ constructed in \cite[§ 5.1.9 (resp. § 4.6.26)]{Bruhat1984}. It is the unique smooth affine $\Oc$-group scheme with generic fibre $G$, connected special fibre and $\Gc_\Omega(\Oc^{\mr{ur}}) = G(K^{\mr{ur}})^0_\Omega$, where $G(K^{\mr{ur}})^0_\Omega$ is the ``connected'' (pointwise) stabiliser of $\Omega$, by \cite[Remark 4]{Haines2008}.

For a bounded subset $\Omega \se \Ac$, we denote by $\mr{cl}(\Omega) = \bigcap_{\psi \in \Psi, \Omega \subseteq \Hc_{\psi \geq 0}} \Hc_{\psi \geq 0}$ the intersection of all half-spaces containing $\Omega$.
Then the corresponding Bruhat-Tits group scheme does not change when replacing $\Omega$ by $\mr{cl}(\Omega)$, compare \cite[§ 4.6.27]{Bruhat1984}.
Hence, we may always assume $\Omega = \mr{cl}(\Omega)$ in the following.
By construction, $\mr{cl}(\Omega)$ is convex.
For two bounded subsets $\Omega, \Omega'$ of $\Ac(G, S, K)$ with $\Omega = \mr{cl}(\Omega)$, we write $\Omega' \prec \Omega$ if $\Omega'$ is contained in $\Omega$.
In this case, we obtain an induced homomorphism of $\Oc$-group schemes $\rho_{\Omega', \Omega} \colon \Gc_{\Omega} \to \Gc_{\Omega'}$ whose restriction to the generic fibre is given by the identity on $G$.
Below, we often take limits over the partially ordered set $\{\ff \prec \Omega\}$ of facets contained in $\Omega$ ordered by inclusion. This poset is connected as $\Omega = \mr{cl}(\Omega)$ is connected.

For a root $a \in \Phi$ and $\Omega$ as above, we denote by $U_{a, \Omega} \se G(K)$ the corresponding root subgroup and by $\Uc_{a, \Omega}$ its integral model, which is a smooth affine $\Oc$-group scheme. 
As for the $\Gc_\Omega$, the group scheme  $\Uc_{a, \Omega}$ only depends on $\mr{cl}(\Omega)$ and for $\Omega' \prec \Omega$ there is a natural map $\Uc_{a, \Omega} \to \Uc_{a, \Omega'}$. 
These integral models are used to construct the \emph{big open cell}
$$ \prod_{a \in \Phi_{\mr{nd}}^{-}} \Uc_{a, \Omega} \times \Zc \times \prod_{a \in \Phi_{\mr{nd}}^{+}} \Uc_{a, \Omega} \hookrightarrow \Gc_\Omega,$$
which is an open immersion by \cite[§ 4.6.2]{Bruhat1984}, where $\Zc$ is the Bruhat-Tits group scheme of the centraliser $Z$ of $S$ defined over $\Oc$. Note that when $G$ is quasi-split, $Z = T$ is a maximal torus in $G$ and $\Zc = \Tc^0$ is the connected N\'eron model of $T$.

The main result of this section is the following theorem.
\begin{theorem}
	\label{thmBTGS}
	Let $G$ be a reductive group over $K$ such that $G$ is quasi-split over the maximal unramified extension $K^{\mr{ur}}$ of $K$. 
	Let $\Omega \se \Ac(G, S, K)$ be a bounded subset with $\Omega = \mr{cl}(\Omega)$.
	The map 
	$$\rho = \varprojlim_{\ff \prec \Omega} \rho_{\ff, \Omega} \colon \Gc_{\Omega} \to \varprojlim_{\ff \prec \Omega} \Gc_\ff$$
	induced by the $\rho_{\ff, \Omega}$ for facets $\ff \prec \Omega$ is an isomorphism of $\Oc$-group schemes. 
\end{theorem}

We need some results on the deformation theory of torsors under (limits of) Bruhat-Tits group schemes.
For us, torsors are always taken with respect to the fppf-topology. 
However, torsors for smooth affine group schemes are always representable by a (necessarily smooth affine) scheme and thus have sections \'etale locally. 
The deformation theory of such sections of torsors can be controlled by the (dual of) the invariant differentials $\omega_{\Gc/\Oc} = e^{\ast} \Omega_{\Gc/\Oc}$, where $e \colon \Oc \to \Gc$ is the identity section, due to the following result.
\begin{lemma}
	\label{lemDefoLifts}
	Let $\Gc$ be a smooth affine $\Oc$-group scheme and let $R$ be an $\Oc$-algebra with an ideal $I$ of square $I^2 = 0$. We denote by $\overline{R} = R/I$ and $r \colon \Oc \to \overline{R}$ the induced map. Let $\Ec$ be a $\Gc$-torsor over $R$. Let $\gamma \in \Ec(\overline R)$ be a section of $\Ec$. 
	Then the set of all lifts of $\gamma$ to $R$ is a torsor under $\gf_{(R,I)} = r^\ast \omega_{\Gc/\Oc}^\vee \otimes_{\overline R} I$.
\end{lemma}
\begin{proof}
	This is essentially a special case of \cite[Expos\'e III, Corollaire 5.2]{SGA1}.
	Recall that $\Ec$ is representable by a smooth affine $\Oc$-scheme. In particular, there exist lifts of $\gamma$ to $R$, so $\Ec$ is trivial. So let us fix a lift $\gamma'$ of $\gamma$ and a trivialisation of $\Ec$ that identifies the section $\gamma'$ with the unit in $\Gc_R$.
	By  \cite[Expos\'e III, Corollaire 5.2]{SGA1}, the set of lifts of $\gamma$ is then a torsor under
	$$ \gamma^{\ast} \Omega^\vee_{\Ec/R} \otimes_{\overline R} I \cong  r^\ast e^{\ast} \Omega^\vee_{\Gc/\Oc} \otimes_{\overline R} I = r^\ast \omega_{\Gc/\Oc}^\vee \otimes_{\overline R} I.$$ 
\end{proof}
We use the following lemma to relate the deformation theory problem to the combinatorics in the Bruhat-Tits building.
\begin{lemma}[{compare \cite[§ 4.6.41]{Bruhat1984}}]
	\label{lemLieAlgHyp}
	Assume that $G$ is quasi-split.
	Let $\psi \in \Ac^\ast$ be an affine functional with gradient $a = \dot{\psi}$. Let $\Omega \se \Ac$ be a bounded subset such that $\Omega \se \Hc_{\psi \leq 0}$. Let moreover $\Omega' \prec \Omega$ such that $\Omega' \se \Hc_{\psi}$. 
	Then the natural map $\omega^{\vee}_{\Uc_{a, \Omega}/\Oc} \to \omega^\vee_{\Uc_{a, \Omega'}/\Oc}$ is an isomorphism.
\end{lemma}
\begin{proof}
	By assumption, we have $U_{a, \Omega} = U_{a, \Omega'}$ as subgroups of $G(K)$. Hence, the induced maps on integral models and consequently on invariant differentials are isomorphisms.
\end{proof}
Note that in the situation of the lemma when $\Omega \cap \Hc_{\psi < 0} \neq \emptyset$ the induced map on Lie algebras for the negative root groups
$$\mr{Lie}(\Uc_{-a, \Omega, k}) = \omega^{\vee}_{\Uc_{-a, \Omega}/\Oc} \otimes_\Oc k \to \mr{Lie}(\Uc_{-a, \Omega', k})  = \omega^{\vee}_{\Uc_{-a, \Omega'}/\Oc} \otimes_\Oc k$$ in the special fibre of $\Spec(\Oc)$ typically (in particular when $a$ is non-divisible and $2a$ is not a root) is the zero map by {\cite[§ 4.6.41]{Bruhat1984}}.

Let $(\Ec_\ff)_{\ff \prec \Omega} \in \varprojlim_{\ff \prec \Omega} \Bf(\Gc_\ff)(R)$ be a compatible system of $\Gc_\ff$-torsors.
We use the previous two lemmas to construct compatible lifts of sections of $\Ec_\Omega = \varprojlim_{\ff \prec \Omega} \Ec_\ff$.
This serves two purposes: On the one hand, we use this result for the trivial torsors $\Ec_\ff = \Gc_\ff$ to show that we can lift sections from the special fibre of $\varprojlim_{\ff \prec \Omega} \Gc_\ff$ in the proof of Theorem \ref{thmBTGS} and on the other hand, we use it in the proof of Proposition \ref{propCritTors}, which gives a criterion when $\Ec_\Omega$ is actually a $\Gc_\Omega$-torsor. 
For a subset $\Omega' \prec \Omega$ we denote by $ \Ec_{\Omega'}= \varprojlim_{\ff \prec \Omega'} \Ec_\ff$.

\begin{lemma}
	\label{lemTorsDefo}
	Assume that $G$ is quasi-split. Let $R$ be an $\Oc$-algebra with an ideal $I$ of square $I^2 = 0$. We denote by $\overline{R} = R/I$.
	
	\begin{enumerate}
		\item 
		\label{lemTorsDefoPair}
		Let $\Omega_1, \Omega_2 \prec \Omega$ be two bounded subsets such that $\Omega_1 = \mr{cl}(\Omega_1)$, $\Omega_2 = \mr{cl}(\Omega_2)$ and that $\Omega_1 \cap \Omega_2$ is contained in an affine root hyperplane $\Hc_\psi$ for some $\psi \se \Psi$. Assume moreover that $\Omega_1 \cup \Omega_2$ is convex and that $\Omega_1 \se \Hc_{\psi \geq 0}$ and $\Omega_2 \se  \Hc_{\psi \leq 0}$ lie in different half-spaces.
		
		Assume that the assertion of Theorem \ref{thmBTGS} holds for $\Gc_{\Omega_1}$ and $\Gc_{\Omega_2}$.
		Assume that there is a section $\gamma \in \Ec_{\Omega_1 \cup \Omega_2}(\overline{R})$ and deformations $\gamma_{\Omega_1} \in  \Ec_{\Omega_1}(R)$ and $\gamma_{\Omega_2} \in  \Ec_{\Omega_2}(R)$ of the images of $\gamma$ in $\Ec_{\Omega_1}$ and $\Ec_{\Omega_2}$, respectively. 
		Then there exists a deformation $\gamma_{\Omega_1 \cup \Omega_2} \in  \Ec_{\Omega_1 \cup \Omega_2}(R)$ of $\gamma$.
		
		\item 
		\label{lemTorsDefoSlice}
		Let now $\Omega' = \mr{cl}(\Omega') \prec \Omega$ and let $a \in \Phi^+_{\mr{nd}}$ and let $\psi_1 < \psi_2 < \ldots < \psi_m$ be the affine roots with gradient $\dot{\psi_i} = a$ such that $\Omega \cap \Hc_{\psi_i} \neq \emptyset$.
		We denote by $\Omega_i = \overline{(\Omega \cap \Hc_{\psi_i \leq 0} )\setminus \Omega_{i-1}}$ for $i = 1, \ldots, m$ with $\Omega_0 = \emptyset$ and $\Omega_{m+1} = \Omega \setminus (\Omega_m\setminus \Hc_{\psi_m}) $.
		
		Assume that the assertion of Theorem \ref{thmBTGS} holds for $\Gc_{\Omega_i}$ for $i = 1, \ldots, m+1$.
		Assume that there is a section $\gamma \in \Ec_{\Omega'}(\overline{R})$ and deformations $\gamma_{\Omega_i} \in  \Ec_{\Omega_i}(R)$ of the image of $\gamma$ in $\Ec_{\Omega_i}$ for all $1 \leq i \leq m+1$.
		Then there exists a deformation $\gamma_{\Omega'} \in  \Ec_{\Omega'}(R)$ of $\gamma$.
	\end{enumerate}	
\end{lemma}
We will prove Theorem \ref{thmBTGS} by induction on $\Omega$ and use this lemma in the inductive step. Hence, it is feasible to assume the validity of Theorem \ref{thmBTGS} for subsets of $\Omega$ here. Once we have established  Theorem \ref{thmBTGS} in full (in particular for the application of the lemma in the proof of Proposition \ref{propCritTors}), these conditions of course are vacuous.
Before we give the proof of the lemma, let us briefly discuss an example that nicely illustrates the main idea.
\begin{example} 
	We consider $G = \GL_2$ over $K = \Fq \dbr{\varpi}$ with $T$ the split maximal diagonal torus. 
	Then $X^\ast(T) \cong \Z^2$ with roots $\Phi = \{\pm (1,-1) \} \se X^\ast(T)$, where the choice of the positive root $a = (1,-1)$ corresponds to the choice of the Borel subgroup given by upper triangular matrices. 
	Let us consider the interval $\Omega = [0,2] \se \R \cong \Ac(\GL_2, T)$ with $\Omega_1 = [0,1]$ and $\Omega_2 = [1,2]$.
	
	\begin{center}
		\begin{tikzpicture}
			\draw (-2,0)-- (6,0);
			\foreach \x in {0,1,2} {
				\draw (2*\x,0.125) -- (2*\x,-0.125) node[below] {\x};
			}
			\draw (1,0.25) node {$\Omega_1$};
			\draw (3,0.25) node {$\Omega_2$};
		\end{tikzpicture}
	\end{center}
	
	Let us consider the case $R = \Fq \dsq{\varpi}/(\varpi^2)$ and $\overline{R} = R/(\varpi) = \Fq$. In this case, for a smooth affine group scheme $\Gc$ over $\Oc$, the module $\gf = e^\ast \omega^\vee_{\Gc/\Oc} \otimes_\Fq (\varpi)/(\varpi^2)$ is given by the tangent space of $\Gc$ at the identity section in its special fibre. 
	Let us assume we are in the situation of Lemma \ref{lemTorsDefo} (\ref{lemTorsDefoPair}).
	We are given a section $\gamma \in \Ec_{[0,2]}(\Fq)$ and sections $\gamma_{[0,1]} \in \Ec_{[0,1]}(\Fq\dsq{\varpi}/(\varpi^2))$ and $\gamma_{[1,2]} \in \Ec_{[1,2]}(\Fq\dsq{\varpi}/(\varpi^2))$ that lift $\gamma$. 
	Recall that by Lemma \ref{lemDefoLifts}, for $\Omega' \prec \Omega$ the set of all lifts of $\gamma$ in $\Ec_{\Omega'}$ is a torsor under $\gf_{\Omega'}$. 
	Hence, after fixing a trivialisation of $\Ec_{\{1\}}$, the images of the lifts $\gamma_{[0,1]}, \gamma_{[1,2]}$ in $\Ec_{\{1\}}$ induce points in $\gf_{\{1\}}$.
	Thus, the question becomes if the intersection of the orbits $\gf_{[0,1]}. \gamma_{[0,1]} \cap \gf_{[1,2]}. \gamma_{[1,2]}$ in $\gf_{\{1\}}$ is non-empty, where $\gf_{[0,1]}$ acts via the natural map $\gf_{[0,1]} \to \gf_{\{1\}}$, similarly for $\gf_{[1,2]}$. 
	
	For $\Omega' \prec \Omega$, we decompose the Lie algebras into its root spaces $\gf_{\Omega'} = \uf_{a, \Omega'} \oplus \hf \oplus \uf_{-a, \Omega'}$, where $a =(1,-1)$ is the positive root. 
	In this situation, the root spaces $\uf_{\pm a, \Omega'}$ are one-dimensional while the Cartan $\hf$ is two-dimensional.
	Then the induced map $\gf_{[0,1]} \to \gf_{\{1\}}$ is the identity on the Cartan algebra $\hf$ as well as on the positive root space $\uf_{a, [0,1]} = \uf_{a, \{1\}}$ by Lemma \ref{lemLieAlgHyp} while it is the zero map $\uf_{-a, [0,1]} \to \uf_{-a, \{1\}}$ on the negative root spaces.
	By a similar argument, for the second facet $\Omega_2 = [1,2]$ the map $\gf_{[1,2]} \to \gf_{\{1\}}$ is the identity on the Cartan and the negative root space, while it is the zero map on the positive root space.
	
	Decomposing the lifts $\gamma_{[0,1]}$ and $\gamma_{[1,2]}$ in their components, this shows that by the action of $\gf_{[0,1]}$ we can guarantee that the $\uf_{a}$-components agree and by the action of $\gf_{[1,2]}$ we can get matching components in the $\uf_{-a}$-component. This shows the non-emptiness of the intersection of the orbits and hence the existence of a compatible set of lifts.
	
	In order to guarantee the correct mapping property in the other directions, it is necessary to have the convexity assumption. This can be seen in the following example in the $\GL_3$-case:
	
	
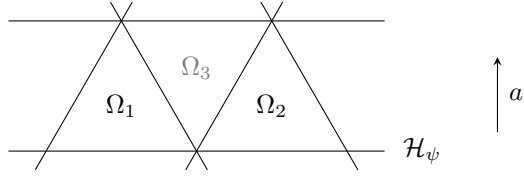
\begin{center}
		\begin{tikzpicture}
			\draw (-0.5,0)-- (4.5,0);
			\draw (-0.5,1.7321)-- (4.5,1.7321);
			\draw (-0.1443, -0.25) -- (1.1443, 1.9821);
			\draw (1.8557, -0.25) -- (3.1443, 1.9821);
			\draw (2.1443, -0.25) -- (0.8557, 1.9821);
			\draw (4.1443, -0.25) -- (2.8557, 1.9821);
			\draw (1,0.6) node {$\Omega_1$};
			\draw (3,0.6) node {$\Omega_2$};
			\draw (5,0) node {$\Hc_\psi$};
			\draw [-stealth](6,0.25) -- (6,1.25);
			\draw (6.25,0.75) node {$a$};
			\draw [gray] (2,1.1) node {$\Omega_3$};
		\end{tikzpicture}
	\end{center}
	
	We are given two chambers $\Omega_1$ and $\Omega_2$ in the standard apartment in the Bruhat-Tits building of $\GL_3$ that intersect in a single vertex. 
	In particular, $\Omega_1 \cup \Omega_2$ is not convex. 
	The base of both of the triangles lies in some affine root hyperplane $\Hc_\psi$ with $\dot{\psi} = a$ while both $\Omega_1$ and $\Omega_2$ are contained in the positive half space $\Hc_{\psi \geq 0}$. 
	But this means that both $\uf_{a, \Omega_1} \to \uf_{a, \Omega_1 \cap \Omega_2}$ and  $\uf_{a, \Omega_2} \to \uf_{a, \Omega_1 \cap \Omega_2}$ are the zero maps.
	Hence, it is in general not possible to lift sections in this situation.	 
	
	The difference to the convex case is the following. We have $\mr{cl}(\Omega_1 \cup \Omega_2) = \Omega_1 \cup \Omega_2 \cup \Omega_3$, where $\Omega_3$ is the triangle ``between'' $\Omega_1$ and $\Omega_2$. For a pair of $\Gc_{\Omega_1}$- (respectively $\Gc_{\Omega_2}$-) torsors $\Ec_{\Omega_1}$ and $\Ec_{\Omega_2}$ the existence of a compatible $\Gc_{\Omega_3}$-torsor $\Ec_{\Omega_3}$ (such a torsor does not exist in general!) can be interpreted as a compatibility condition on the $a$-root spaces, as it will guarantee by Lemma \ref{lemTorsDefo} (\ref{lemTorsDefoPair}) that for two given lifts $\gamma_{\Omega_1} \in \Ec_{\Omega_1}(\Fq \dsq{\varpi}/(\varpi^2))$ and $\gamma_{\Omega_2} \in \Ec_{\Omega_2}(\Fq \dsq{\varpi}/(\varpi^2))$ their image in $\uf_{a, \Omega_1 \cap \Omega_2}$ agrees.
\end{example}

\begin{proof}[Proof of Lemma \ref{lemTorsDefo}]
	\begin{enumerate}
		\item 
		Given some $\Omega' \prec \Omega$ (for which Theorem \ref{thmBTGS} holds), the set of all lifts of $\gamma \in \Ec_{\Omega'}(\overline{R})$ to $\Ec_{\Omega'}(R)$ is a torsor under $\gf_{\Omega'} = \gf_{\Omega', (R,I)}$ (if such lifts exist at all) by Lemma \ref{lemDefoLifts}. Using the decomposition of the big open cell in $\Gc_{\Omega'}$, we can decompose $\gf_{\Omega'}$ into the root spaces as 
		$$\gf_{\Omega'} = \bigoplus_{a \in \Phi^-_{\mr{nd}}} \uf_{a, \Omega'} \oplus \hf \oplus \bigoplus_{a \in \Phi^+_{\mr{nd}}} \uf_{a, \Omega'}.$$ 
		After fixing a trivialisation of $\Ec_{\Omega_1 \cap \Omega_2}$, the images of the lifts $\gamma_{\Omega_1}$ and $\gamma_{\Omega_2}$ in $\Ec_{\Omega_1 \cap \Omega_2}$ thus define elements of $\gf_{\Omega_1 \cap \Omega_2}$. 
		The question whether there exists a lift $\gamma_{\Omega_1 \cup \Omega_2} \in \Ec_{\Omega_1 \cup \Omega_2}(R)$ of $\gamma$, or in other words, a compatible pair of lifts  $\gamma'_{\Omega_1}$ and $\gamma'_{\Omega_2}$ in $\Ec_{\Omega_1}$ (respectively in $\Ec_{\Omega_2}$), 
		is thus the question if the orbits in $\gf_{\Omega_1 \cap \Omega_2}$ have a non-empty intersection
		$$\gf_{\Omega_1}. \gamma_{\Omega_1} \cap \gf_{\Omega_2}. \gamma_{\Omega_2} \neq \emptyset.$$ 
		We treat this question componentwise with respect to the decomposition into root spaces. 
		On the torus part this is clear as the maps $\gf_{\Omega_i} \to \gf_{\Omega_1 \cap \Omega_2}$ restrict to isomorphisms on $\hf$ by construction for $i = 1,2$. 
		It suffices to show that for all roots $a \in \Phi_{\mr{nd}}$ at least one of $\gf_{\Omega_i} \to \gf_{\Omega_1 \cap \Omega_2}$ restricts to an isomorphism $\uf_{a,\Omega_i} \to \uf_{a,\Omega_1 \cap \Omega_2}$.
		For $a = \pm \dot{\psi}$ this directly follows from Lemma \ref{lemLieAlgHyp}.
		
		Let now $a \in \Phi \setminus \{ \pm \dot \psi\}$, and let $\psi' \in \Ac$ the minimal affine functional with gradient $\dot{ \psi}' = a$ such that $\Omega_1 \cap \Omega_2 \se \Hc_{\psi' \leq 0}$. By the convexity assumption, at least one of the $\Omega_i$ is contained in $\Hc_{\psi' \leq 0}$ for $i = 1,2$. 
		But then $\uf_{a, \Omega_i} \xrightarrow{\cong} \uf_{a, \Omega_1 \cap \Omega_2}$ is an isomorphism by Lemma \ref{lemLieAlgHyp}.
		
		\item For each $i = 1, \ldots, m$, the pair of subsets $\bigcup_{1 \leq j \leq i} \Omega_j, \Omega_{i+1}$ of $\Omega'$ satisfies the assumptions of (\ref{lemTorsDefoPair}) by construction (in particular, their intersection is contained in $\Hc_{\psi_i}$).
		Using induction on $i$, we construct lifts of $\gamma$ for all $\Ec_{\bigcup_{1 \leq j \leq i} \Omega_i}$ using (\ref{lemTorsDefoPair}), and hence in particular for $\Ec_{\Omega'}$. 
	\end{enumerate}
\end{proof}

\begin{proof}[Proof of Theorem \ref{thmBTGS}]			
	We first remark that the limit $\varprojlim_{\ff \prec \Omega} \Gc_\ff$ is a finite limit of affine $\Oc$-group schemes of finite type, hence is again an affine $\Oc$-group scheme of finite type. Moreover, as all transition maps are identities on the generic fibres, the generic fibre of the limit is isomorphic to $G$ and $\rho$ induces an isomorphism on the generic fibre. 
	
	By \'etale descent it suffices to work over $\breve K$, the completion of the maximal unramified extension of $K$. We may thus assume that $K = \breve K$, in which case $G$ is quasi-split by assumption. 
	Moreover, we have 
	$$(\varprojlim_{\ff \prec \Omega} \Gc_\ff)(\Oc) = \varprojlim_{\ff \prec \Omega} (\Gc_\ff(\Oc)) = \bigcap_{\ff \prec \Omega} G(K)^0_\ff = G(K)^0_\Omega.$$
	It remains to show that $\varprojlim_{\ff \prec \Omega} \Gc_\ff$ is smooth, as smoothness implies by \cite[§ 1.7.3]{Bruhat1984} that $\varprojlim_{\ff \prec \Omega} \Gc_\ff$ is \'etoff\'e in the sense of \cite[D\'efinition 1.7.1]{Bruhat1984}. But this means that $\rho$ is an isomorphism by the previous observations.
	
	We use induction on $\Omega$ to show that $\varprojlim_{\ff \prec \Omega} \Gc_\ff$ is smooth. 
	Let us fix some enumeration of the set of non-divisible positive roots $\Phi^+_{\mr{nd}} = \{ a_1, \ldots, a_m \}$.
	We inductively cut down $\Omega$ into slices by hyperplanes with gradient $a_i$ and in each step use Lemma \ref{lemTorsDefo} (\ref{lemTorsDefoSlice}) to construct lifts of the section in the special fibre.
	For the start of the induction, note that the theorem clearly is satisfied when $\Omega$ is (the closure of) a facet.
	More concretely, in the last step of the induction we write $\Omega = \bigcup_{1 \leq i \leq m+1} \Omega_i$ using the notation from Lemma \ref{lemTorsDefo} (\ref{lemTorsDefoSlice}) with $a = a_1$. By induction, we assume that the theorem holds for each $\Omega_i$ (that we got by cutting down each $\Omega_i$ using hyperplanes with gradient $a_2$). 
	
	We check that $\varprojlim_{\ff \prec \Omega} \Gc_\ff$ is formally smooth. 
	Let $R$ be an $\Oc$-algebra and let $I \se R$ be an ideal of square zero. We denote by $\overline{R} = R/I$.
	Let us fix a section $\overline g \in \varprojlim_{\ff \prec \Omega} \Gc_{\ff}(\overline{R})$.
	Using the inductive hypothesis, there exist sections $g_i \in \varprojlim_{\ff \prec \Omega_i} \Gc_{\ff}(R) = \Gc_{\Omega_i}(R)$.
	By Lemma \ref{lemTorsDefo} (\ref{lemTorsDefoSlice}), we then obtain a lift $g \in \varprojlim_{\ff \prec \Omega} \Gc_{\ff}(R)$.
	As $ \varprojlim_{\ff \prec \Omega} \Gc_{\ff}$ is an affine scheme of finite presentation over $\Oc$, this shows that $\Gc_\Omega$ is smooth.
	This finishes the proof of the theorem.
\end{proof}

\begin{corollary}
	The Bruhat-Tits group scheme $\Gc_\Omega$ is isomorphic to the closure of the diagonal in the generic fibre 
	$$G \xrightarrow{\Delta}  
	\prod_{\ff \prec \Omega} \Gc_\ff.$$
\end{corollary}
\begin{proof}
	The inclusion $\varprojlim_{\ff \prec \Omega} \Gc_\ff \to \prod_{\ff \prec \Omega} \Gc_\ff$ is a closed immersion since all $\Gc_\ff$ are affine and thus separated over $\Oc$. Since $\Gc_\Omega$ is in particular flat over $\Oc$, it is the closure of its generic fibre. The claim then follows from Theorem \ref{thmBTGS}.
\end{proof}

We consider torsors for the Bruhat-Tits group schemes above.
Recall that a limit of $\Gc_\ff$-torsors for facets $\ff \prec \Omega$ is a $\Gc_\Omega$-pseudo torsor by Lemma \ref{lemPTor}, but may fail to be a $\Gc_\Omega$-torsor in general.
We give a criterion when a limit of $\Gc_\ff$-torsors is already a $\Gc_\Omega$-torsor.
\begin{proposition}
	\label{propCritTors} 
	Let $\Omega \se \Ac$ be a bounded subset with $\Omega = \mr{cl}(\Omega)$ and let $R$ be an $\Oc$-algebra.
	Let $(\Ec_\ff)_{\ff \prec \Omega} \in \varprojlim_{\ff \prec \Omega} \Bf(\Gc_{\ff})(R)$.
	Then 
	$$ \Ec_\Omega = \varprojlim_{\ff \prec \Omega} \Ec_\ff $$
	is a smooth affine $R$-scheme. 
	In particular, $\Ec_{\Omega}$ is a $\Gc_\Omega$-torsor if and only if $\Ec_{\Omega} \to \Spec(R)$ is surjective.
\end{proposition}

\begin{proof}
	The second assertion follows from the first one using Lemma \ref{lemPTor}, Theorem \ref{thmBTGS} and \cite[Expos\'e XI, Proposition 4.2]{SGA1}.
	
	The first assertion is \'etale-local on $\Spec(R)$, so we may assume that $G$ is quasi-split.
	We proceed by induction on $\Omega$ as in the proof of Theorem \ref{thmBTGS}.
	In particular, we may write $\Ec_\Omega$ as an iterated fibre product of the $\Ec_{\mathfrak{f}}$. This shows that $\Ec_\Omega$ is affine and of finite presentation over $R$ because the $\Ec_{\mathfrak{f}}$ are.
	It remains to show that $\Ec_\Omega \to \Spec(R)$ is formally smooth, 
	but this follows from Lemma \ref{lemTorsDefo} (\ref{lemTorsDefoSlice}).
\end{proof}

\subsection{Torsors for global Bruhat-Tits group schemes}

The goal of this subsection is to show that the isomorphism of Bruhat-Tits group schemes of Theorem \ref{thmBTGS} induces an immersion on the level of the corresponding moduli stacks $\Bun_\Gc$.
Therefore, let us now change perspective and consider (global) Bruhat-Tits group schemes in the following sense. 
Let again $X$ be a smooth projective and geometrically connected curve over $\Fq$ with function field $F$.

\begin{definition}
	\label{defnBTgrpschm}
	A smooth, affine group scheme $\Gc \to X$ is called a \emph{(global) Bruhat-Tits group scheme} if it has geometrically connected fibres, its generic fibre $\Gc_F = G$ is a reductive group over $F$ and if for all closed points $x$ of $X$ the pullback $\Gc_{\Oc_x} = \Gc \times_X \Spec(\Oc_x)$ is of the form $\Gc_\Omega$ for some bounded subset $\Omega$ contained in an apartment of the Bruhat-Tits building $B(G/F_x)$.
	The group scheme $\Gc$ is called a \emph{parahoric (Bruhat-Tits) group scheme} if moreover all $\Gc_{\Oc_x}$ are parahoric group schemes.
\end{definition}

Let $G$ be a (connected) reductive group over the function field $F$ of $X$. Bruhat-Tits group schemes with generic fibre $G$ can be constructed as follows.

\begin{construction}
	\label{consBTgrpschm}
	\begin{enumerate}
		\item 
		There exists a reductive model $G \to U$ of $G$ over some dense open subset $U \se X$. For each of the finitely many points $x \in X \setminus U$ in the complement of $U$ we choose a parahoric group scheme $\Gc^{(x)} \to \Spec(\Oc_x)$ with generic fibre $\Gc^{(x)}_{F_x} = G_{F_x}$. 
		As $U \amalg \coprod_{x \in X \setminus U} \Spec(\Oc_x) \to X$ is an fpqc-cover, we can glue $G \to U$ with all $\Gc^{(x)}$ using fpqc-descent to obtain a smooth affine group scheme $\Gc \to X$, which is a parahoric group scheme by construction.
		\item 
		\label{consBTgrpschmOmega}
		Let us now fix a parahoric model $\Gc \to X$ and a closed point $x_0$ of $X$. 
		For a connected bounded subset $\Omega$ in an apartment of the Bruhat-Tits building of $G_{F_{x_0}}$ as in the previous paragraph, we denote by $\Gc_\Omega \to \Spec(\Oc_{x_0})$ the corresponding (local) Bruhat-Tits group scheme. 
		We glue $\Gc_{\Omega}$ with $\Gc$ along the identity over $F_{x_0}$ and denote the resulting smooth affine group scheme over $X$ by a slight abuse of notation again by $\Gc_{\Omega}$.
		Then $\Gc_\Omega$ is a Bruhat-Tits group scheme in the sense of the previous definition and parahoric if and only if $\Omega$ is contained in the closure of a facet.
		
		The local homomorphisms $\rho_{\Omega', \Omega} \colon \Gc_\Omega \to \Gc_{\Omega'}$ over $\Spec(\Oc_{x_0})$ for $\Omega' \prec \Omega$ glue with the identity away from $x_0$ to morphisms of group schemes $\rho_{\Omega', \Omega} \colon \Gc_\Omega \to \Gc_{\Omega'}$ on $X$.
	\end{enumerate}
\end{construction}

In particular, the isomorphism of Theorem \ref{thmBTGS} extends to an isomorphism
$ \Gc_\Omega \xrightarrow{\cong} \varprojlim_{\ff \prec \Omega} \Gc_\ff$
of the corresponding global Bruhat-Tits group schemes.

For any smooth affine group scheme $\Hc$ on $X$, we denote by $\Bun_\Hc$ the moduli stack of $\Hc$-bundles on $X$. 
By the functoriality of $\Bun$, the maps $\rho_{\ff, \Omega}$ induce maps  $\rho_{\ff, \Omega, \ast} \colon \Bun_{\Gc_\Omega} \to  \Bun_{\Gc_\ff}$ for all facets $\ff \prec \Omega$.
\begin{theorem}
	\label{thmBunGImm}
	Let $G$ be a reductive group over $F$, let $x_0$ be a closed point of $X$ and let $\Omega = \mr{cl}(\Omega)$ be a  bounded subset of an apartment in the Bruhat-Tits building $\Bc(G_{F_{x_0}}, F_{x_0})$. Let $\Gc_{\Omega} \to X$ be the corresponding Bruhat-Tits group scheme from Construction \ref{consBTgrpschm} (\ref{consBTgrpschmOmega}). 
	The map 
	$$ \rho_{\Omega, \ast}  := \varprojlim_{\ff \prec \Omega}  \rho_{\ff, \Omega, \ast} \colon \Bun_{\Gc_\Omega} \to \varprojlim_{\ff \prec \Omega} \Bun_{\Gc_\ff} $$
	induced by the $\rho_{\ff, \Omega, \ast}$ for facets $\ff \prec \Omega$ is schematic and a quasi-compact open immersion.
\end{theorem}

\begin{proof} 
	By \cite[Proposition 3.19]{Breutmann2019}, the maps $\rho_{\ff, \Omega, \ast}$ are schematic and quasi-projective for all facets $\ff \prec\Omega$. By Lemma \ref{lemLim} below, the map $ \rho_{\Omega, \ast} $ is schematic, separated and of finite type.
	By Lemma \ref{lemPTor}, the map $\rho_{\Omega, \ast}$ is a monomorphism. 
	We show that $\rho_{\Omega, \ast}$ is formally \'etale.
	As all $\Bun_{\Gc_\ff}$ are locally of finite type over $\Fq$ by \cite[Proposition 1]{Heinloth2010}, it suffices to check the infinitesimal lifting criterion for local Artin rings over $\Fq$ by \cite[Proposition 17.14.2]{EGA4IV}.
	Let $R$ be a local Artin $\Fq$-algebra and let $I \se R$ be an ideal of square zero. Let moreover $(\Ec_\ff)_{\ff \prec \Omega} \in \varprojlim_{\ff \prec \Omega} \Bun_{\Gc_\ff} (R)$ such that $\varprojlim_{\ff \prec \Omega} \Ec_{\ff}$ is a $\Gc_\Omega$-torsor over $X_{\overline{R}}$, where $\overline{R} = R/I$. 
	We claim that $\varprojlim_{\ff \prec \Omega} \Ec_{\ff}$ is already a $\Gc_\Omega$-torsor over $X_R$. 
	The map $\widehat{(X_R)_{x_0}} \cup (X \setminus \{x_0\})_R \to X_R$ is a fpqc-cover, where $\widehat{(X_R)_{x_0}} = \Spec(\Oc_{x_0} \widehat{\otimes}_\Fq R)$, with $\Oc_{x_0} \widehat{\otimes}_\Fq V$ being the underlying $\Fq$-algebra of the completion of $X_R$ along $x_0$.
	As all maps $\Gc_\Omega \to \Gc_\ff$ for $\ff \prec \Omega$ are the identity away from $x_0$, all transition maps $\Ec_{\ff', R} \times^{\Gc_{\ff'}} \Gc_\ff \to \Ec_{\ff, R}$ are isomorphisms away from $x_0$. 
	Using Proposition \ref{propCritTors}, it remains to check that the pullback to $\varprojlim_{\ff \prec \Omega} \Ec_{\ff} \to \widehat{(X_R)_{x_0}}$ is surjective, but the underlying topological spaces of $\widehat{(X_R)_{x_0}}$ and $\widehat{(X_{\overline{R}})_{x_0}}$ agree.
	
	Hence, $\rho_{\Omega, \ast}$ is formally \'etale and thus a quasi-compact open immersion being a flat monomorphism of finite presentation.
\end{proof}


\section{Level maps and integral models with deep Bruhat-Tits level}
\label{sectIntMod}

We construct integral models for moduli spaces of shtukas with deep Bruhat-Tits level structures and show that these integral models admit proper, surjective and generically \'etale level maps.


\subsection{Some lemmata on algebraic stacks}
We collect some results on finite connected limits of algebraic stacks we use below for which we could not find a reference in the literature.

In this section, $I$ will always denote a finite connected index category and $(\Xc_i)_{i \in I}$ denotes a diagram over $I$ of (fppf-) Artin stacks over some base scheme $S$.

\begin{lemma}
	\label{lemLimProj}
	Assume that all algebraic stacks $\Xc_i$ have a diagonal that is schematic.
	Let all transition maps in $(\Xc_i)_{i \in I}$ be schematic. Then the projections $\varprojlim_{i \in I} \Xc_i \to \Xc_j$ are schematic for all $j \in I$.
	
	Moreover, assume that all $\Xc_i$ are separated over $S$ and that all transition maps have a property $\mathbf{P}$ of morphisms of schemes that is stable under base change and composition and is smooth local on the target such that all proper maps have $\mathbf{P}$. Then the projections $\varprojlim_{i \in I} \Xc_i \to \Xc_j$ have property $\mathbf{P}$ for all $j \in I$.
\end{lemma}
\begin{proof}
	It suffices to show the claim for fibre products and equalisers. For fibre products this is clear.
	Let us thus consider the equaliser diagram
	\begin{center}
		\begin{tikzcd}
			\Xc_1 \arrow[shift left=.75ex]{r}{f} \arrow[shift right=.75ex, swap]{r}{g}
			& \Xc_2.
		\end{tikzcd}
	\end{center}
	
	The equaliser of this diagram is given by the fibre product $\Xc = \Xc_2 \times_{\Delta, \Xc_2 \times_S \Xc_2, (f,g)} \Xc_1$.
	Thus, the projection $\Xc \to \Xc_1$ arises as the base change of the diagonal of $\Xc_1$ and is thus schematic in the first case and moreover proper in the second case (as we assumed $\Xc_1$ to be separated).
	The projection $\Xc \to \Xc_2$ has the required properties as it is the composition $\Xc \to \Xc_1 \to \Xc_2$.
\end{proof}

\begin{lemma}
	\label{lemLim}
	Let $(f_i \colon \Xc \to \Xc_i)_{i \in I}$ be a cone over the diagram $(\Xc_i)_{i \in I}$ such that all maps $f_i$ are schematic.
	Then the limit $f \colon \Xc \to \varprojlim_{i \in I} \Xc_i$ is schematic as well.
	
	Assume moreover that all $f_i$ are separated and have a property $\mathbf{P}$ of morphisms of schemes that is stable under base change and composition and is smooth local on the target such that all closed immersions have $\mathbf{P}$. Then $f$ has $\mathbf{P}$. 
\end{lemma}

\begin{proof}
	
	Let $T$ be an $S$-scheme. Let us fix a map $T \to \varprojlim_{i \in I} \Xc_i$. As different limits commute, we get that
	$$ T \times_{\varprojlim_{i \in I} \Xc_i} \Xc  = \varprojlim_{i \in I} (T \times_{\Xc_i} \Xc),$$
	which is representable by a scheme by assumption.
	For the second part, let us denote by $T_i =  T \times_{\Xc_i} \Xc$.
	Then $T_i$ is a separated $T$-scheme by assumption.  As $I$ is connected, we may take the limit on the right hand side in the category of $T$-schemes (as opposed to the category of $S$-schemes).
	We represent the limit as an equaliser between products
	\begin{center}
		\begin{tikzcd}
			\varprojlim_{i \in I} T_i = \mr{eq}\left( \prod_{i \in I} T_i \right. \arrow[r, shift left=.75ex] \arrow[r, shift right=.75ex, swap] &  \left. \prod_{i \in I} T_i \right),
		\end{tikzcd}
	\end{center}
	where the products are taken in the category of $T$-schemes.
	As all $T_i$ are separated over $T$, the inclusion of $\varprojlim_{i \in I} T_i \hookrightarrow \prod_{i \in I} T_i$ is a closed immersion. Moreover, as all $T_i \to T$ have property $\mathbf{P}$, so does their product. Hence, $\varprojlim_{i \in I} T_i \to T$ has property $\mathbf{P}$. 
\end{proof}

\begin{lemma}
	\label{lemSubSchm}
	Let $f \colon \Xc \to \Xc'$ be a schematic map of algebraic stacks and let $\Yc \se \Xc$ and $\Yc' \se \Xc'$ be two closed substacks such that $f|_{\Yc}$ factors through $\Yc'$. Then $f|_{\Yc} \colon \Yc \to \Yc'$ is schematic.
\end{lemma}
\begin{proof}
	Let $S$ be a scheme and let us fix a map $y' \colon S \to \Yc'$. As $f$ is schematic, the fibre product $T = S \times_{y, \Xc', f} \Yc$ is representable by a scheme. Then $T = S \times_{\Yc'} \Yc$. 
\end{proof}

\subsection{Integral models with deeper Bruhat-Tits level structure}
In this section, we define the integral model of the moduli space of shtukas with deep Bruhat-Tits level structure as the schematic image of the moduli space of shtukas for the Bruhat-Tits group scheme inside the limit of all the corresponding spaces for parahoric level and show that these integral models admit proper surjective and generically \'etale level maps as in the parahoric case.

In this section we work in the following setup that we record for future reference.
\begin{assumption}
	Let $x_0$ be a fixed closed point of $X$. Let $G$ be a reductive group over $F$ and $\Gc_\Omega \to X$ is a Bruhat-Tits group scheme for a bounded subset $\Omega = \mr{cl}(\Omega)$ of a single apartment in the Bruhat-Tits building for $G_{F_{x_0}}$ as in Construction \ref{consBTgrpschm}. Recall that $\Gc_\Omega = \varprojlim_{\ff \prec \Omega} \Gc_\ff$, where the limit runs over all facets $\ff$ contained in $\Omega$.
	
	Moreover, let $\Zc$ be a generically defined bound for $G$ as in Remark \ref{remBdGenFibre}. For example, $\Zc$ can be given by an $I$-tuple of conjugacy classes of geometric cocharacters for $G$. 
	For a smooth model $\Gc \to X$ of $G$ we denote be the representative of $\Zc$ defined over its reflex scheme by $Z_{\Gc}$.
	\label{assBT}
\end{assumption}
\begin{lemma}
	The map
	$ \rho_{\Omega, \ast} \colon \Gr^{I_\bullet}_{\Gc_\Omega, X} \to \varprojlim_{\ff \prec \Omega} \Gr^{I_\bullet}_{\Gc_\ff, X}$
	is an open immersion and an isomorphism over $(X \setminus \{x_0\})^I$.
	Moreover, it restricts to a locally closed immersion   
	$ \rho^\Zc_{\Omega, \ast} \colon Z_{\Gc_{\Omega}}\to \varprojlim_{\ff \prec \Omega} Z_{\Gc_{\ff}},$
	which again is an isomorphism over $(X \setminus \{x_0\})^I \times_{X^I} \widetilde X_\Zc^I$.
	\label{lemFunAffGrassBT}
\end{lemma}
\begin{proof}
	The map $\rho_{\Omega, \ast}$ is an open immersion by Theorem \ref{thmBunGImm}. From the description of the affine Grassmannian from Remark \ref{remBLDesc}, we see that the maps $\Gr^{I_\bullet}_{\Gc_\Omega, X} \to \Gr^{I_\bullet}_{\Gc_\ff, X}$ and $\Gr^{I_\bullet}_{\Gc_{\ff}, X} \to \Gr^{I_\bullet}_{\Gc_{\ff'}, X}$ are isomorphisms over $(X \setminus \{x_0\})^I$ for all $\ff \prec \Omega$ and $\ff' \prec \ff$. 
	
	The map $\rho^\Zc_{\Omega, \ast}$ factors as a closed immersion followed by an open immersion
	$$ Z_{\Gc_\Omega} \to \Gr_{\Gc_\Omega, X}^{I_\bullet} \times_{\varprojlim_{\ff \prec \Omega} \Gr_{\Gc_\ff, X}^{I_\bullet}} \varprojlim_{\ff \prec \Omega} Z_{\Gc_\ff} \to \varprojlim_{\ff \prec \Omega} Z_{\Gc_\ff}$$ 
	and is hence locally closed immersion. By construction of the bounds $\Zc$, it is an isomorphism over $(X \setminus \{x_0\})^I$.
\end{proof}

\begin{theorem} 
	\label{thmImmersion}
	In the situation of Assumption \ref{assBT}, the map 
	$ \rho_{\Omega, \ast} \colon  \Sht^{\Zc}_{\Gc_\Omega} \to \varprojlim_{\ff \prec \Omega} \Sht^{\Zc}_{\Gc_\ff}$
	is schematic and representable by a quasi-compact locally closed immersion.  
	Moreover, $\rho_{\Omega, \ast}$ is an open and closed immersion over $(X \setminus x_0)^I$.
	When $\Omega$ is a facet, $\rho_{\Omega, \ast}$ is an isomorphism.
\end{theorem}
\begin{proof}
	The assertion in the case that $\Omega$ is a facet is clear because the index set $\{\ff \prec \Omega\}$ has the final object $\Omega$.
	By Corollary \ref{corLvlMapParahoric} and Lemma \ref{lemLim}, the map is schematic, separated and of finite type.
	By Theorem \ref{thmBunGImm}, the corresponding map on the unbounded moduli stacks of shtukas is representable by an open immersion.
	Hence,  $ \rho_{\Omega, \ast}$ is a locally closed immersion as being bounded by $\Zc$ is a closed condition.
	
	Over $(X \setminus x_0)^I$, an element of $\Sht_{\Gc_\Omega}$ is bounded by $\Zc$ if and only if its image under $\rho_{\ff, \Omega, \ast}$ for one (or equivalently all) facet $\ff \prec \Omega$ is bounded by $\Zc$ by the second part of Lemma \ref{lemFunAffGrassBT}.
	Thus, $\rho_{\Omega, \ast}$ is an open immersion over $(X \setminus x_0)^I$.
	Moreover, the map $\rho_{\Omega, \ast}$ is finite away from $x_0$ by Lemma \ref{lemLim}, hence also a closed immersion. 
	\end{proof}
	
	\begin{definition}
		\label{defnIntMod}
		The integral model $\SSht^{\Zc}_{\Gc_\Omega} = \SSht^{I_\bullet, \Zc}_{\Gc_\Omega, X}$ of the moduli space of shtukas with $\Gc_\Omega$-level 
		is defined to be the schematic image in the sense of \cite{Emerton2021} of the map
		$$ \rho_{\Omega, \ast} \colon  \Sht^{\Zc}_{\Gc_\Omega} \to \varprojlim_{\ff \prec \Omega} \Sht^{\Zc}_{\Gc_\ff}.$$
	\end{definition}
	By Theorem \ref{thmImmersion}, we have $\SSht^{\Zc}_{\Gc_\ff} = \Sht^{\Zc}_{\Gc_\ff}$ in the parahoric case. Moreover, the inclusion $\Sht^{\Zc}_{\Gc_\Omega}  \to \SSht^{\Zc}_{\Gc_\Omega} $ is an isomorphism away from $x_0$ by Theorem \ref{thmImmersion} together with the fact that the schematic closure commutes with flat base change.
	\begin{remark}
		As the map $\rho_{\Omega, \ast}$ is schematic and the schematic image commutes with flat base change and is fpqc-local on the target (compare \cite[Remark 3.1.2]{Emerton2021}), the schematic image of $\rho_{\Omega, \ast}$ is given by descending the schematic image of the map of schemes $\rho_{\Omega, \ast, S} \colon S \times_{\varprojlim_{\ff \prec \Omega} \Sht^{\Zc}_{\Gc_\ff}} \Sht^{\Zc}_{\Gc_\Omega} \to S$ for any smooth atlas $S \to \varprojlim_{\ff \prec \Omega} \Sht^{\Zc}_{\Gc_\ff}$. 
	\end{remark} 
	
	By construction, we have level maps $\overline{\rho}_{\ff, \Omega}\colon \SSht^{\Zc}_{\Gc_\Omega} \to \Sht^{\Zc}_{\Gc_\ff}$ for all facets $\ff \prec \Omega$. 
	In particular, for $\Omega' \prec \Omega$ we obtain a map $\overline{\rho}_{\Omega', \Omega} \colon\SSht^{\Zc}_{\Gc_\Omega}\to \varprojlim_{\ff \prec \Omega'} \Sht^{\Zc}_{\Gc_\ff}$ that factors through $\SSht^{\Zc}_{\Gc_{\Omega'}}$ by construction.
	\begin{theorem}
		\label{thmLvlMapGeneral}
		Let $\Omega, \Omega'$ be two bounded connected subsets of an apartment in the Bruhat-Tits building of $G_{K_{x_0}}$ such that $\Omega' \prec \Omega$. Then, the level map
		$$\overline{\rho}_{\Omega', \Omega} \colon \SSht^{\Zc}_{\Gc_\Omega} \to \SSht^{\Zc}_{\Gc_{\Omega'}}$$
		is schematic, proper, surjective and finite \'etale away from $x_0$.
	\end{theorem}
	\begin{proof}
		As a first step, we show that $\overline{\rho}_{\Omega', \Omega}$ is schematic. By Lemmas \ref{lemLimProj} and \ref{lemLim}, the map 
		$ \varprojlim_{\ff \prec \Omega} \Sht^{\Zc}_{\Gc_\ff} \to \varprojlim_{\ff' \prec \Omega'} \Sht^{\Zc}_{\Gc_{\ff'}} $
		is schematic. The claim for  $\overline{\rho}_{\Omega', \Omega}$ then follows from Lemma \ref{lemSubSchm}.
		
		That the map is finite \'etale away from $x_0$ follows from the fact that the map $\Sht^{\Zc}_{\Gc_\Omega} \to \SSht^{\Zc}_{\Gc_\Omega}$ is an isomorphism away from $x_0$ by the observation above together with Corollary \ref{corLvlMapParahoric}.
		
		Moreover, the map $\SSht^{\Zc}_{\Gc_{\Omega'}} \to \varprojlim_{\ff \prec \Omega'} \Sht^{\Zc}_{\Gc_\ff}$ is a closed immersion by definition. Thus, by Lemma \ref{lemLim}, it suffices to consider the level maps 
		$$ \SSht^{\Zc}_{\Gc_\Omega} \to \Sht^{\Zc}_{\Gc_\ff} $$
		for facets $\ff \prec \Omega$ to show the properness. Similarly, by construction of  $\SSht^{\Zc}_{\Gc_\Omega}$, it suffices to show the claim for the projections
		$$ \varprojlim_{\ff \prec \Omega} \Sht^{\Zc}_{\Gc_\ff} \to \Sht^{\Zc}_{\Gc_\ff}.$$
		But for the projections the claim follows from Lemma \ref{lemLimProj}.
		The surjectivity follows as in the parahoric case in the proof of Theorem \ref{thmLvlMapGen}.
	\end{proof}
	
	\begin{remark}
		In a similar fashion we define an integral model of the moduli space of bounded $\Gc$-shtukas for \emph{arbitrary} Bruhat-Tits group schemes $\Gc \to X$.
		More precisely, let $x_1, \ldots, x_n$ be closed points of $X$ such that $\Gc|_{X \setminus \{x_1, \ldots, x_n\}}$ is parahoric. We set $U = X \setminus \{x_1, \ldots, x_n\}$. 
		For each $1 \leq m \leq n$ let $\Omega_m \se \Bc(G_{K_{x_m}}, K_{x_m})$ be a bounded subset contained in an apartment with $\Omega_m = \mr{cl}(\Omega_m)$ such that $\Gc|_{\Oc_{x_m}} = \Gc_{\Omega_m}$. In this case, we write $\Gc= \Gc_{(\Omega_m)_{1 \leq m \leq n}}$ and define the corresponding integral model
		$$ \SSht^{\Zc}_{\Gc} = \mr{image}\left( \Sht^{\Zc}_{\Gc}  \hookrightarrow \varprojlim_{(\ff_m)_{m=1, \ldots, n} \prec (\Omega_m)_{m = 1, \ldots, n}} \Sht^{\Zc}_{\Gc_{(\ff_m)_{m=1, \ldots, n}}}  \right) $$ 
		as above as the schematic image in the sense of \cite{Emerton2021} of the embedding of the stack of bounded $\Gc$-shtukas into the limit of the corresponding stacks with parahoric level. By the same arguments as above, the analogous assertions from Theorem \ref{thmImmersion} and Theorem \ref{thmLvlMapGeneral} hold in this setting as well.
		In particular, the natural map $\Sht^{\Zc}_{\Gc} \to \SSht^{\Zc}_{\Gc}$ is an open immersion and an isomorphism over $U$, and for any map of Bruhat-Tits group schemes $f \colon \Gc \to \Gc'$ that is an isomorphism generically we obtain a proper, surjective and generically \'etale level map 
		$$f_\ast \colon \SSht^{\Zc}_{\Gc}\to \SSht^{\Zc}_{\Gc'}.$$ 
	\end{remark}
	
	\begin{example}
		Let us consider the Drinfeld case, that is $G = \GL_r$ and $\underline{\mu} = ((1,0,\ldots,0),(0, \ldots,0,-1))$. 
		In this case, \cite{Bieker2023} defines \emph{Drinfeld $\Gamma_0(\pf^n)$-level structures} for shtukas adapting the notion of Drinfeld $\Gamma_0(p^n)$-level structures for elliptic curves of \cite{Katz1985}. Moreover, the moduli space of Drinfeld shtukas with Drinfeld $\Gamma_0(\pf^n)$-level structure identifies with $\SSht^{ (1,2), \leq \underline \mu}_{\GL_{r,\Omega}, X}$ by \cite[Theorem 6.7]{Bieker2023} for the standard simplex $\Omega$ of side length $n$ in the Bruhat-Tits building of $\GL_r$.
	\end{example}

	\subsection{The case that $\Omega$ is not contained in a single apartment}
	\label{sectOmegaGeneral}
	
	In principle, our techniques should also extend to the more general situation where the set $\Omega$ is not necessarily contained in a single apartment. 
	Generalising our results might yield further interesting applications.
	For example, when $\Omega$ is the ball of radius $r$ around a point $x$ in the building, its (connected, pointwise) stabiliser $G(\breve{F}_{x_0})^0_{\Omega}$ should be closely related (up to the centre) to the Moy-Prasad subgroup $G(\breve{F}_{x_0})_{x, r}$. 
	However, in this situation analogues of the results from Bruhat-Tits theory that go into the proofs in Section \ref{sectBT} (in particular of Lemma \ref{lemTorsDefo}) do not seem to be available.
	We present the extent to which we can generalise our main results.
	
	In the situation of Assumption \ref{assBT}, we drop the assumption that $\Omega$ is contained in a single apartment. In this section, let $\Omega \se \Bf(G_{F_{x_0}}, F_{x_0})$ be a non-empty bounded subset of the Bruhat-Tits building of $G_{F_{x_0}}$. 
	Recall that the smoothening of the closure of 
	$$ G \xrightarrow{\Delta} \prod_{\ff \prec \Omega} \Gc_{\ff},$$
	where $\Delta$ denotes the diagonal embedding in the generic fibre, is by construction a smooth affine $\Oc_{x_0}$-group scheme $\Gc_{\Omega}$ with generic fibre $G$ and such that $\Gc_{\Omega}(\breve{\Oc}_{x_0}) = \bigcap_{\ff \prec \Omega} G(\breve{F}_{x_0})_\ff^0$.
	As in the case where $\Omega$ is contained in an apartment, we denote by $\mr{cl}(\Omega)$ the maximal subset of $\Bf(G, F_{x_0})$ that is pointwise stabilised by $\Gc_{\Omega}(\breve{\Oc}_{x_0})$. Then $\mr{cl}(\Omega)$ is convex and closed.
	
	We consider the following conditions.
	\begin{condition}
		\begin{enumerate}
			\item The group scheme $\Gc_\Omega$ is a smooth affine group scheme with geometrically connected fibres. Moreover, the natural map $ \Gc_\Omega \xrightarrow{\cong} \varprojlim_{\ff \prec \Omega} \Gc_\ff$
			is an isomorphism.
			\item The map
			$ \rho_{\Omega, \ast} \colon \Bun_{\Gc_\Omega} \to \varprojlim_{\ff \prec \Omega} \Bun_{\Gc_\ff} $
			is schematic and representable by a quasi-compact immersion.
		\end{enumerate}
		\label{condGOmega}
	\end{condition}
	Recall that when $\Omega$ is contained in a single apartment, the group schemes $\Gc_\Omega$ satisfy the two conditions by Theorems \ref{thmBTGS} and \ref{thmBunGImm}. We do not know if the properties remain satisfied in general.
	
	\begin{proposition}
		Let $\Omega = \mr{cl}(\Omega)$ be a bounded subset of the Bruhat-Tits building (not necessarily contained in a single apartment).
		\begin{enumerate}
			\item   
			The map 
			$\rho_{\Omega, \ast} \colon \Sht^{\Zc}_{\Gc_\Omega} \to  \varprojlim_{\ff \prec \Omega} \Sht^{\Zc}_{\Gc_\ff}$ 
			is schematic, separated and of finite type, and a monomorphism.
			Moreover, it is an open and closed immersion over $(X\setminus\{x_0\})^I$.
			\label{propShtOmGenRep}
			\item 
			Assume that Condition \ref{condGOmega} holds for $\Gc_\Omega$. Then $\rho_{\Omega, \ast}$ is a quasi-compact immersion and an isomorphism over $(X \setminus \{x_0\})^I$.
			\label{propShtOmGenImm}
		\end{enumerate}
		\label{propShtOmGen}
	\end{proposition}
	\begin{proof}
		We proceed as in the proof of Theorem \ref{thmImmersion}.
		Assertion (\ref{propShtOmGenRep}) follows from Lemma \ref{lemPTor}, Corollary \ref{corLvlMapParahoric}, and Lemma \ref{lemLim}. Moreover, (\ref{propShtOmGenImm}) follows from Condition \ref{condGOmega}.
	\end{proof}

	\begin{remark}
		In the situation of Proposition \ref{propShtOmGen}, one could define an integral model $\SSht^{\Zc}_{\Gc_\Omega}$ again as the schematic image of the map $\rho_{\Omega, \ast}$. The argument in the proof of Theorem \ref{thmLvlMapGeneral} clearly generalises to this case. Hence, the models $\SSht^{\Zc}_{\Gc_\Omega}$ defined in this way again admit proper, surjective and generically \'etale level maps.
		Assuming Condition \ref{condGOmega},  $\SSht^{\Zc}_{\Gc_\Omega}$ is again a relative compactification of  $\Sht^{\Zc}_{\Gc_\Omega}$ as before by Proposition \ref{propShtOmGen} (\ref{propShtOmGenImm}). 
	\end{remark}


\subsection{Newton stratification}
We recall the Newton stratification on stacks of global shtukas and define a Newton stratification on our integral models with deep level. 
We show that the expected closure relations of Newton strata are satisfied in the hyperspecial case.

Let $K \cong k \dbr{t}$ be a local field in characteristic $p$ with ring of integers $\Oc \cong k \dsq{t}$ and finite residue field $k$. We denote by $\bar{K} = K^{\mr{sep}}$ a fixed separable closure and by $\breve{K} \cong k^{\mr{alg}} \dbr{t}$ the completion of the maximal unramified extension of $K$.
Let $G/K$ be a reductive group and 
let us fix $T \se G$ be a maximal torus defined over $K$.
Recall that $G_{\breve{K}} = G \times_K \breve{K}$ is quasi-split by a theorem of Steinberg and Borel-Springer, compare \cite[Theorem 1.9]{Steinberg1965} \cite[Section 8]{Borel1968}. We choose a Borel $B \se G_{\breve{K}}$ containing $T_{\breve{K}}$.  
We denote by $X_\ast(T)$ its group of geometric cocharacters and by $\pi_1(G)$ the algebraic fundamental group of $G$ given by the quotient of the cocharacter lattice by the coroot lattice.

We denote by $B(G)$ the set of $\sigma$-conjugacy classes in $G(\breve{K}) = LG(k^{\mr{alg}})$.  
By \cite{Kottwitz1985, Kottwitz1997, Rapoport1996}, the elements of $B(G)$ are classified by two invariants: the \emph{Kottwitz map} denoted by 
$$ \kappa \colon B(G) \to \pi_1(G)_{\Gal({\bar K}/K)} $$
and the \emph{Newton map} denoted by 
$$ \nu \colon B(G) \to (\Hom(\D_{\bar K}, G_{\bar K})/G({\bar K}))^{\Gal({\bar K}/K)},$$
where $\D$ denotes the pro-torus with character group $\Q$ and $G({\bar K})$ acts by conjugation. Note that we can identify 
$$ (\Hom(\D_{\bar K}, G_{\bar K})/G({\bar K}))^{\Gal({\bar K}/K)} = X_\ast(T)_{\Q}^{+, \Gal({\bar K}/K)} = X_\ast(T)_{\Q, \Gal({\bar K}/K)}^{+}$$ with the set of rational dominant (with respect to the choice of $B$) Galois-invariant cocharacters, and that $\kappa(b) = \nu(b)$ in $\pi_1(G)_{\Q, {\Gal({\bar K}/K)}}$.

The choice of Borel determines a set of simple positive roots and consequently defines the dominance order on $X_\ast(T)_{\Q}$ by $\mu_1 \leq \mu_2$ if $\mu_2 - \mu_1$ is a $\Q$-linear combination of positive simple roots with non-negative coefficients.
Via $\kappa$ and $\nu$ the dominance order induces a partial order on $B(G)$ by 
$b_1 \leq b_2$ if and only if $\kappa(b_1) = \kappa(b_2)$ and $\nu(b_1) \leq \nu(b_2)$.

Let $\Gc \to \Spec(\Oc)$ be a smooth affine group scheme such that $\Gc_K = G$.
Note that for an algebraically closed extension $\ell$ of $k$ the set of $\sigma$-conjugacy classes in $LG(\ell)$ does not depend on the choice of $\ell$ by \cite[Lemma 1.3]{Rapoport1996}.
It classifies quasi-isogeny classes of local $\Gc$-shtukas by associating to $(L^+\Gc,b^{-1})$ the class $[b] \in B(G)$. 
For a local $\Gc$-shtuka $\underline{\Ec}$ over $S = \Spec(R)$ and a point $s \in S$ we denote by $[\underline{\Ec}_s] \in B(G)$ the corresponding element. This does not depend on the choice of an algebraic closure of the residue field at $s$. 

Let us shift perspective back to the global setting again and consider a smooth affine group scheme $\Gc \to X$ with generic fibre $\Gc_F = G$ a reductive group. Let us moreover fix a tuple $\underline{y} = (y_i)_{i \in I}$ of pairwise distinct closed points of $X$. 
Let us fix a (global) bound $\Zc$ and points  $\underline{y}' = (y'_i)_{i \in I} \in X_\Zc^I$ lying over $\underline{y}$. 
We denote by $\Sht^{I_\bullet, \Zc}_{\Gc, X, \F_{\underline{y}'}} = \Sht^{\Zc}_{\Gc, \F_{\underline{y}'}} = \Sht^{\Zc}_{\Gc, \underline{y}'} \times_{X_\Zc^I} \Spec(\F_{\underline{y}'})$ the special fibre of the moduli space of shtukas at $\underline{y}'$.

\begin{definition}[{\cite[Definition 4.12]{Breutmann2019}}]
	Let $\ell$ be an algebraically closed extension of $\F_{\underline{y}'}$. The global-to-local functor induces maps
	$$
	\delta_{\Gc, y_i, \ell} \colon \Sht^{\Zc}_{\Gc, \F_{\underline{y}'}}(\ell)  \to B(G_{y_i}), \qquad 
	{\underline{\Ec}}  \mapsto [\widehat{\underline{\Ec}_{y_i}}]
	$$
	for all $i \in I$ and 
	$$ \delta_{\Gc, \underline{y}, \ell} = \prod_{i \in I} \delta_{\Gc, y_i, \ell}\colon   \Sht^{\Zc}_{\Gc, \F_{\underline{y}'}}(\ell) \to \prod_{i \in I} B(G_{y_i}).$$
	Let $\underline{b} = (b_i)_{i \in I} \in \prod_{i \in I} B(G_{y_i})$. The locus in  $\Sht^{\Zc}_{\Gc, \F_{\underline{y}'}}$ where $\delta_{\Gc, \underline{y}}$ maps to $\underline{b}$ is locally closed by \cite[Theorem 7.11]{Hartl2011}, compare also \cite{Rapoport1996}.
	The reduced substack on this locally closed subset is denoted by $\Sht^{\Zc, \underline{b}}_{\Gc, \F_{\underline{y}'}}$ and called the \emph{Newton stratum} associated to $\underline{b}$. 
\end{definition}

The Newton map is compatible with changing the group scheme in the following sense.
\begin{lemma}[compare {\cite[Section 5.2]{Breutmann2019}}]
	\label{lemNewtonStratComp}
	Let $G/F$ be a reductive group and let $\Gc$ and $\Gc'$ be two smooth affine models of $G$ over $X$. 
	Let $f \colon (\Gc, \Zc) \to (\Gc', \Zc')$ be a map of shtuka data such that $f \colon \Gc \to \Gc'$ is given by the identity on $G$ in the generic fibre. Assume that $\Zc'$ admits a representative over its reflex scheme.   
	The induced map 
	$ f_\ast \colon \Sht^{\Zc}_{\Gc} \times_{\widetilde X^I_\Zc} \widetilde X^I_{\Zc. \Zc'} \to \Sht^{\Zc'}_{\Gc'} \times_{\widetilde X^I_{\Zc'}} \widetilde X^I_{\Zc. \Zc'} $
	satisfies 
	$ \delta_{\Gc', \underline{y}} \circ f_\ast = \delta_{\Gc, \underline{y}}.$
\end{lemma}
\begin{proof}
	The proof of {\cite[Section 5.2]{Breutmann2019}} carries over to this situation.
\end{proof}

Using the compatibility of the Newton map with the change in level structures, we define a Newton stratification on our integral models for deeper level structures.
Let us now consider the situation of Assumption \ref{assBT}. In particular, let $\Omega = \mr{cl}(\Omega)$ be a bounded subset of an apartment of the Bruhat-Tits building of $G_{F_{x_0}}$ for a fixed closed point $x_0$ of $X$. Let $\Gc_\Omega$ be the corresponding Bruhat-Tits group scheme.
Let $\Zc$ be a generically defined bound for $G$. Let moreover $\underline{y}' = (y'_i)$ be a tuple of closed points of $X_{\Zc}$ lying over $\underline{y}$. 
In order to define a Newton stratification on $\SSht^{\Zc}_{\Gc_\Omega, \F_{\underline{y}'}}$, we note that by construction and by the previous lemma the map
$$ \delta_{\Gc_\ff, \underline{y}} \circ \rho_{\ff, \Omega} \colon \SSht^{\Zc }_{\Gc_\Omega, \F_{\underline{y}'}} \to \Sht^{\Zc}_{\Gc_\ff, \F_{\underline{y}'}} \to \prod_{i \in I} B(G_{y_i}) $$
does not depend on the choice of the facet $\ff \prec \Omega$. Hence, we obtain a well-defined map 
$$ \bar \delta_{\Gc_\Omega, \underline{y}} \colon \SSht^{\Zc}_{\Gc_\Omega, \F_{\underline{y}'}} \to \prod_{i \in I} B(G_{y_i}).$$
Let $\underline{b} = (b_i)_{i \in I} \in \prod_{i \in I} B(G_{y_i})$. The locus in  $\SSht^{\Zc}_{\Gc_\Omega, \F_{\underline{y}'}}$ where $\bar \delta_{\Gc_\Omega, \underline{y}}$ maps to $\underline{b}$ is again locally closed by the result in the parahoric case together with Lemma \ref{lemNewtonStratComp}.
\begin{definition}
	\label{defnNewtonDeep}
	Let $\underline{b} = (b_i)_{i \in I} \in \prod_{i \in I} B(G_{y_i})$. The \emph{Newton stratum} in $\SSht^{\Zc}_{\Gc_\Omega, \F_{\underline{y}'}} $ associated to $\underline{b}$ is the reduced locally closed substack on the set of points where $\bar \delta_{\Gc, \underline{y}}$ maps to $\underline{b}$. 
	It is denoted by $\SSht^{\Zc, \underline{b}}_{\Gc_\Omega, \F_{\underline{y}'}}$. 
\end{definition}

We have the obvious analogue of Lemma \ref{lemNewtonStratComp} for deep level, in other words, the Newton stratification for deep levels is still compatible with the level maps.
\begin{corollary}
	\label{corNewtonLevel}
	Let $\Omega' \prec \Omega$ be two connected bounded subsets of the Bruhat-Tits building. Then 
	$$ \bar{\delta}_{\Gc_{\Omega'}, \underline{y}} \circ \bar{\rho}_{\Omega', \Omega} = \bar{\delta}_{\Gc_{\Omega}, \underline{y}}.$$
	In particular, $\SSht^{\Zc, \underline{b}'}_{\Gc_\Omega, \F_{\underline{y}'}} \cap \overline{\SSht^{\Zc, \underline{b}}_{\Gc_\Omega, \F_{\underline{y}'}}} \neq \emptyset$ only if $\underline{b}' \leq \underline{b}$.
\end{corollary}
\begin{proof}
	This follows from the construction and Lemma \ref{lemNewtonStratComp}. The second statement then follows directly from the parahoric case in \cite[Proposition 4.11, Section 5]{Breutmann2019}, compare also \cite[Theorem 7.11]{Hartl2011}.
\end{proof}

We conclude by showing the strong stratification property of the Newton stratification in the hyperspecial case. 

\begin{theorem}
	\label{thmNewtonHyp}
	Let $\Gc \to X$ be a parahoric group scheme that is hyperspecial at $y_i$ for all $i \in I$. Let $\underline{\mu} = (\mu_i)_{i \in I}$ be an $I$-tuple of conjugacy classes of geometric cocharacters of $G$. Then the Newton stratification at $\underline{y}'$ satisfies the strong stratification property in the sense that
	$$\overline{\Sht^{\leq \underline{\mu}, \underline{b}}_{\Gc, \F_{\underline{y}'}}}  = \bigcup_{\underline{b}' \leq \underline{b}} \Sht^{\leq \underline{\mu},  \underline{b'}}_{\Gc, \F_{\underline{y}'}}$$
	for all $ \underline{b} \in \prod_{i \in I} B(G_{y_i})$.
\end{theorem}

\begin{proof}
	Let $\underline{b}, \underline{b}' \in \prod_{i \in I} B(G_{y_i})$ with $\underline{b}' \leq \underline{b}$. It suffices to show that every closed point $ \bar{s} = \underline{\Ec} \in \Sht^{\leq \underline{\mu},  \underline{b'}}_{\Gc, \F_{\underline{y}'}} (\F^{\mr{alg}}_{\underline{y}'})$ lies in the closure of $ \Sht^{\leq \underline{\mu},  \underline{b}}_{\Gc, \F_{\underline{y}'}}$.	
	Let $R$ be the $\Oc_{\underline{y}'}$-algebra pro-representing the deformation functor of $\bar{s}$. 
	Then $\bar{s}$ lies in the closure of $\overline{\Sht^{\leq \underline{\mu}, \underline{b}}_{\Gc, \F_{\underline{y}'}}}$ if and only if the same is true in the Newton stratification on $\Spec R$.
	By the bounded Serre-Tate Theorem (Proposition \ref{propSerreTate}) the universal deformation ring factors as $\Spec R = \prod_{i \in I} \Spec R_i$, where $R_i$ is the universal deformation ring of the corresponding bounded local shtuka at $y_i$. Under this isomorphism we have $\Spec(R)_{\underline{b} } = \prod_{i \in I} \Spec(R_i)_{b_i}$, where we denote by $\Spec(R_i)_{{b}_{i}}$ the corresponding Newton strata in $\Spec R_i$ for $i \in I$.
	On $\Spec R_i$ the closure properties hold by \cite[Theorem 2, Lemma 21 (2)]{Viehmann2011}, and thus they hold on $\Spec R$. This proves the assertion.
\end{proof}

	\printbibliography	
\end{document}